\documentclass{amsart}

\usepackage{amsmath}
\usepackage{amssymb}
\usepackage{amscd}
\usepackage{graphics}
\usepackage{hyperref}
 \usepackage{xypic}
\usepackage{tikz}
\usetikzlibrary{positioning}

\newtheorem{theorem}{Theorem}[section]
\newtheorem{thm}[theorem]{Theorem}
\newtheorem{cor}[theorem]{Corollary}
\newtheorem{lem}[theorem]{Lemma}
\newtheorem{prop}[theorem]{Proposition}
\newtheorem{claim}[theorem]{Claim}

\theoremstyle{definition}
\newtheorem{defn}[theorem]{Definition}

\newtheorem{ques}[theorem]{Question}

\newtheorem{sit}[theorem]{Situation}
\newtheorem{rem}[theorem]{Remark}
\newtheorem{constr}[theorem]{Construction}
\newtheorem{conv}[theorem]{Convention}
\newtheorem{notat}[theorem]{Notation}
\newtheorem{ex}[theorem]{Example}

\newtheorem{conj}[theorem]{Conjecture}

\theoremstyle{remark}

\newcommand{\mbb}{\mathbb}
\newcommand{\FF}{\mbb{F}}
\newcommand{\QQ}{\mbb{Q}}

\newcommand{\ZZ}{\mbb{Z}}
\newcommand{\CC}{\mbb{C}}
\newcommand{\RR}{\mbb{R}}
\newcommand{\AAA}{\mbb{A}}
\newcommand{\PP}{\mbb{P}}

\newcommand{\mc}{\mathcal}

\newcommand{\mcX}{\mc{X}}

\newcommand{\OO}{\mc{O}}

\newcommand{\SP}{\text{Spec }}

\newsavebox{\sembox}
\newlength{\semwidth}
\newlength{\boxwidth}

\newsavebox{\semrbox}
\newlength{\semrwidth}
\newlength{\boxrwidth}

\newcommand{\Semr}[1]{%
\sbox{\semrbox}{\ensuremath{#1}}%
\settowidth{\semrwidth}{\usebox{\semrbox}}%
\sbox{\semrbox}{\ensuremath{\left(\usebox{\semrbox}\right)}}%
\settowidth{\boxrwidth}{\usebox{\semrbox}}%
\addtolength{\boxrwidth}{-\semrwidth}%
\left(\hspace{-0.3\boxrwidth}%
\usebox{\semrbox}%
\hspace{-0.3\boxrwidth}\right)%
}

\newcommand{\cyl}{\text{cyl}}
\newcommand{\eq}{\text{eq}}
\newcommand{\ch}{\text{Chow}}

\newcommand{\st}{\text{st}}

\newcommand{\et}{\text{\'et}}
\newcommand{\ci}{\text{ci}}
\newcommand{\ecomb}{\text{EComb}}
\newcommand{\free}{\text{free}}

\newcommand{\sm}{\text{sm}}

\title[Local-global principle and integral Tate conjecture ]{Local-global principle and integral Tate conjecture for certain varieties}

\author[Tian]{Zhiyu Tian}
\address{
Beijing International Center for Mathematical Research\\
Peking University\\
100871, Beijing, China}
\email{zhiyutian@bicmr.pku.edu.cn}

\date{\today}

\begin{document}


\begin{abstract}
We give a geometric criterion to check the validity of the integral Tate conjecture for one-cycles on a smooth projective variety that is separably rationally connected in codimension one, and to check that the Brauer-Manin obstruction is the only obstruction to the local-global principle for zero-cycles on a separably rationally connected variety defined over a global function field.
 
 We prove that the Brauer-Manin obstruction is the only obstruction to the local-global principle for zero-cycles on all geometrically rational surfaces defined over a global function field, and to the Hasse principle for rational points on del Pezzo surfaces of degree four defined over a global function field of odd characteristic.  

 Along the way, we also prove some results about the space of one-cycles on a smooth projective variety that is separably rationally connected in codimension one, which leads to the equality of the coniveau filtration and the strong coniveau filtration on degree $3$ homology of such varieties.
\end{abstract}


\maketitle

\tableofcontents

\section{Introduction}
In this paper, we adopt the following conventions. 
A variety over a field $k$ is a finite type, integral, separated $k$-scheme. 
Given a variety $V$ defined over a field $k$, 
we use $\overline{V}$ to denote the base change of $V$ to a pre-fixed algebraic closure of $k$.
An algebraic set over a field $k$ is a finite type, geometrically reduced, separated $k$-scheme. 
By a point in a $k$-scheme, we mean a closed point, unless otherwise specified.
A pointed $k$-scheme is a $k$-scheme $X$ together with a finite number of $k$-rational points $(X, x_1, \ldots, x_n)$. A morphism between pointed schemes $f:(X, x_1, \ldots, x_n) \to (Y, y_1, \ldots, y_n)$ is a morphism of schemes $f: X \to Y$ such that $f(x_i)=y_i, i=1, \ldots, n$.

For any abelian group $A$, any integer $m$, and any prime number $\ell$, we use $A[m]$ to denote the group of $m$-torsions in $A$,  $A/m$ to denote the quotient $A/mA$, and $A \hat{\otimes} \ZZ_\ell$ to denote the inverse limit $\lim \limits_{\xleftarrow[n]{}} A/\ell^n$.
 
\subsection{The Local-global principle for zero-cycles}

Given a smooth projective variety defined over a global field, a natural and important problem is to find criteria for the existence of rational points and a description of the set of all rational points. 
The Hasse principle and the weak approximation problem, or the local-global principle, gives a characterization of this set in terms of the adelic points.
 There are various obstructions for the local-global principle to hold, notably the so called Brauer-Manin obstruction. A conjecture due to Colliot-Th\'el\`ene states that for rationally connected varieties defined over a global field, this is the only obstruction.
The study of zero-cycles, as natural generalizations of rational points, has also drawn a lot of attentions in recent years.
Motivated by the case of rational points,
Colliot-Th\'el\`ene has formulated the following conjectures on the local-global principle for zero-cycles. 
\begin{conj}\cite[Conjecture 2.2]{CTLocalGlobalChow}\label{conj:CT1}
Let $X$ be a smooth projective geometrically integral variety defined over the function field $\FF_q(B)$ of a smooth curve $B$ defined over a finite field $\FF_q$. Fix a prime number $\ell$ different from the characteristic.
For every place $\nu$ of $\FF_q(B)$, let $z_\nu \in \text{CH}_0(X_\nu)$. 
Suppose that for all element $A\in \text{Br}(X)\{\ell\}$, we have $\sum_{\nu} \text{Inv}(A(z_\nu))=0$. Then for all $n>0$, there is a cycle $z_n \in \text{CH}_0(X)$ such that for all places $\nu$ we have that 
\[
cl(z_n) =cl(z_{\nu}) \in H^{2d}_{\text{\'et}}(X_{\nu}, \mu_{\ell^n}^{\otimes d}).
\]
\end{conj}
Here $\text{Inv}(A(z_\nu))$ means the value of $(A, z_\nu)$ under the pairing
\[
\text{Br}(X_\nu)\{\ell\} \times \text{CH}_0(X_\nu) \to \QQ/\ZZ,
\]
and $cl:\text{CH}_0(X_v)/\ell^n \to H^{2d}_\et(X_v, \mu_{\ell^n}^{\otimes d})$ is the cycle class map.

A particular case of the above conjecture is the following.
\begin{conj}\label{conj:CT2}
Let $X$ be a smooth projective geometrically integral variety defined over the function field $\FF_q(B)$ of a smooth curve $B$ defined over a finite field $\FF_q$. Fix a prime number $\ell$ different from the characteristic.
Suppose that for every place $\nu$ of $\FF_q(B)$, there is a zero-cycle $z_\nu \in \text{CH}_0(X_\nu)$ of degree prime to $\ell$. 
Suppose that for all elements $A\in \text{Br}(X)\{\ell\}$, we have 
$\sum_{\nu} \text{Inv}(A(z_\nu))=0$.
 Then there is a cycle $z \in \text{CH}_0(X)$ of degree prime to $\ell$.
\end{conj}

Conjectures \ref{conj:CT1} and \ref{conj:CT2} are consequences of the following Conjecture \ref{conj:E} via considering the commutative diagram of various cycle class maps to the \'etale cohomology.
\begin{conj}[{\cite[Conjecture 4.2]{CT_Problem_list}}]\label{conj:E}
Let $X$ be a smooth projective geometrically integral variety defined over a global field $K$. Let $l$ be a prime number different from the characteristic of $K$. 
When $K$ is a number field and $\nu$ is a real place, one defines $\text{CH}_0(X_\nu)^*=\text{CH}_0(X_\RR)/\text{CH}_0(X_\CC)$, that is the Chow group over $\RR$ modulo the image of Chow groups over $\CC$ by the norm map. For other places, $\text{CH}_0(X_\nu)^*$ is defined to be the usual Chow group.
Then the Brauer-Manin obstruction is the only obstruction for the local-global principle for zero-cycles on $X$ to hold. Namely, there is an exact sequence:
\[
\text{CH}_0(X) \hat{\otimes} \ZZ_\ell \to \Pi_{\nu \in \Omega(K)} \text{CH}_0(X_\nu)^* \hat{\otimes} \ZZ_\ell \to \text{Hom}(\text{Br}(X)\{\ell\}, \QQ/\ZZ).
\]
\end{conj}

If the cycle class map $\text{CH}_0(X_\nu) \hat{\otimes} \ZZ_{\ell} \to H^{2d}_{\text{\'et}}(X_\nu, \ZZ_{\ell}(d))$ is injective for every place $\nu$, Conjecture \ref{conj:CT1} and this stronger conjecture \ref{conj:E} are equivalent.
For a discussion of the injectivity for separably rationally connected varieties, see Section \ref{sec:questions}.

Unlike the case of rational points, these conjectures apply to all smooth projective varieties.

One of the main theorems of this article (to be proved in Section \ref{proof}) is the following.
\begin{thm}\label{thm:zerocycle}
Let $X$ be a smooth projective geometrically rational surface defined over the function field $\FF_q(B)$ of a smooth projective curve $B$ defined over a finite field $\FF_q$. Then Conjectures \ref{conj:E}, and hence \ref{conj:CT1} and \ref{conj:CT2} hold true for $X$ .
\end{thm}

By a happy coincidence, we deduce a corollary for rational points (to be proved in Section \ref{proof}).
\begin{thm}\label{thm:delpezzo4}
Let $X$ be a del Pezzo surface of degree $4$ defined over a global function field of odd characteristic. Then the Brauer-Manin obstruction is the only obstruction for the Hasse principle for rational points on $X$ to hold.
\end{thm}

\subsection{The Integral Tate conjecture for one-cycles}

There is a closely related conjecture.

\begin{conj}[Integral Tate conjecture for one-cycles]
    Let $V$ be a smooth projective geometrically irreducible variety of dimension $d$ defined over a finite field $\FF$. Fix a prime number $\ell$ different from the characteristic. The cycle class map:
\begin{equation}\label{int_Tate}
\text{CH}_1(V) \otimes \ZZ_\ell \to H^{2d-2}_\et(V, \ZZ_\ell(d-1))
\end{equation}
is surjective.
\end{conj}

Such a conjecture has been raised many times. See for example \cite{CTSzamuelyIntegralTate,CTKahnCycleCodim2, CT_Scavia_ITC}.

Saito \cite{SaitoMotivicCoh_ArithmeticScheme} and Colliot-Th\'el\`ene \cite{CTLocalGlobalChow} proved that there is a close relation between the surjectivity of this cycle class map  and Colliot-Th\'el\`ene's conjectures. Namely, Conjecture \ref{conj:CT1} for $X/\FF_q(B)$ follows from the surjectivity of (\ref{int_Tate}) for a smooth projective model of $X$. See Section \ref{sec:ITateHasse} Theorem \ref{thm:TateImpiesCT} for a more precise statement.

We prove Theorem \ref{thm:zerocycle} as a consequence of the following theorem, whose proof is contained in Section \ref{proof}.
\begin{thm}\label{thm:integralTateSurface}
Let $\pi: \mcX \to B$ be a projective flat family of surfaces over a smooth projective curve $B$ defined over a finite field $\FF_q$. Fix a prime number $\ell$ different from the characteristic. Assume that $\mcX$ is smooth and that the geometric generic fiber is a smooth rational surface. Then the integral Tate conjecture for one-cycles holds. More concretely, the cycle class map
 \[
\text{CH}_1(\mcX)\otimes \ZZ_\ell \to H^{4}_\text{\'et}(\mcX, \ZZ_\ell(2))
\] 
 is surjective.
\end{thm}

This theorem is a consequence of our study of one-cycles on a particular class of varieties as defined below.

\begin{defn}
    Let $X$ be a variety defined over a field $k$. Let $K/k$ be a field extension, and let $f: \PP^1_K \to X_K$  be a $K$-morphism  whose image lies in the smooth locus of $X_K$. The pull-back $f^*T_{X_K}$ decomposes as direct sum of line bundles $\oplus_{i=1}^{\dim X} \OO(a_i)$.
    Such a morphism $f$ is called a \emph{very free curve} (resp. \emph{free curve}) if $a_i>0$ (resp. $a_i\geq 0$) for all $i$.
    
    We say that $X$ is \emph{separably rationally connected}, or \emph{SRC}, if $X$ has a very free curve.
    
    We say that $X$ is \emph{separably rationally connected in codimension one} or \emph{SRC in codimension one} if $X$ has a free curve with $a_i> 0$ for all but one $a_i$.
\end{defn}
\begin{rem}
    The term SRC was introduced by Koll\'ar-Miyaoka-Mori \cite{KMM92RC}. The term SRC in codimension one was introduced in \cite{Kollar_Tian}. 
\end{rem}

Recall the following definition of algebraic equivalence over a field $k$ (not necessarily algebraically closed) as in \cite[Definition 10.3]{Fulton98}.

\begin{defn}\label{def:alg_equiv}
    Let $X$ be a $k$-scheme. A $r$-cycle $Z$ ($r\geq 0$) is \emph{algebraically equivalent to $0$} if there is a non-singular $k$-variety $T$, two $k$-rational points $t_1, t_2 \in T$, and a cycle $\mathcal{Z} \in \text{CH}_{\dim T+r}(X\times T)$ such that $Z=\mathcal{Z}_{t_1}-\mathcal{Z}_{t_2}$ in $\text{CH}_r(X)$, where $\mathcal{Z}_{t_1}$ (resp. $\mathcal{Z}_{t_2}$) is the Gysin pull-back of $\mathcal{Z}$ via the inclusion $t_1 \to T$ (resp. $t_2 \to T$).
    
    Algebraic  equivalence is the equivalence  relation generated by cycles that are algebraically equivalent to $0$.
    
    We denote by $A_r(X)$ the group of $r$-cycles in $X$ modulo algebraic equivalence, and $\text{CH}_r(X)_{\text{alg}}$ the subgroup of $\text{CH}_r(X)$ generated by $r$-cycles algebraically equivalent to $0$.
\end{defn}

In general, one can deduce the following geometric criterion for the validity of the integral Tate conjecture, whose proof is contained in Section \ref{proof}.
\begin{thm}\label{thm:integralTate}
Let $X$ be a smooth projective geometrically integral variety defined over a finite field $\FF_q$.
Fix a prime number $\ell$ different from the characteristic.
Assume that $X$ has dimension $d (d \geq 2)$, and is separably rationally connected in codimension one. Consider the following hypotheses:
\begin{itemize}
\item [(A)] The cycle class map $A_1(\overline{X})\otimes \ZZ_\ell \to H^{2d-2}_\text{\'et}(\overline{X}, \ZZ_\ell(d-1))$ is surjective.
\item [(B)] The cycle class map $A_1(\overline{X})\otimes \ZZ_\ell \to H^{2d-2}_\text{\'et}(\overline{X}, \ZZ_\ell(d-1))$ is injective.
\item [(C)] The cycle class map from higher Chow groups (Definition \ref{def:higherchow})
\[
\lim \limits_{\xleftarrow[n]{}}\text{CH}_1(\overline{X}, 1, \ZZ/\ell^n \ZZ) \to H^{2d-3}_\text{\'et}(\overline{X}, \ZZ_\ell(d-1))
\]
is surjective.
\item [(D)] The subgroup $N^{d-2}H^{2d-3}_\text{\'et}(\overline{X}, \ZZ_\ell)$ (Definition \ref{def:filtration_1}) equals the whole cohomology group $H^{2d-3}_\text{\'et}(\overline{X}, \ZZ_\ell)$.
\end{itemize}
The following holds.
\begin{enumerate}
    \item\label{1.10.1} If $\overline{X}$ satisfies hypotheses (A) and (B),
then the cycle class map
 \[
\text{CH}_1(X)\otimes \ZZ_l \to H^{2d-2}_\text{\'et}(X, \ZZ_l(d-1))\to  H^{2d-2}_\text{\'et}(\overline{X}, \ZZ_l(d-1))^{Gal(\bar{\FF}_q/\FF_q)}
 \] 
 is surjective.
 
 \item\label{1.10.2} If $\overline{X}$ satisfies hypothesis (C) or (D),
then the cycle class map
 \[
\text{CH}_1(X)_{\text{alg}}\otimes \ZZ_l \to H^1(\FF_q, H^{2d-3}_\text{\'et}(\overline{X}, \ZZ_l(d-1)))
 \] 
 is surjective.
\end{enumerate}

\end{thm}

\begin{rem}
In dimension $3$, the hypotheses (B), (C), and (D) in Theorem \ref{thm:integralTate} are satisfied for a variety that is separably rationally connected in codimension one. See Lemma \ref{3fold} and the proof of Theorem \ref{thm:integralTateSurface} in Section \ref{proof} for details.

Therefore, if a threefold $X$ defined over ${\FF}_q$ that is separably rationally connected in codimension one satisfies hypothesis (A) in the theorem, i.e. the cycle class map $\text{CH}_1(\overline{X})\otimes \ZZ_\ell \to H^4_\et(\overline{X}, \ZZ_\ell(2))$ is surjective, then the cycle class map $$\text{CH}_1({X})\otimes \ZZ_\ell \to H^4_\et({X}, \ZZ_\ell(2))$$ is surjective.
\end{rem}

\begin{rem}
    Hypotheses (A)-(D) only depend on the stable birational class of the generic fiber over $\bar{\FF}_q(B)$ (Lemma \ref{lem:birational}).
    See Section \ref{sec:questions} for further discussions about the validity of these hypotheses.
\end{rem}

We have a corollary of Theorem \ref{thm:integralTateSurface}, which confirms a conjecture of Colliot-Th\'el\`ene and Kahn (\cite[Conjecture 5.8]{CTKahnCycleCodim2}) up to $p$-torsion. The proof is contained in Section \ref{proof}.

\begin{cor}\label{cor:CTK}
Let $X$ be a smooth projective threefold defined over a finite field $\FF$.
Fix a prime number $\ell$ different from the characteristic.
Assume that $X$ admits a fibration structure over a smooth projective curve with smooth projective geometrically rational generic fiber.
Then we have 
\[
H^3_{\text{nr}}(X, \QQ_\ell/\ZZ_\ell(2))=0,
\]
\[
\text{CH}_1(X) \otimes \ZZ_\ell \xrightarrow{\cong} H^4(X, \ZZ_\ell(2)),
\]
and a short exact sequence
\[
0 \to H^1(\FF, H^3(\overline{X}, \ZZ_\ell(2))\{\ell\}) \to \text{CH}_1(X) \otimes \ZZ_\ell \to \text{CH}_1(\overline{X})^G \otimes \ZZ_\ell \to 0.
\]
\end{cor}

\subsection{Method of proof}
We use a geometric approach based on the technique developed in a joint work of the author with J\'anos Koll\'ar \cite{Kollar_Tian}.
This technique, very roughly speaking, allows one to turn algebraic equivalence into deformations of stable maps, which is the key geometric input behind all the results proved in this paper.
 
We sketch the proof of Theorems \ref{thm:integralTateSurface} and \ref{thm:integralTate}, which are contained in Section \ref{proof}.

The Hochschild-Serre spectral sequence gives an exact sequence:
\[
0 \to H^1(\FF_q, H^{2d-2r-1}_\et(\overline{X}, \ZZ_\ell(d-r))) \to H^{2d-2r}_\et(X, \ZZ_\ell(d-r)) \to H^{2d-2r}_\et(\overline{X}, \ZZ_\ell(d-r))^G \to 0.
\]
Thus the integral Tate conjecture consists of a geometric part, i.e. surjectivity of 
\[
\text{CH}_r(X) \otimes \ZZ_\ell \to H^{2d-2r}_\et(\overline{X}, \ZZ_\ell(d-r))^G,
\]
and an arithmetic part, i.e. surjectivity of
\[
\text{CH}_r(X)_\text{hom}\otimes \ZZ_\ell \to H^1(\FF_q, H^{2d-2r-1}_\et(\overline{X}, \ZZ_\ell(d-r))),
\]
where $\text{CH}_r(X)_\text{hom}\otimes \ZZ_\ell$ is the ``geometrically homologically trivial" part, i.e. the kernel of 
\[
\text{CH}_r(X) \otimes \ZZ_\ell \to H^{2d-2r}_\et(\overline{X}, \ZZ_\ell(d-r))^G.
\]

For the geometric part, we need the following theorem, obtained in the joint work with J. Koll\'ar \cite{Kollar_Tian}.

\begin{thm}\cite[Theorem 7]{Kollar_Tian}\label{thm:G_inv_cycle}
Let $X_k$ be a smooth projective variety defined over a perfect field $k$. Assume that every geometrically irreducible $k$-variety has a $0$-cycle of degree $1$ (e.g. $k$ is a finite field or a PAC field). Assume that $X$ is separably rationally connected in codimension one.
We have an isomorphism 
\[
A_1(X_k) \xrightarrow{\cong} A_1(X_{\bar{k}})^G,
\]
where $G$ is the absolute Galois group of $k$.  
\end{thm}

We apply this theorem to prove the geometric part (\ref{1.10.1}) of Theorem \ref{thm:integralTate}.

For the arithmetic part, we recall the following definitions,  recently revisited by Benoist-Ottem, Voisin, Scavia-Suzuki, and Benoist in \cite{BO_Coniveau, VoisinConiveau, ScaviaSuzuki_examples, Benoist2022Steenrod}.

\begin{defn}\label{def:filtration_1}
Let $X$ be a smooth projective variety of dimension $d$ defined over an algebraically closed field. 
Given an abelian group $A$ that is one of $\ZZ/m\ZZ, \ZZ_\ell, \ZZ, \QQ$, or $\QQ_\ell$, where $m$ is a positive integer, $\ell$ a prime number, both of which are relative prime to the field characteristic, we simply write $H^k(X, A)$ as either the \'etale cohomology with coefficient $A$ or the singular cohomology with coefficient $A$ (if $X$ is a complex variety), and $H_i(*, A)$ as the \'etale Borel-Moore homology or singular Borel-Moore homology. We define the following filtrations on the  cohomology group $H^k(X, A)$.
\begin{enumerate}
\item The coniveau filtration $\{N^*H^k(X, A)\}$ with
\[
N^cH^k(X, A):=\sum_{f: Y \to X}f_*(H_{2d-k}(Y, A)) \subset H_{2d-k}(X, A)\cong H^k(X, A), 
\]
where the sum is taken over all morphisms from projective algebraic sets $f: Y \to X, \dim Y \leq d-c$;
\item  The strong coniveau filtration $\{\tilde{N}^*H^k(X, A)\}$ with
\[
\tilde{N}^cH^k(X, A):=\sum_{f: Y \to X}f_*(H_{2d-k}(Y, A)) \subset H_{2d-k}(X, A)\cong H^k(X, A),
\]
where the sum is taken over all morphisms from \textit{smooth} projective varieties $f: Y \to X, \dim Y \leq d-c$.
\end{enumerate}
\end{defn}

\begin{rem}
    If $X$ is smooth but not necessarily projective, we may similarly define the coniveau filtration by taking the sum over all algebraic sets $Y$ with a proper morphism $f: Y \to X$.
\end{rem}

In a recent preprint \cite{ScaviaSuzuki2023coniveau}, Scavia and Suzuki systematically investigated the question of the surjectivity in the arithmetic part and related this question to the comparison of the strong coniveau filtration and the coniveau filtration.

\begin{prop}\cite[Proposition 5.6]{ScaviaSuzuki2023coniveau}\label{thm:SS3}
Let $X$ be a smooth projective geometrically connected variety over a finite field $\FF$ and $\ell$ be a prime number invertible in $\FF$. Then the algebraic Walker Abel-Jacobi map restricted to cycles algebraically equivalent to $0$
\[
\text{CH}^r(X)_\text{alg}\otimes \ZZ_\ell  \to H^1(\FF , N^{r-1}H^{2r-1}(\overline{X}, \ZZ_\ell(r)))
\]
has the same image as
\[
H^1(\FF , \tilde{N}^{r-1}H^{2r-1}(\overline{X}, \ZZ_\ell(r))) \to H^1(\FF , {N}^{r-1}H^{2r-1}(\overline{X}, \ZZ_\ell(r))).
\]
\end{prop}
Here the algebraic Walker Abel-Jacobi map (restricted to cycles algebraically equivalent to $0$)
\[
\text{CH}^r(X)_\text{alg}\otimes \ZZ_\ell  \to H^1(\FF , N^{r-1}H^{2r-1}(\overline{X}, \ZZ_\ell(r)))
\]
is a lifting of the cycle class map 
\[
\text{CH}^r(X)_\text{alg}\otimes \ZZ_\ell  \to H^1(\FF , H^{2r-1}(\overline{X}, \ZZ_\ell(r))),
\]
constructed in \cite[Formula (4.2), Section 4.2]{ScaviaSuzuki2023coniveau}.
\begin{rem}
In the paper of Scavia-Suzuki \cite{ScaviaSuzuki2023coniveau}, the authors use $\text{CH}^r(X)_\text{F-alg}$ to denote the group $\text{CH}^r(X)_\text{alg}$ in our notations, and use $\text{CH}^r(X)_\text{alg}$ to denote a different group that contains $\text{CH}^r(X)_\text{F-alg}$ (see the first two paragraphs in \cite[Section 4.1]{ScaviaSuzuki2023coniveau}). 
\end{rem}

Therefore part (\ref{1.10.2}) of Theorem \ref{thm:integralTate} is implied by the following theorem, which generalizes a result of Voisin in \cite{VoisinConiveau} (see Section \ref{sec:coniveau} for a detailed account).
\begin{thm}[=Theorem \ref{thm:image_sheaf}]
Let $X$ be a $d$-dimensional smooth projective variety defined over an algebraically closed field, which is separably rationally connected in codimension one. Fix a prime number $\ell$ distinct from the characteristic.
Then the following two subgroups of $H^{2d-3}_\et(X, \ZZ_\ell)
$ agree:
\[
\tilde{N}^{d-2}H_{\et}^{2d-3}(X, \ZZ_\ell)=N^{d-2}H^{2d-3}_\et(X, \ZZ_\ell).
\]
\end{thm}
 The proof relies on the techniques developed in the joint work with Koll\'ar \cite{Kollar_Tian}, which allows us to prove a structural result about the space of  one-cycles on a smooth projective variety  that is SRC in codimension one in Section \ref{sec:space} Theorem \ref{thm:filtered}, and deduce its consequences for the various coniveau filtrations in Section \ref{sec:complex} and Section \ref{sec:general}.

 Finally, combining results of Esnault-Wittenberg \cite{WittenbergEsnault_0_cycle}, and Bloch-Srinivas \cite{BlochSrinivas} (which uses Merkujev-Suslin's theorem on $K_2$ \cite[Theorem 11.5]{Merkujev_Suslin}), one proves that hypotheses (A), (B), (C) and (D) in Theorem \ref{thm:integralTate} holds for rational surface fibrations over a curve. This proves Theorem \ref{thm:integralTateSurface}.

\subsection{Related work}
\subsubsection{The Integral Tate conjecture for one-cycles}

 Colliot-Th\'el\`ene and Kahn studied the surjectivity of the $\ZZ_\ell$-coefficient cycle class map of codimension $2$ cycles and its relation with the degree $3$ unramified cohomology $H^3_{\text{nr}}(X, \QQ_\ell/\ZZ_\ell(2))$ in \cite{CTKahnCycleCodim2} (over the complex numbers, such a relation is studied in \cite{CTVoisin}). For more information, see Section \ref{sec:ITateHasse}.

Over finite fields, several groups of authors proved the vanishing of the degree $3$ unramified cohomology on certain threefolds and deduced the integral Tate conjecture for one-cycles, and thus proving Conjecture \ref{conj:CT1} and \ref{conj:CT2} for some surfaces defined over a global function field. See, for example, \cite{Parimala_Suresh_ruled} for the case of conic bundles over a surface for which the Tate conjecture is known, \cite{CT_Scavia_ITC} and \cite{Scavia_ITC} for the case of a product of a curve with a $\text{CH}_0$-trivial surface.

\subsubsection{The Local-global principle}
There is a vast literature on the local-global principles for zero-cycles/rational points on geometrically rational surfaces. Let us only mention a few relevant results and refer the readers to survey articles such as \cite{WittenbergRationalPointSurvey} etc. for a more comprehensive list.

Work of Colliot-Th\'el\`ene \cite{CT_zero_cycle_ruled} made important progress towards Conjecture \ref{conj:E} for ruled surfaces defined over number fields. Together with later works of Frossard \cite{Frossard}, van Hamel \cite{vanHamel_SB_curve}, they proved that Conjecture \ref{conj:E} holds  if the Jacobian of the base curve has finite Tate-Shafarevich group.
The global function field version for ruled surfaces is studied by Parimala-Suresh \cite{Parimala_Suresh_ruled}, which depends on the computation of the degree $3$ unramified cohomology as explained above. 
Assuming that the base surface satisfies the Tate conjecture, they establishes the integral Tate conjecture for conic bundle over surfaces defined over finite fields and Conjecture \ref{conj:E}.
Colliot-Th\'el\`ene-Swinnerton-Dyer \cite{CTSDPencil} studied the integral Tate conjecture and the local-global principle for zero-cycles on surfaces of the form $(f+tg=0)\subset \PP^3 \times \AAA^1_t$. In addition to proving that the integral Tate conjecture for $1$-cycles holds for such surfaces, for cubic surfaces of this form, they also prove that the existence of rational points is equivalent to the existence of a  zero-cycle of degree $1$ for such surfaces.

The study of complete intersections of two quadrics has also drawn a lot of attention. It started with the work of Colliot-Th\'el\`ene, Sansuc, and Swinnerton-Dyer \cite{CTSD}. Heath-Brown proved that the Hasse principle for rational points holds for smooth complete intersections of two quadrics in $\PP^7$ over number fields \cite{Heath-Brown_2_quadrics}, see also a different proof by Colliot-Th\'el\`ene \cite{CT2022CI22}. Under the assumption on the finiteness of the Tate-Shafarevich group of elliptic curves and the validity of Schinzel's hypothesis, Wittenberg proved that the Hasse principle holds for such complete intersections in $\PP^5$ and some cases in $\PP^4$ (with the Brauer-Manin obstruction) over number fields \cite{Wittenberg_Thesis}.
The author has shown in a previous paper \cite{Hasse} that the Hasse principle for rational points holds for smooth complete intersections of two quadrics in $\PP^n (n\geq 5)$ defined over a global function field of odd characteristic.

One is also interested in weak approximation for rational points on a del Pezzo surface of degree $4$.
For a del Pezzo surface of degree $4$ over a number field, assuming that there is a rational point, Salberger and Skorobogatov \cite{SS1991} proved that the Brauer-Manin obstruction is the only obstruction to weak approximation. 
As the author has been informed by Colliot-Th\'el\`ene, essentially the same argument also proves that over a global function field of odd characteristic, the Brauer-Manin obstruction is the only obstruction to weak approximation once there is a rational point. 
In characteristic $2$, some partial results are contained in the joint work of the author with Letao Zhang \cite{WACubicGlobal}.

\subsubsection{Coniveau and strong coniveau filtration} Voisin \cite{VoisinConiveau} was the first to study the two filtrations on $H^{2d-3}$ for rationally connected varieties and threefolds over the complex numbers. 
In particular, she related the strong coniveau filtration to the filtration given by families of one-cycles (Lemma \ref{lem:inclusion}), a strategy that we also follow. 
On the other hand, there are many examples showing that the two filtrations differ for a general smooth projective variety \cite{BO_Coniveau, ScaviaSuzuki_examples,Benoist2022Steenrod}.

\subsection{Structure of the paper}
Section \ref{sec:pre} consists of some preliminary results. 
We explain some basic properties of being SRC in codimension one (Section \ref{sec:SRC1}), the Bloch-Kato conjecture (the norm residue isomorphism theorem), the Beilinson-Lichtenbaum conjecture (Section \ref{sec:bkbl}), and functoriality of the higher cycle class map from Bloch's higher Chow groups to Borel-Moore homology (Section \ref{sec:bm}), several notions of filtrations due to Voisin/Benoist-Ottem (Section \ref{sec:coniveau}), and we verify that the hypotheses in Theorem \ref{thm:integralTate} only depends on the stable birational class of the generic fiber over $\overline{\FF}_q(B)$ (Section \ref{sec:bir}).

The main parts of the paper can be seen as applications of the general method developed in \cite{Kollar_Tian}. 
The first application is in Section \ref{sec:space}, where we describe some structural results of the space of one-cycles. We use Corollary 35 and Theorem 48 in \cite{Kollar_Tian} to prove Theorems \ref{thm:filtered_sm}, \ref{thm:filtered} about the space of one-cycles in Section \ref{sec:space}.
These structural results are then used in Sections \ref{sec:complex} \ref{sec:general} to study the coniveau filtration on $H_3$ of a smooth projective variety that is separably rationally connected in codimension one (Theorems \ref{thm:image}, \ref{thm:image_sheaf}), which, in turn, is the foundation for the main theorems about the integral Tate conjecture (Theorems \ref{thm:integralTate}, \ref{thm:integralTateSurface}) and the local-global principles (Theorem \ref{thm:zerocycle}), all of which are proved in Section \ref{proof}.

In Section \ref{sec:complex}, we study the coniveau filtration on $H_3$ of complex varieties using Lawson homology and the result is topological in nature. 
These results are not needed for the arithmetic applications.
But the proof illustrates the main ideas without using the heavy machinery of Chow sheaves.
Moreover, this approach gives ``integral" results, some of which are stronger than the results that one can obtain in the general case. 
We present this first to give the readers some flavor of the argument. 
To deal with the general case requires the use of Chow sheaves introduced by Suslin and Voevodsky in \cite{SV_Chow_Sheaves}, hence more abstract. 
Unfortunately, in this case we only have results for torsion coefficients and have to pass to the inverse limit from time to time. 
The criterion for the surjectivity in the arithmetic part is proved in Corollary \ref{cor:IHC}. 

Section \ref{sec:ITateHasse} reviews the relation between the integral Tate conjecture for  one-cycles, the degree $3$ unramified cohomology, and the local-global principle. The main theorems are proved in Section \ref{proof}.  

In Section \ref{sec:example}, we give some examples where the criterion in Theorem \ref{thm:integralTate} can be effectively checked.
Section \ref{sec:questions} consists of some questions and expectations.

The following diagram shows the dependence of the main theorems.

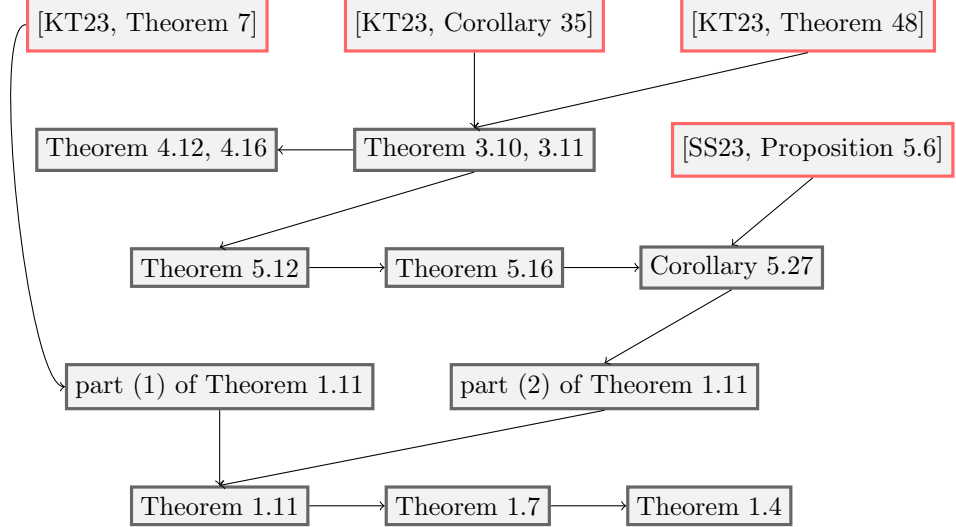
\begin{tikzpicture}[
roundnode/.style={rectangle, draw=red!60, fill=black!5, very thick, minimum size=7mm},
squarednode/.style={rectangle, draw=black!60, fill=black!5, very thick, minimum size=5mm},
]

\node[squarednode]      (structure)  {Theorem \ref{thm:filtered_sm}, \ref{thm:filtered}};

\node[squarednode](image)[below=of structure]
{Theorem \ref{thm:image_sheaf}};

\node[squarednode](surj)[left=of image]
{Theorem \ref{thm:sheaf_surj}};

\node[squarednode](corollary)[right=of image]{Corollary \ref{cor:IHC}};

\node[squarednode](surj_L)[left=of structure]
{Theorem \ref{thm:surjectivity}, \ref{thm:image}};

\node[roundnode]        (KT1)       [above=of structure] {\cite[Corollary 35]{Kollar_Tian}};

\node[roundnode]        (KT3)       [right=of KT1] {\cite[Theorem 48]{Kollar_Tian}};

\node[roundnode]        (SS)       [right=of structure] {\cite[Proposition 5.6]{ScaviaSuzuki2023coniveau}};

\node[roundnode]        (KT2)       [left=of KT1] {\cite[Theorem 7]{Kollar_Tian}};

\node[squarednode]      (geometric)       [below=of surj] {part (1) of Theorem \ref{thm:integralTate}};

\node[squarednode]      (arithmetic)   [right=of geometric]     {part (2) of Theorem \ref{thm:integralTate}};

\node[squarednode]      (ITC)       [below=of geometric] {Theorem \ref{thm:integralTate}};

\node[squarednode]      (ITC_surface)       [right=of ITC] {Theorem \ref{thm:integralTateSurface}};

\node[squarednode]      (lg)       [right=of ITC_surface] {Theorem \ref{thm:zerocycle}};

\draw[->] (KT3.south) -- (structure.north);
\draw[->] (KT1.south) -- (structure.north);
\draw[->] (structure.west) -- (surj_L.east);
\draw[->] (structure.south) -- (surj.north);
\draw[->] (surj.east) -- (image.west);
\draw[->] (image.east) -- (corollary.west);
\draw[->] (corollary.south)--(arithmetic.north);
\draw[->] (KT2.west) .. controls +(left: 5mm)and +(left:5 mm) .. (geometric.west);
\draw[->] (geometric.south) -- (ITC.north);
\draw[->] (arithmetic.south)-- (ITC.north);
\draw[->] (ITC.east) -- (ITC_surface.west);
\draw[->] (ITC_surface.east) -- (lg.west);
\draw[->] (SS.south) -- (corollary.north);
\end{tikzpicture}\\

\textbf{Acknowledgment:} I would like to thank Jean-Louis Colliot-Th\'el\`ene and Olivier Wittenberg for many helpful and constructive comments. 
I am grateful to J\'anos Koll\'ar for generously sharing his ideas and for co-authoring the article \cite{Kollar_Tian} which produces stronger results than those proved in the first version of this paper, 
and which provides the technical results needed for this paper. 
Finally, I thank the referees, who have provided many invaluable suggestions on how to improve the exposition. 
This work is partially supported by NSFC grants No. 11890660, No.11890662, No. 11871155.

\section{Preliminaries}\label{sec:pre}
In this section we collect some basic results that will be useful later, most (if not all) of which are well-known to experts. 
We include the proof for those results for which the author does not know a proof in the published literature for the sake of completeness.
\subsection{Separable rational connectedness in codimension one}\label{sec:SRC1}
We prove some basic properties of varieties that are SRC in codimension one in this section.
\begin{defn}
    A \emph{separably rationally connected fibration over a curve} is a flat morphism from a normal proper variety $Y$ to a smooth proper curve $C$ such that the smooth locus of the geometric generic fiber has a very free curve.
\end{defn}
\begin{lem}\label{lem:classification}
    Let $X$ be a smooth proper variety that is separably rationally connected in codimension one. Then $X$ is either birational to a separably rationally connected fibration over a curve, or separably rationally connected, or rationally connected by a family of free curves. The third case only happens in positive characteristic. 
\end{lem}

\begin{proof}
    One can take the rational quotient of the variety $X$ by the relation generated by free rational curves (\cite[Example IV.4.11, Theorem IV4.13]{Kollar96} or \cite[Theorem 1.3]{starr2006maximal}). 
    This gives an open subset $U \subset X$ and a morphism $U \to Z$ to a variety $Z$ such that the closure of a fiber over a geometric point of $Z$ is the equivalence class of points connected by a chain of free rational curves. 
    The SRC in codimension one condition implies that for any general geometric point, there is at least one free rational curve passing through it such that the deformation of this free rational curve containing the geometric point covers at least a divisor in $X$. So the quotient $Z$ is either a curve or a point. 

    If $Z$ is a point, then two general points of $X$ are connected by a chain of free rational curves.  A general smoothing of this chain of free rational curves gives a free rational curve that connects two general points of $X$. That is, $X$ is rationally connected by a family of free rational curves in the terminology of \cite{ShenFRC}. In characteristic $0$, separable rational connectedness is the same as rational connectedness for smooth proper varieties. 

    If $Z$ is a curve, we may assume that $Z$ is smooth projective and extend the morphism to an open subset $V \subset X$ so that the complement of $V$ has codimension at least $2$. General deformations of free rational curves lie in $V$. Lemma 2.5 in \cite{starr2006maximal} says that every free rational curve in $V$ that has non-empty intersection with the smooth locus of $V \to Z$ lie in the smooth locus. In particular, by definition of SRC in codimension one, there are free rational curves in the smooth locus of $V \to Z$ which are very free when considered as curves in the fiber. So $X$ is birational to a separably rationally connected fibration over a curve that is a normal projective compactification of $V$ with a fibration structure over $Z$.
\end{proof}

\begin{rem}
    In positive characteristic, examples of the third kind can be constructed using purely inseparable covers of SRC varieties (\cite[Example V.5.19]{Kollar96}). They belong to the class of \emph{freely rationally connected varieties} studied in \cite{ShenFRC}.
\end{rem}

\begin{defn}
    Let $X$ be a variety defined over an algebraically closed field $k$ and $Z \subset X$ a closed subset, endowed with the reduced closed subscheme structure. The rational Chow group of zero-cycles on $X$ is \emph{universally supported in $Z$} if for every algebraically closed field extension $L/k$, the map
    \[
    \text{CH}_0(Z_L)\otimes \QQ \to \text{CH}_0(X_L)\otimes \QQ
    \]
    is surjective.
\end{defn}
The following is well-known (see for example \cite[Proposition 1]{BlochSrinivas}).
\begin{lem}\label{lem:dd}
    Let $X$ be a smooth proper variety defined over an algebraically closed field $k$ and $Z \subset X$ a closed subset. The rational Chow group of zero-cycles on $X$ is universally supported in $Z$ if and only if there is a decomposition of the diagonal of the form
    \[
    N \Delta_X=\Gamma_1+\Gamma_2 \in \text{CH}_{\dim X}(X \times X),
    \]
    where $N$ is an integer, $\Gamma_1$ is a cycle supported in $X \times Z$, and $\Gamma_2$ is a cycle supported in $D\times X$ for some divisor $D \subset X$.
\end{lem}
As a corollary of the structure of varieties that are SRC in codimension one, we have the following.
\begin{cor}\label{rem:support}
    Let $X$ be a smooth proper variety that is separably rationally connected in codimension one. Then the rational Chow group of zero-cycles on $X$ is universally supported on a curve.
    Equivalently, there is a one dimensional closed subset $Z\subset X$, a divisor $D \subset X$, and a decomposition of the diagonal of the form
    \[
    N \Delta_X=\Gamma_1+\Gamma_2 \in \text{CH}_{\dim X}(X \times X),
    \]
    where $N$ is an integer, $\Gamma_1$ is a cycle supported in $X \times Z$, and $\Gamma_2$ is a cycle supported in $D\times X$. 
\end{cor}
\begin{proof}
    Because of the structure of $X$ described in Lemma \ref{lem:classification}, we may take the curve to be a general smooth complete intersection of very ample divisors. Then every geometric point is connected to such a curve by a chain of rational curves. The rest follows from Lemma \ref{lem:dd}.
\end{proof}
\subsection{The Bloch-Kato conjecture and the Beilinson-Lichtenbaum conjecture}\label{sec:bkbl}
In this section, we collect several results related to the Bloch-Kato conjecture and give references.
\begin{thm}[the Bloch-Kato conjecture, or the norm residue isomorphism theorem, {\cite[Theorem 6.1]{Voevodsky_Bloch_Kato}}]\label{bk}
    Let $k$ be a field and $m$ a positive integer relatively prime to the characteristic of $k$. The norm residue map from Milnor K-group to \'etale cohomology induces an isomorphism
    \[
    K^M_i(k)/m K^M_i(k) \xrightarrow{\cong} H^i_\et(k, \mu_m^{\otimes i})
    \]
    for all non-negative integers $i$.
\end{thm}
The case $i=2$, also called the Merkurjev-Suslin theorem on $K_2$, is proved by Merkurjev-Suslin earlier \cite[Theorem 11.5]{Merkujev_Suslin}.
It is known that the Bloch-Kato conjecture is equivalent to the following conjecture of Beilinson-Lichtenbaum (\cite[Theorem 7.4]{Suslin_Voevodsky_BK_BL}, \cite[Theorem 1.1]{Geisser_Levine} ).
\begin{thm}[the Beilinson-Lichtenbaum conjecture]\label{bl}
    Fix a field $k$. For all positive integers $q$, for all positive integers $m$ relatively prime to the characteristic of $k$, the natural map 
    \[
    \ZZ(q) \otimes \ZZ/m\ZZ \xrightarrow{\cong}\tau^{\leq q}R\pi_*(\mu_m^{\otimes q})
    \]
is an isomorphism in the derived category of Zariski sheaves on a smooth
variety $X$ over $k$, where $\ZZ(q) (q \geq 0)$ is the motivic complex defined in \cite[Definition 3.1]{Suslin_Voevodsky_BK_BL}, $\pi: X_\et \to X_{\text{Zar}}$ is the morphism from the \'etale site of $X$ to the Zariski site of $X$, and $\tau$ is the good truncation functor.
\end{thm}

For later use, we record the following corollary. It is well-known to experts, but it is not so easy to find a proof in the literature. Over the complex numbers with singular cohomology, the statement is \cite[Proposition 2.8]{BO_Coniveau}, which depends on \cite[Th\'eor\`eme 3.1]{CTVoisin}. 
\begin{cor}\label{cor:torsionfree}
    Let $X$ be a smooth variety defined over an algebraically closed field $k$. Fix $\ell$ a prime number distinct from the characteristic of $k$.  All the torsion classes in the \'etale cohomology $H^{i+1}(X, \ZZ_\ell(i)) (i \geq 0)$ are supported in divisors, and $H^{i+1}(X, \ZZ_\ell(i))/N^1H^{i+1}(X, \ZZ_\ell(i))$ is torsion free.

    The same conclusion holds if $X$ is a complex variety and we use singular cohomology with $\ZZ$-coefficients.
\end{cor}
\begin{proof}
    For each positive integer $n$, we have a long exact sequence
    \[
    \ldots \to H^i(X, \ZZ_\ell(i))\xrightarrow{\cdot \ell^n} H^i(X, \ZZ_\ell(i))\xrightarrow{\mod \ell^n} H^i(X, \mu_{\ell^n}^{\otimes i})\to H^{i+1}(X, \ZZ_\ell(i)) \to \ldots,
    \]
    and a similar long exact sequence for any open subset of $X$.
    A torsion class $\alpha_0$ in $H^{i+1}(X, \ZZ_\ell(i))$ comes from a class $\alpha \in H^i(X, \mu_{\ell^n}^{\otimes i})$ for some $n$.
    By the Bloch-Kato conjecture (Theorem \ref{bk}), we have a surjection
    \[
    K_i^M(k(X)) \to K_i^M(k(X))/\ell^n \xrightarrow{\cong} H^{i}(k(X), \mu_{\ell^n}^{\otimes i}).
    \]
    So the image of $\alpha$ in $H^{i}(k(X), \mu_{\ell^n}^{\otimes i})$
     is the image of a finite sum of symbols
     \[
     \sum_j (a_{j1}, \ldots, a_{ji}), a_{j1}, \ldots, a_{ji} \in k(X)^*.
     \]
     There is an open subset $U\subset X$, where all the rational functions involved in this finite sum of symbols are regular functions on $U$, and thus define a class $\tilde{\alpha} \in K_i^M(U)$.
     Then  $\tilde{\alpha}$ defines a cohomology class $[\tilde{\alpha}]$ in
    \[
    H^i(U, \ZZ_\ell(i))=\lim \limits_{\xleftarrow[m]{}}H^i(U, \mu_{\ell^m}^{\otimes i})
    \]
    via the compatible maps $$\wedge^i H^0(U, \mathbb{G}_m) \to \wedge^i H^1(U, \mu_{\ell^m})\to H^i(U, \mu_{\ell^m}^{\otimes i}).$$
    The image of the cohomology class $[\tilde{\alpha}]$ under the composition 
    \[
    H^i(U, \ZZ_\ell(i)) \xrightarrow{\mod \ell^n} H^i(U, \mu_{\ell^m}^{\otimes i}) \to H^i(k(X), \mu_{\ell^n}^{\otimes i})\]
     maps to the image of $\alpha$ in $H^i(k(X), \mu_{\ell^n}^{\otimes i})$.
    So there is an open subset $V\subset U$ such that the cohomology class $[\tilde{\alpha}]|_V$ 
     maps to $\alpha|_V$ via the map
    \[
    H^i(V, \ZZ_\ell(i)) \xrightarrow{\mod \ell^n} H^i(V, \mu_{\ell^n}^{\otimes i}).
    \]
    Therefore $\alpha_0|_V=0$.
    That is, the class $\alpha_0$ is supported in a divisor.

    For complex varieties, \cite[Proposition 2.8]{BO_Coniveau}, which depends on \cite[Th\'eor\`eme 3.1]{CTVoisin}, shows that torsion classes in singular cohomology with $\ZZ$-coefficient are supported on divisors.

     Let $\alpha$ be a class in $H^{i+1}(X, \ZZ_\ell(i))$ that becomes torsion in $$H^{i+1}(X, \ZZ_\ell(i))/N^1H^{i+1}(X, \ZZ_\ell(i)).$$ Then some multiple of $\alpha$ is supported in a codimension one algebraic set $D$ in $X$. We have the localization sequence for \'etale cohomology:
    \[
    \ldots \to H^{i+1}_D(X, \ZZ_\ell(i)) \to H^{i+1}(X, \ZZ_\ell(i)) \to H^{i+1}(X\backslash D, \ZZ_\ell(i)) \to \ldots.
    \]
    Since torsion classes are supported on divisors for \emph{any} smooth varieties, in particular, $X\backslash D$,  there is a codimension one algebraic set $D'$ in $X$ such that $\alpha$ restrict to the zero class in $H^{i+1}(X\backslash (D\cup D'), \ZZ_\ell(i))$. That is, $\alpha$ is supported in $D\cup D'$.

    The same argument works for singular cohomology on smooth complex varieties.
\end{proof}

\subsection{Higher Chow groups, Borel-Moore homology, and higher cycle class map}\label{sec:bm}
Higher Chow groups were first defined by Bloch in \cite{Bloch_higher_Chow}. 
\begin{defn}\label{def:higherchow}
    Let $k$ be a scheme and let $X$ be an equi-dimensional finite type $k$-scheme. We write $z^q(X,p)$ for the free abelian group generated by all codimension $q$ subvarieties on $X \times \Delta^p$ which intersect all faces $X \times \Delta^j$ properly for all $j<p$, where $$\Delta^p=\SP k[t_0, \ldots, t_p]/\langle \sum_j t_j=1 \rangle$$ is the algebraic $p$-simplex. 
    We write $z^q(X, *)$ for the associated chain complex, whose boundary maps are alternating sums of restrictions to faces. 
    
    For any abelian group $A$, the higher Chow group with coefficient $A$, denoted by $\text{CH}^q(X, p, A)$, is the $p$-th homology of the complex $z^q(X, *)\otimes_\ZZ A$. 
    We define $\text{CH}_i(X, j, A)$ as $\text{CH}^{\dim X-i}(X, j, A)$. We may omit $A$ if $A=\ZZ$.
\end{defn}
Later, higher Chow groups were shown to be naturally isomorphic to motivic cohomology for smooth varieties. 
\begin{thm}\cite[Corollary 2]{Voevodsky_higherChow_motivic}
    For any field $k$, any smooth variety $X$ over $k$ and any $p, q \in \ZZ_{\geq 0}$, there is a natural isomorphism
    \[
    H^{2q-p}(X, \ZZ(q)) \xrightarrow{\cong} \text{CH}^q(X, p),
    \]
    and the same holds for the motivic cohomology and higher Chow groups with coefficients.
\end{thm}
By definition, the groups in the above theorem vanish if $p<0$ or $q<0$.
The motivic cohomology group $H^{2q-p}(X, \ZZ(q))$ vanishes if $2q-p<0$ and $X$ is smooth \cite[Theorem 19.3]{Mazza_Voevodsky_Weibel_Lecture_motivic_cohomology}.

As a corollary of this theorem and the Beilinson-Lichtenbaum conjecture \ref{bl},
we have
\begin{cor}\label{cor:chow_zar}
    Let $X$ be a smooth variety over a perfect field $k$ and $m$ an integer relatively prime to the characteristic. For $p\geq 0, q \geq 0$, we have an isomorphism
    \[
    \text{CH}^q(X, p, \ZZ/m)\xrightarrow{\cong} H^{2q-p}(X, \ZZ(q)\otimes\ZZ/m) \xrightarrow{\cong} H_{\text{Zar}}^{2q-p}(X, \tau^{\leq q}R\pi_*(\mu_m^{\otimes q})).
    \]
\end{cor}

At least for the purpose of this paper, it is more natural to view higher Chow groups as a Borel-Moore type homology theory.
And one can define a higher cycle class map to the Borel-Moore homology over a field $k$.  This material is well-known to experts. When $k$ is an algebraically closed field, we find the details in \cite[Section 4]{kok_Zhou}.
\begin{defn}
    For a scheme $f: X \to \SP k$ defined over a perfect field $k$, and fix $\Lambda=\ZZ/m\ZZ$ such that $m$ is relatively prime to the characteristic of $k$, the Borel-Moore homology is defined as
\[
H^{\text{BM}}_n(X, \Lambda(i)):=H^{-n}_\et(X, f^!\Lambda(-i)).
\]
For a prime number $\ell$ different from the characteristic of $k$, we define
\[
H^{\text{BM}}_n(X, \ZZ_\ell(i)):=\lim \limits_{\xleftarrow[s]{}}H^{\text{BM}}_n(X, \ZZ/\ell^s(i)).
\]
\end{defn}

In particular, if $X$ is smooth of pure dimension $d$, $f^!=f^*$ and we have a natural isomorphism
\begin{equation}\label{sm}
    H^{\text{BM}}_n(X, \Lambda(i))\xrightarrow{\cong} H_\et^{2d-n}(X, \Lambda(d-i)).
\end{equation}

Borel-Moore homology is functorial for proper push-forward, \'etale pull-back, and satisfies the localization long exact sequence (\cite[Example 2.1]{BlochOgus}). Since $f^!=f^*$ on smooth varieties, for a morphism $g: X \to Y$ between smooth varieties, we have pull-back via the isomorphism (\ref{sm}):
\[
g^*: H^{\text{BM}}_{2\dim Y-n}(Y, \Lambda(\dim Y-i)) \to H^{\text{BM}}_{2\dim X-n}(X, \Lambda(\dim X-i)).
\]

In \cite[Definition 4.2]{kok_Zhou}, for any variety $X$ defined over an algebraically closed field, the authors define a higher cycle class map
\begin{equation}\label{eq:highclassmapBM}
    cl_{i,j}: \text{CH}_i(X, j, \Lambda) \to H^{\text{BM}}_{2i+j}(X, \Lambda(i)). 
\end{equation}
It seems to the author that it is not necessary to assume that the base field is algebraically closed for their definition and results that we quote in the following to hold. 
But to be consistent with the reference \cite{kok_Zhou}, and since we will only need the results for varieties defined over algebraically closed fields, we work over an algebraically closed field $k$ in the following.

For smooth varieties over an algebraically closed field $k$, by identifying the Borel-Moore homology with \'etale cohomology, this cycle class map can also be defined as (\cite[Definition 4.2]{kok_Zhou})
\begin{equation}\label{eq:highclassmapet1}
    cl^{d-i,j}: \text{CH}^{d-i}(X, j, \Lambda) \to H_{\text{\'et}}^{2d-2i-j}(X, \Lambda(d-i)). 
\end{equation}

\begin{rem}\label{rem:class}
    The higher cycle class map (\ref{eq:highclassmapBM}) is defined via a change of topology: higher Chow group of $X$ is computed as the Hom-set between two objects in the derived category of Voevodsky motives in the cdh topology \cite[Proposition 4.1]{kok_Zhou}:
\[
\text{CH}_i(X, j, \Lambda)\cong \text{Hom}_{DM_{cdh}}(\Lambda(i)[2i+j], Rf_*f^!\Lambda),
\]
 while Borel-Moore homology is computed as the Hom-set between two objects in the derived category of Voevodsky motives in the h topology:
\[
H_{2i+j}^{\text{BM}}(X, \Lambda(i))\cong \text{Hom}_{DM_{h}}(\Lambda(i)[2i+j], Rf_*f^!\Lambda).
\] 
Then the higher cycle class map is induced by the ``realization functor" (\cite[9.1.1]{Cisinski_Deglise_mix_motive})
\[
DM_{cdh} \to DM_h.
\]
 The definition of higher cycle class maps for smooth varieties is obtained from this higher cycle class map by applying duality in $DM_{cdh}$ and $DM_h$ (\cite[Definition 4.2]{kok_Zhou}).
\end{rem}

\begin{prop}[{\cite[Proposition 4.3]{kok_Zhou}}]\label{prop:func}
    For varieties defined over an algebraically closed field, the higher cycle class map $cl_{i, j}$ defined above is functorial with respect to proper push-forward, \'etale pull-back, and compatible with the localization long exact sequence.
\end{prop}

For smooth varieties,  Geisser-Levine defines a higher cycle class map \cite{Geisser_Levine}(see formula (3.3) and the ensuing paragraph).
\begin{align}\label{eq:higherclassmapet2}
    cl^{q, p}: &\text{CH}^q(X, p, \ZZ/m) \xrightarrow{\cong} H_{\text{Zar}}^{2q-p}(X, \tau^{\leq q}R\pi_*(\mu_m^{\otimes q}))\\ 
    \to &H_{\text{Zar}}^{2q-p}(X, R\pi_*(\mu_m^{\otimes q}))\xrightarrow{\cong} H^{2q-p}_\et(X, \mu_m^{\otimes q}), p, q\geq 0.\nonumber
\end{align}
\begin{prop}\cite[Proposition 4.4]{kok_Zhou}
    For smooth varieties defined over an algebraically closed field, the two definitions (\ref{eq:highclassmapet1}), (\ref{eq:higherclassmapet2}) agree (up to a natural isomorphism).
\end{prop}

As an immediate corollary, we have the following.
\begin{cor}\label{cor:ch_BM}
    Let $X$ be a reduced finite type scheme defined over an algebraically closed field. Fix a positive integer $m$ that is relatively prime to the field characteristic. Then for any $n, 0 \leq n \leq 2d$, the higher cycle class map induces an  isomorphism
    \[
    cl_{0, n}:\text{CH}_0(X, n, \ZZ/m) \xrightarrow{\cong} H_n^{\text{BM}}(X, \ZZ/m).
    \]
\end{cor}
\begin{proof}
    If $X$ is smooth of equi-dimension $d$, the higher cycle class map is a composition of isomorphisms
    \begin{align*}
        &\text{CH}_0(X, n, \ZZ/m) \xrightarrow{\cong} H_{\text{Zar}}^{2d-n}(X, \tau^{\leq d}R\pi_*(\mu_m^{\otimes d})) \\
        \xrightarrow{\cong} &H_{\text{Zar}}^{2d-n}(X, R\pi_*(\mu_m^{\otimes d}))\xrightarrow{\cong} H^{2d-n}_\et(X, \mu_m^{\otimes d}) \cong H_n^{\text{BM}}(X, \ZZ/m),
    \end{align*}
    where for the second isomorphism, we use the fact that $R\pi_*(\mu_m^{\otimes d})$ has no cohomology in degree $d+1$ and higher.
    
    In general, we use induction on dimension. Dimension $0$ case is trivially true. Write $X$ as the disjoint union of its singular locus $Z$ and smooth locus $U$. Then the statement is true for both $Z$ (by induction hypothesis) and $U$ (because it is smooth). Since the cycle class maps are compatible with the localization sequence for higher Chow groups and Borel-Moore homology (Proposition \ref{prop:func}), we have a commutative diagram of localization long exact sequences for higher Chow groups and Borel-Moore homology with $\ZZ/m$-coefficients:
    \[
    \begin{CD}
       \ldots \text{CH}_0(Z, j) @>>>\text{CH}_0(X, j) @>>>\text{CH}_0(U, j) @>>>\text{CH}_0(Z, j-1) \ldots \\
        @VVcl_{0, j}^Z V @VVcl_{0, j}^X V @VVcl_{0, j}^U V @VVcl_{0, j-1}^ZV\\
       \ldots H_{j}^{\text{BM}}(Z) @>>>H_{j}^{\text{BM}}(X) @>>>H_{j}^{\text{BM}}(U) @>>>H_{j-1}^{\text{BM}}(Z)\ldots .\\
    \end{CD}
    \]
    So the desired isomorphism follows from five lemma.
\end{proof}

Another immediate corollary of identifying the two higher cycle class maps \cite[Proposition 4.4]{kok_Zhou} is the following.
\begin{cor}[{\cite[Proposition 4.2(i), 4.7]{Geisser_Levine}}]\label{cor:pull-back}
    The higher cycle class maps (\ref{eq:highclassmapet1}), (\ref{eq:higherclassmapet2}) are compatible with pull-back between smooth varieties and the product structure on higher Chow groups and the \'etale cohomology for smooth varieties defined over an algebraically closed field.
\end{cor}

Let $X, Y$ be smooth projective varieties and $\Gamma \in \text{CH}_{\dim X+n}(X\times Y)\otimes \Lambda$ be a correspondence.  Since correspondence actions are compositions of pulling-back, taking product, and pushing-forward, we have the following.

\begin{prop}\label{class_cor}
    The higher cycle class maps for smooth projective varieties (\ref{eq:highclassmapet1}) (and hence also (\ref{eq:higherclassmapet2})) are compatible with correspondence action.
\end{prop}

\begin{lem}\label{lem:modN}
    Let $X$ be a variety defined over a field $k$. Fix a short exact sequence of abelian groups
    \[
    0 \to A \xrightarrow{f} B \xrightarrow{g} C \to 0.
    \]
    There is a long exact sequence of higher Chow groups with coefficients in $A, B, C$: 
    \[
    \ldots \to \text{CH}^q(X, p, A) \xrightarrow{id\otimes f} \text{CH}^q(X, p, B) \xrightarrow{id \otimes g} \text{CH}_q(X, p, C) \to \text{CH}^q(X, p-1, A) \to\ldots,
    \]
    for $0 \leq q \leq \dim X$.
\end{lem}
\begin{proof}
    This comes from taking homology of the short exact sequence of complexes:
    \[
    0 \to z^{q}(X, *)\otimes A \xrightarrow{id \otimes f}z^{q}(X, *) \otimes B\xrightarrow{id \otimes g}z^{q}(X, *)\otimes C \to 0.
    \]
    We have injection on the left because the group $z^{q}(X, p)$ is free.
\end{proof}

Finally, we define the higher cycle class maps
\[
\text{CH}_i(X, j, \ZZ_\ell) \to H_{2i+j}^{\text{BM}}(X, \ZZ_\ell(i))
\]
by the composition
\[
\text{CH}_i(X, j, \ZZ_\ell) \to \lim \limits_{\xleftarrow[n]{}} \text{CH}_i(X, j, \ZZ/\ell^n) \to  \lim \limits_{\xleftarrow[n]{}}H_{2i+j}^{\text{BM}}(X, \ZZ/\ell^n(i)) = H_{2i+j}^{\text{BM}}(X, \ZZ_\ell(i)).
\]
With this definition, we have the following.

\begin{prop}\label{prop:bockstein}
    Let $X$ be a variety defined over an algebraically closed field and $\ell$ a prime number different from the field characteristic. We have a commutative diagram of long exact sequences:
    \[
    \begin{CD}
        \ldots \to \text{CH}_i(X, j, \ZZ_\ell) @>\times \ell^n>> \text{CH}_i(X, j, \ZZ_\ell) @>>> \text{CH}_i(X, j, \ZZ/\ell^n) \\
         @VVV @VVV @VVV  \\
        \ldots \to H_{2i+j}^{\text{BM}}(X, \ZZ_\ell(i)) @>\times \ell^n >> H_{2i+j}^{\text{BM}}(X, \ZZ_\ell(i)) @>>> H_{2i+j}^{\text{BM}}(X, \ZZ/\ell^n(i)) 
    \end{CD}
    \]
    \[
    \begin{CD}
        \to {\text{CH}}_i(X, j-1, \ZZ_\ell) @>\times \ell^n >> \text{CH}_i(X, j-1, \ZZ_\ell) @>>> \text{CH}_i(X, j-1, \ZZ/\ell^n)\ldots \\
         @VVV @VVV @VVV  \\
        \to H_{2i+j-1}^{\text{BM}}(X, \ZZ_\ell(i)) @>\times \ell^n >> H_{2i+j-1}^{\text{BM}}(X, \ZZ_\ell(i)) @>>> H_{2i+j-1}^{\text{BM}}(X, \ZZ/\ell^n(i))\ldots
    \end{CD}
    \]
\end{prop}
\begin{proof}
    We have a commutative diagram of short exact sequences of coefficients ($n, m \geq 0$):
    \[
    \begin{CD}
        0 @>>> \ZZ_\ell @>\times \ell^n>> \ZZ_\ell @>\mod \ell^n>> \ZZ/\ell^n @>>> 0\\
        @| @VV\mod \ell^m V @VV\mod \ell^{n+m}V @| @|\\
        0 @>>> \ZZ/\ell^m @>>> \ZZ/\ell^{n+m} @>>> \ZZ/\ell^n @>>> 0.\\
    \end{CD}
    \]
    This induces a commutative diagram of long exact sequences (written vertically), 
    \[
    \begin{CD}
         \text{CH}_i(X, j, \ZZ_\ell) @>{\mod \ell^m}>> \text{CH}_i(X, j, \ZZ/\ell^m) @>cl_{i, j}>> H_{2i+j}^{\text{BM}}(X, \ZZ/{\ell^m}(i)) \\
        @VV\times \ell^n V @VVV @VVV \\
         \text{CH}_i(X, j, \ZZ_\ell) @>{\mod \ell^{n+m}}>> \text{CH}_i(X, j, \ZZ/{\ell^{n+m}}) @>cl_{i, j}>> H_{2i+j}^{\text{BM}}(X, \ZZ/{\ell^{n+m}}(i))\\
        @VV\mod \ell^n V @VV\mod \ell^n V @VV\mod \ell^n V \\
          \text{CH}_i(X, j, \ZZ/\ell^n) @=\text{CH}_i(X, j, \ZZ/\ell^n) @>cl_{i, j}>> H_{2i+j}^{\text{BM}}(X, \ZZ/{\ell^n}(i))\\
          @VVV @VVV @VVV\\
          \text{CH}_i(X, j-1, \ZZ_\ell)@>{\mod \ell^m}>> \text{CH}_i(X, j-1, \ZZ/\ell^m) @>cl_{i, j-1}>> H_{2i+j-1}^{\text{BM}}(X, \ZZ/{\ell^m}(i))\\
          @VV\times \ell^n V @VVV @VVV \\
         \text{CH}_i(X, j-1, \ZZ_\ell) @>{\mod \ell^{n+m}}>> \text{CH}_i(X, j-1, \ZZ/{\ell^{n+m}}) @>cl_{i, j-1}>> H_{2i+j-1}^{\text{BM}}(X, \ZZ/{\ell^{n+m}}(i)),
    \end{CD}
    \]
    where the first two columns are long exact sequences of higher Chow groups constructed in Lemma \ref{lem:modN}, and the last column is the Bockstein long exact sequence of Borel-Moore homology induced by the exact sequence $0 \to \ZZ/\ell^m \to \ZZ/\ell^{n+m} \xrightarrow{\mod \ell^n} \ZZ/\ell^n \to 0$.
    
    Then we have a commutative diagram after taking the inverse limit with respect to $m$.
    \[
    \begin{CD}
         \text{CH}_i(X, j, \ZZ_\ell) @>>> \lim \limits_{\xleftarrow[m]{}}\text{CH}_i(X, j, \ZZ/\ell^m) @>cl_{i, j}>> H_{2i+j}^{\text{BM}}(X, \ZZ_\ell(i)) \\
        @VV\times \ell^n V @VVV @VV\times \ell^nV \\
         \text{CH}_i(X, j, \ZZ_\ell) @>>> \lim \limits_{\xleftarrow[m]{}}\text{CH}_i(X, j, \ZZ/{\ell^{n+m}}) @>cl_{i, j}>> H_{2i+j}^{\text{BM}}(X, \ZZ_\ell(i))\\
        @VV\mod \ell^n V @VV\mod \ell^n V @VV\mod \ell^n V \\
          \text{CH}_i(X, j, \ZZ/\ell^n) @=\text{CH}_i(X, j, \ZZ/\ell^n) @>cl_{i, j}>> H_{2i+j}^{\text{BM}}(X, \ZZ/{\ell^n}(i))\\
          @VVV @VVV @VVV\\
          \text{CH}_i(X, j-1, \ZZ_\ell)@>>> \lim \limits_{\xleftarrow[m]{}}\text{CH}_i(X, j-1, \ZZ/\ell^m) @>cl_{i, j-1}>> H_{2i+j-1}^{\text{BM}}(X, \ZZ_\ell(i))\\
          @VV\times \ell^n V @VVV @VV\times \ell^nV \\
         \text{CH}_i(X, j-1, \ZZ_\ell) @>>> \lim \limits_{\xleftarrow[m]{}}\text{CH}_i(X, j-1, \ZZ/{\ell^{n+m}}) @>cl_{i, j-1}>> H_{2i+j-1}^{\text{BM}}(X, \ZZ_\ell(i)).
    \end{CD}
    \]
    The right column is still a long exact sequence since every Borel-Moore homology group is finite and the Mittag-Leffler condition is satisfied.
    This proves the desired statement.
\end{proof}
     
\subsection{Various filtrations}\label{sec:coniveau}
This section studies several different notions of filtrations on the cohomology of a variety  introduced by Benoist-Ottem and Voisin in \cite{BO_Coniveau, VoisinConiveau}. 
\begin{conv}\label{conv}
We use the following conventions for cohomology and homology in this section.
Let $X$ be a smooth projective variety of dimension $d$ defined over an algebraically closed field $k$. 
Given an abelian group $A$ that is one of $\ZZ/m, \ZZ_\ell, \ZZ, \QQ, \QQ_\ell$, where $m$ is a positive integer, $\ell$ a prime number, both of which are relatively prime to the characteristic of $k$, we simply write $H^k(X, A)$ as the \'etale cohomology with coefficient $A$ (resp. the singular cohomology with coefficient $A$, if $X$ is a complex variety), and $H_i(*, A)$ as the \'etale Borel-Moore homology (resp. singular Borel-Moore homology).
\end{conv}
\begin{defn}\label{def:filtration}
 With conventions as in Convention \ref{conv}, we define the following filtrations on the cohomology $H^k(X, A)$.
\begin{enumerate}
\item The coniveau filtration $\{N^*H^k(X, A)\}$ with
\[
N^cH^k(X, A):=\sum_{f: Y \to X}f_*(H_{2d-k}(Y, A)) \subset H_{2d-k}(X, A)\cong H^k(X, A), 
\]
where the sum is taken over all morphisms from projective algebraic sets $f: Y \to X, \dim Y \leq d-c$;
\item  The strong coniveau filtration $\{\tilde{N}^*H^k(X, A)\}$ with
\[
\tilde{N}^cH^k(X, A):=\sum_{f: Y \to X}f_*(H_{2d-k}(Y, A)) \subset H_{2d-k}(X, A)\cong H^k(X, A),
\]
where the sum is taken over all morphisms from \textit{smooth} projective varieties $f: Y \to X, \dim Y \leq d-c$.
\item The strong cylindrical filtration $\{\tilde{N}_{*, \cyl}H^k(X, A)\}$ with
\[
\tilde{N}_{c, \cyl}H^{k}(X, A) :=\sum \Gamma_*( H_{2d-k-2c}(Z, A)) \subset H_{2d-k}(X, A) \cong H^k(X, A)
\]
where the sum is taken over all \emph{smooth} projective varieties $Z$ and subvarieties $\Gamma \subset Z\times X$ such that $\Gamma \to Z$ is dominant of relative dimension $c$.

\item We use the notations $N^cH_k$ etc. to denote the filtrations on Borel-Moore or singular homology $H_k$. Since $X$ is smooth, this is the same as the filtrations $N^cH^{2d-k}$. 
\end{enumerate}
\end{defn}

\begin{defn}\label{defn:corr}
    Let $X, Y, Z$ be smooth projective varieties, and let $p:Y \to Z, q:Y \to X$ be two morphisms. We define the correspondence action $Y_*$ on higher Chow groups and cohomology as
    \[
    Y_*=q_*\circ p^*.
    \]
\end{defn}
\begin{lem}\label{lem:image_corr}
    Keep the same notations as in Definition \ref{defn:corr}. If $Y \xrightarrow{p\times q} Z\times X$ is generically finite of degree $d$ onto its image, then we have the following formula:
\begin{equation}\label{eq:push-pull}
    Y_*=d \cdot \Gamma_{Y*},
\end{equation}
where $\Gamma_Y \in \text{CH}_{\dim Y}(Z\times X)$ is the cycle class of the image of $Y$. 

Consequently, if $Y$ is birational to its image, or if we consider higher Chow groups or cohomology with coefficients in which $d$ is invertible, the image of $Y_*$ is the same as $\Gamma_{Y*}$.
\end{lem}
\begin{proof}
Denote by $\pi_1, \pi_2$ the projection of $Z \times X$ to $Z$ and $X$.
    By the push-pull formula for higher Chow groups and cohomology, for any element $\alpha$ of the higher Chow group or cohomology of $Z$, we have
    \[
    \pi_1^*\alpha \cup (d\cdot \Gamma_Y)=\pi_1^*\alpha \cup (p\times q)_*([Y])=(p\times q)_*((p\times q)^* (\pi_1^*\alpha))=(p\times q)_*(p^*\alpha).
    \]
    Applying $\pi_{2*}$ to the above equality proves Formula (\ref{eq:push-pull}). The last statement is clear by the formula.
\end{proof}

We note that the strong cylindrical filtration is indeed a filtration.
\begin{lem}\label{lem:contain}
    Keep the same notations as in Convention \ref{conv}.
    We have $$\tilde{N}_{c+1, \cyl}H^k(X, A) \subset \tilde{N}_{c, \cyl}H^k(X, A), c\geq 1.$$
\end{lem}
\begin{proof}
    Note that the statement for $\ZZ/m$ follows from that for $\ZZ/\ell^r$ by considering prime factors of $m$ separately. So we only consider \'etale cohomology with $\ZZ/\ell^r$, $\ZZ_\ell$, or $\QQ_\ell$ coefficients, or singular cohomology with $\ZZ$ or $\QQ$ coefficients (over complex numbers) in the following.
    
     We may find a blow-up $\Gamma_1$ of $\Gamma$, with morphisms $\Gamma_1 \to Y \to Z$, where $Y \to Z$ has relative dimension $1$ and $\Gamma_1 \to Y$ has relative dimension $c$. 
    
    Using resolution of singularities (in case $A=\ZZ$ or $\QQ$) or Gabber's refinement of de Jong's alteration (in other cases), we may find smooth projective varieties $\Gamma'$ and $Y'$, with a commutative diagram of varieties over $Z$
    \[
    \begin{CD}
        \Gamma' @>>>Y'\\
        @VVV @VVV\\
        \Gamma_1 @>>>Y,
    \end{CD}
    \]
    where $\Gamma' \to \Gamma_1$ and $Y'\to Y$ are birational (in case $A=\ZZ$ or $\QQ$) or generically finite with degree relatively prime to $\ell$ (in other cases).

    By Lemma \ref{lem:image_corr}, as a correspondence from $Y$ to $X$, $\Gamma_*$ has the same image as $\Gamma_*'$, which is the composition of first pulling back to $\Gamma'$ and then pushing forward to $X$.
    So the image is contained in $\Gamma_*'(H_{2d-k-2c}(Y'))$ for the correspondence given by the morphism $\Gamma'\to Y' \times X$, which lies in $\tilde{N}_{c, \cyl}H^k(X, A)$.
\end{proof}

\begin{cor}\label{cor:contain}
    Keep the same notations as in Convention \ref{conv}.
    Let $Z$ be a smooth projective variety, $\Gamma \subset Z \times X$ a subvariety of dimension $\dim Z+c (c\geq 1)$, and $W \subset Z$ the image of $\Gamma$ under the projection to $Z$. Assume that $\Gamma$ is of relative dimension $c+r$ over $W$. Then
    \[
    \Gamma_*(H_{2d-k-2c}(Z, A)) \subset \tilde{N}_{c+r, \cyl}H^k(X, A) \subset \tilde{N}_{c, \cyl}H^k(X, A).
    \]
    As a corollary,
    \[
    \tilde{N}_{c, \cyl}H^k(X, A)=\sum_{\Gamma \in \text{CH}_{\dim Z+c}(Z \times X)}\Gamma_*(H_{2d-k-2c}(Z, A)).
    \]
\end{cor}
\begin{proof}
    The case for $A=\ZZ/m$ follows from $A=\ZZ/\ell^n$. So we assume that $A$ is one of the following: $\ZZ/\ell^n, \ZZ_\ell, \ZZ, \QQ$, or $\QQ_\ell$.
    
    We take either a resolution of singularities $W'\to W$ (if we work over the complex numbers and $A=\ZZ$ or $\QQ$) or an alteration of degree relatively prime to $\ell$ (in other cases) and find a smooth projective variety $\Gamma'$ with morphisms $\Gamma' \to W'\times X$ such that $\Gamma' \to W \times X$ has image equal to $\Gamma$ and $\Gamma' \to \Gamma$ is either birational (if we work over the complex numbers and $A=\ZZ$ or $\QQ$) or generically finite of degree relatively prime to $\ell$ (in other cases). 
    We see that
    \begin{align*}
        &\Gamma_*(H_{2d-k-2c}(Z))=\Gamma_*'(H_{2d-k-2c}(Z))\subset \Gamma_*'((H_{2d-k-2c-2r}(W')))\\
    \subset&\tilde{N}_{c+r, \cyl}H^k(X)\subset\tilde{N}_{c, \cyl}H^k(X),
    \end{align*}
    where the first equality is due to Lemma \ref{lem:image_corr}, and the last inclusion follows from Lemma \ref{cor:contain}.
\end{proof}

Strong cylindrical filtrations are generated by a restricted class of correspondences.
To see this, we need the following lemma, which will be used several times.

\begin{lem}[The Flattening Lemma]{\cite[Page 37]{Platification}, \cite[2.19]{deJong_alteration}}\label{lem:flattening}
    Let $f: X \to S$ be a projective, dominant morphism between noetherian schemes. 
    Assume that $S$ is integral.
    There exists a birational projective morphism $S' \to S$ such that the strict transform $f': X' \to S'$ is flat.
\end{lem}
Note that the cited reference \cite{deJong_alteration} does not explicitly state the projectivity of the birational morphism $S' \to S$.
But this can be deduced from the proof. 
See the first two paragraphs of page 37 in \cite{Platification}.
\begin{lem}\label{cyl_equiv}
    Keep the same notations as in Convention \ref{conv}. We have
    \[
    \tilde{N}_{c, \cyl}H^k(X, A)=\sum_{\Gamma \subset Z \times X}\Gamma_*(H_{2d-k-2c}(Z, A)), c\geq 1
    \]
    where the sum is taken over all smooth projective varieties $Z$ and all subvarieties $\Gamma \subset Z \times X$ such that $\Gamma \to Z$ is  flat (or has equi-dimensional fibers).
\end{lem}
\begin{proof}
As before, we only need to consider $A=\ZZ, \ZZ_\ell, \ZZ/\ell^n, \QQ, \QQ_\ell$. Also we only need to show that the left hand side is contained in the right hand side.

Consider a subvariety $\Gamma\subset Z \times X$ such that $\Gamma \to Z$ has relative dimension $c$.
By the flattening lemma \ref{lem:flattening}, there is a commutative diagram:
    \[
    \begin{CD}
        \Gamma_1 @>>> \Gamma\\
         @VVV @VVV\\
         Z_1 @>>> Z,
    \end{CD}
    \]
    where $Z_1 \to Z$ and $\Gamma_1 \to \Gamma$ are birational, and the morphism $\Gamma_1 \to Z_1$ is flat.
We then use resolution of singularities in characteristic $0$ or Gabber's refinement of de Jong's alteration and extend the previous commutative diagram:
    \[
    \begin{CD}
        \Gamma_2 @>>>\Gamma_1 @>>> \Gamma\\
        @VVV @VVV @VVV\\
        Z_2 @>>> Z_1 @>>> Z,
    \end{CD}
    \]
    where  $Z_2 \to Z_1$ is a resolution of singularities (in the case $A=\ZZ$ or $\QQ$) or an alteration of degree relatively prime to $\ell$ (in other cases), and the left square is Cartesian.
    Note that by construction,
    \[
    \Gamma_1 \subset Z_1 \times X, \Gamma_2 \subset Z_2 \times X,
    \]
    are closed subschemes. In particular, we may take their fundamental classes in the Chow group and get correspondences.
    Finally take an alteration $\Gamma_3 \to \Gamma_2$ of degree relatively prime to $\ell$.
    Then by Lemma \ref{lem:image_corr}
    \begin{align*}
        &\Gamma_*(H_{2d-c-k}(Z, A))= \Gamma_{3*}(H_{2d-c-k}(Z, A))\\
        \subset&\Gamma_{3*}(H_{2d-c-k}(Z_2, A))=\Gamma_{2*}(H_{2d-c-k}(Z_2, A)).
    \end{align*}
\end{proof}

A different way of defining the strong coniveau filtration is the following.
\begin{lem}\label{lem:str_coniveau}
    Keep the same notations as in Convention \ref{conv}. We have an equality
    \[
    \tilde{N}^cH^{k}(X, A)=\sum_{\Gamma \in \text{CH}_{\dim Z}(Z \times X), \dim Z \leq \dim X-c}\Gamma_*(H_{2\dim X-k}(Z, A)), c\geq 1.
    \]
\end{lem}
When $A=\ZZ$, this is \cite[Equation 2.1]{BO_Coniveau}. The general case only needs to replace the resolution of singularities in their argument by Gabber's refinement of de Jong's alteration.
\begin{proof}
    The left hand side is contained in the right hand side since we may take $\Gamma$ to be the graph of $Z \to X$.
To see the reverse inclusion, we may assume that $\Gamma$ is represented by an integral subvariety of $Z\times X$. Take a resolution of singularities $\Gamma' \to \Gamma$ (in case $A=\ZZ$ or $\QQ$) or an alteration $\Gamma' \to \Gamma$ whose degree is relatively prime to $\ell$ (in other cases). 
Then by Lemma \ref{lem:image_corr}, $\Gamma_*(H_{2\dim X-k}(Z, A))$ is contained in the push-forward of $H_{2\dim X-k}(\Gamma', A)$ via the morphism $\Gamma' \to X$.
\end{proof}

\begin{cor}\label{cor:str_co}
     Keep the same notations as in Convention \ref{conv}.
     Let $$\Gamma \in \text{CH}_{\dim X}(X \times Y)$$ be a correspondence between smooth projective varieties $X, Y$. Assume that $\dim Y=\dim X+r$. Then for any $c\geq 1$,
     \[
     \Gamma_*(\tilde{N}^cH^k(X, A))\subset \tilde{N}^{c+r}H^{2r+k}(Y, A),\quad \Gamma_*(\tilde{N}_{c, \cyl}H^k(X, A))\subset \tilde{N}_{c, \cyl}H^{2r+k}(Y, A).
     \]
\end{cor}
\begin{proof}
    
    By Lemma \ref{lem:str_coniveau}, $\tilde{N}^cH^{k}(X, A)$ is generated by $\gamma_*(H_{2\dim X-k}(Z, A))$ for all $\gamma \in \text{CH}_{\dim Z}(Z \times X), \dim Z \leq \dim X-c$.
    Then \[
    \Gamma_*(\gamma_*(H_{2\dim X-k}(Z, A)))=(\Gamma \circ \gamma)_*(H_{2\dim X-k}(Z, A)),\]
    where $\Gamma\circ\gamma \in \text{CH}_{\dim Z}(Z \times Y)$ is the composition of the correspondences $\Gamma$ and $\gamma$.

    The proof for strong cylindrical filtrations is the same by Corollary \ref{cor:contain}.
\end{proof}

\begin{lem}[{\cite[Lemma 2.3]{ScaviaSuzuki2023coniveau}}]\label{lem:fil_corr}
    Keep the same notations as in Convention \ref{conv}. The action of any correspondence $\Gamma \in \text{CH}_{\dim X}(X \times Y)$ on cohomology preserves the coniveau filtration.
\end{lem}
\begin{proof}
    This is essentially {\cite[Lemma 2.3]{ScaviaSuzuki2023coniveau}}, where the authors prove the statement for \'etale cohomology with $\ZZ_\ell$-coefficients. But the same proof works for other cohomology groups as well. The key point is that one can use a moving lemma to replace the correspondence by a rationally equivalent cycle that meets the support of the cohomology class properly.
\end{proof}

A very important observation, first made in \cite{VoisinConiveau} for complex varieties with integral singular cohomology, is the following.
\begin{lem}\cite[Proposition 1.3]{VoisinConiveau} \cite[Lemma 2.4]{ScaviaSuzuki2023coniveau}\label{lem:inclusion}
Keep the same notations as in Convention \ref{conv}.
For $A=\ZZ, \ZZ_\ell$ and $X$ of dimension $d$, we have the following equality:
\[
 \tilde{N}_{d-c, \cyl}H^{2c-1}(X, A) = \tilde{N}^{c-1}H^{2c-1}(X, A), c\geq 1.
\]
\end{lem}

A natural question is whether or not the filtrations in Definition \ref{def:filtration} agree with each other. 
If we use $\QQ$ or $\QQ_\ell$ coefficients, the theory of weights shows that the strong coniveau and coniveau filtrations are the same.
Since the difference between some of these filtrations also gives stable birational invariants, one wonders if this could be used to prove non-stable-rationality for some rationally connected varieties.

Examples over $\CC$ with $\ZZ$-coefficients where the strong coniveau filtration and coniveau filtration differ are constructed in \cite{BO_Coniveau} for every possible $c, k$ (coniveau $c$ on degree $k$ cohomology). Examples in positive characteristic are constructed by Scavia-Suzuki \cite{ScaviaSuzuki_examples} and Benoist \cite{Benoist2022Steenrod}.
These examples all have dimension at least $4$.

Voisin studied the (strong) coniveau filtrations on $H^{2d-3}$ \cite{VoisinConiveau} on complex projective manifolds.
For the definition of $N_{1, \cyl, \st}H^{2d-3}(X, \ZZ)$, see \cite[Definition 1.12]{VoisinConiveau}.
\begin{thm}\cite[Theorem 2.6, Corollary 2.7, Theorem 2.17]{VoisinConiveau}
Let $X$ be a smooth projective variety of dimension $d$ defined over $\CC$.
\begin{enumerate}
    \item Assume the Walker Abel-Jacobi map (\cite{Walker_Morphic_AJ})
    \[
    \phi: \text{CH}_1(X)_{\text{alg}} \to J(N^{d-2}H^{2d-3}(X, \ZZ))
    \]is injective on torsion. Then we have
    \[
    N_{1, \cyl, \st}H^{2d-3}(X, \ZZ)/\text{Tor} = N^{d-2}H^{2d-3}(X, \ZZ)/\text{Tor}.
    \]
    \item If $\dim X$ is $3$, we have
    \[
    N_{1, \cyl, \st}H^{3}(X, \ZZ)/\text{Tor} = N^1H^{3}(X, \ZZ)/\text{Tor}.
    \] 
    \item If $X$ is rationally connected, we have
    \[
    N_{1, \cyl, \st}H^{2d-3}(X, \ZZ) = \tilde{N}^{d-2}H^{2d-3}(X, \ZZ).
    \]
    As a consequence,
    \[
    N_{1, \cyl, \st}H^{2d-3}(X, \ZZ)/\text{Tor}=\tilde{N}^{d-2}H^{2d-3}(X, \ZZ)/\text{Tor} = N^{d-2}H^{2d-3}(X, \ZZ)/\text{Tor}.
    \]
\end{enumerate}
\end{thm}

We will prove an improvement of Voisin's result in Section \ref{sec:complex} and the $\ZZ_\ell$-counterpart in Section \ref{sec:general}, which include all the torsion classes.

Finally we note the following.
\begin{lem}\label{cor:2d-3torsion}
     Let $X$ be a smooth projective variety of dimension $d$ defined over an algebraically closed field $k$, whose rational Chow group of  zero-cycles is universally supported on a surface. 
     Fix $\ell$ a prime number distinct from the characteristic of $k$.
     Then the quotients $$H^{2d-3}(X, \ZZ_\ell)/N^{d-2}H^{2d-3}(X, \ZZ_\ell),\quad H^{2d-3}(X, \ZZ_\ell)/\tilde{N}^{d-2}H^{2d-3}(X, \ZZ_\ell)$$ are torsion.   
     In particular, this is true for smooth projective varieties that are separably rationally connected in codimension one.
\end{lem}

\begin{proof}
Since we have inclusions $$\tilde{N}^{d-2}H^{2d-3}(X, \ZZ_\ell) \subset {N}^{d-2}H^{2d-3}(X, \ZZ_\ell)\subset H^{2d-3}(X, \ZZ_\ell),$$
it suffices to show that $H^{2d-3}(X, \ZZ_\ell)/\tilde{N}^{d-2}H^{2d-3}(X, \ZZ_\ell)$ is torsion.

    By Lemma \ref{lem:dd}, there is a decomposition of diagonal of the form
    \[
    N \Delta_X=\Gamma_1+\Gamma_2 \in \text{CH}_{\dim X}(X \times X),
    \]
    where $N$ is an integer, $\Gamma_1$ is a cycle supported in $X \times Z$ for some $2$-dimensional closed subset $Z$, and $\Gamma_2$ is a cycle supported in $D\times X$ for some divisor $D \subset X$.

    By Corollary \ref{cor:str_co},  correspondences preserve the strong coniveau filtration. So it suffices to show that $\Gamma_1, \Gamma_2$ acts by zero on the quotient
    \[
    H^{2d-3}(X, \ZZ_\ell)/\tilde{N}^{d-2}H^{2d-3}(X, \ZZ_\ell).
    \]
    
    Since the correspondence action is linear, we may assume that $\Gamma_1$ is the class of an irreducible subvariety of $X\times X$, still denoted by $\Gamma_1$, and that $Z$ is also irreducible.
    Take alterations $\Gamma_1', Z'$ of $\Gamma_1$ and $Z$ of degree relatively prime to $\ell$ such that there is a commutative diagram
    \[
    \begin{CD}
       \Gamma_1' @>>> \Gamma_1\\
       @VVV @VVV\\
       Z' @>>>Z.
    \end{CD}
     \] Denote by $f: Z' \to Z \subset X$ the morphism to $X$. 
    Then by Lemma \ref{lem:image_corr},
    \[
    \Gamma_{1*}(H^{2d-3}(X, \ZZ_\ell))=\Gamma_{1*}'(H^{2d-3}(X, \ZZ_\ell))\subset f_*(H_3(Z', \ZZ_\ell))\subset \tilde{N}^{d-2}H^{2d-3}(X, \ZZ_\ell).\]
    So $\Gamma_{1*}$ acts by zero on $H^{2d-3}(X, \ZZ_\ell)/\tilde{N}^{d-2}H^{2d-3}(X, \ZZ_\ell)$. 
    
    By Corollary \ref{cor:contain} and Lemma \ref{lem:inclusion},
    \[
    \Gamma_{2*}(H^{2d-3}(X, \ZZ_\ell))\subset \tilde{N}_{1, \cyl}H^{2d-3}(X, \ZZ_\ell) = \tilde{N}^{d-2}H^{2d-3}(X, \ZZ_\ell).
    \]
    Thus $\Gamma_{2*}$ also acts by zero on $H^{2d-3}(X, \ZZ_\ell)/\tilde{N}^{d-2}H^{2d-3}(X, \ZZ_\ell)$. 
\end{proof}

\subsection{Some stable birational invariants}\label{sec:bir}
We include the proof of the following.
\begin{lem}\label{lem:birational}
    Let $X$ be a smooth projective variety of dimension $d$ defined over an algebraically closed field. Fix a prime number $\ell$ different from the field characteristic. All cohomology groups considered are \'etale cohomology groups. The following invariants are stable birational invariants of $X$.
    \begin{enumerate}
        \item The kernel and cokernel of the cycle class map 
        \[
        A_1(X)\otimes \ZZ_\ell \to H^{2d-2}(X, \ZZ_\ell(d-1)).
        \]
        
        \item The kernel and cokernel of the cycle class maps
        \begin{equation}\label{eq:highercyc}
            \lim \limits_{\xleftarrow[n]{}}\text{CH}_1({X}, r, \ZZ/\ell^n \ZZ) \to H^{2d-2-r}({X}, \ZZ_\ell(d-1)).
        \end{equation}

        \item The quotients $$H_{i}(X, \ZZ_\ell)/\tilde{N}_{1, \cyl}H_{i}(X, \ZZ_\ell)=H^{2d-i}(X, \ZZ_\ell)/\tilde{N}_{1, \cyl}H^{2d-i}(X, \ZZ_\ell),$$ $$H^{2d-3}(X, \ZZ_\ell)/\tilde{N}^{d-2}H^{2d-3}(X, \ZZ_\ell),$$
        and 
        $$H^{2d-3}(X, \ZZ_\ell)/{N}^{d-2}H^{2d-3}(X, \ZZ_\ell).$$

    \end{enumerate}
\end{lem}
 A word of caution is in order. These invariants are homological and covariant for proper morphisms. 
 One should be particularly careful when comparing the invariants $H^{2d-i}(*, \ZZ_\ell)/\tilde{N}_{1, \cyl}H^{2d-i}(*, \ZZ_\ell)$ for $X\times \PP^n$ and $X$. The invariance means that the push-forward (instead of pull-back) along the first factor induces an isomorphism between $H^{2d-i+2n}(X\times \PP^n, \ZZ_\ell)/\tilde{N}_{1, \cyl}H^{2d-i+2n}(X\times \PP^n, \ZZ_\ell)$ and $H^{2d-i}(X, \ZZ_\ell)/\tilde{N}_{1, \cyl}H^{2d-i}(X, \ZZ_\ell)$.

This lemma is well-known to experts. Over the complex numbers with singular cohomology, the first statement can be found in \cite{Voisin_SB_Survey} as Lemma 2.22 and Lemma 2.23, and the third in \cite[Proposition 1.4]{VoisinConiveau}. The proofs in the references use the weak factorization theorem to reduce to a smooth blow-up case, and thus do not immediately apply in positive characteristic.
The following proof using correspondences is also well-known to experts. 
\begin{proof}
    We first prove that these are birational invariants.
    
    Let $Y$ be a smooth projective variety birational to $X$, with birational maps $f:X \dashrightarrow Y, g: Y \dashrightarrow X$ that are inverse to each other. Denote by $\Gamma_f \subset X \times Y, \Gamma_g \subset Y \times X$ the closure of the graph of $f, g$. Finally, denote by $\Gamma_X=\Gamma_g \circ \Gamma_f$ the composition of correspondences.
    
    First observe that any correspondence $\Gamma \in \text{CH}_{\dim X}(X \times Y)$ acts on higher Chow groups and cohomology, and preserves $\tilde{N}_{1, \cyl}H^i(X, \ZZ_\ell)$ (Corollary \ref{cor:str_co}) and ${N}^{d-2}H^{2d-3}(X, \ZZ_\ell)$ (Lemma \ref{lem:fil_corr}). 
     Moreover, the action on higher Chow groups and cohomology is compatible with cycle class maps by Proposition \ref{class_cor}.
     Therefore, correspondences act on all the invariants that we consider.
     
    So it suffices to show that $\Gamma_X$ acts by identity on the invariants. Since then by symmetry, $\Gamma_Y=\Gamma_f \circ \Gamma_g$ also acts by identity, and thus both $\Gamma_f$ and $\Gamma_g$ are isomorphisms.
    
    Since $f, g$ are inverse birational maps, there is an open subset $U \subset X$ such that $g\circ f: U \to X$ is the inclusion. Therefore, when restricted to $U \times X$, the cycle $\Gamma_X$ is the same as the graph of the inclusion of $U \subset X$. By the localization sequence for Chow groups, we know that 
    \[
    \Gamma_X=\Delta_X+Z \in \text{CH}_{\dim X}(X\times X)
    \]
    for some cycle $Z$ supported in $D \times X$, where $\Delta_X$ is the diagonal, and $D\subset X$ is a divisor.
    Since $\Delta_X$ acts as identity, we only need to show that $Z$ acts as the zero map on the invariants.

    Since the correspondence action is linear, we may assume that $Z$ and $D$ are irreducible. By Gabber's refinement of de Jong's alteration, there are smooth projective varieties $Z', D'$ with the following commutative diagram
    \[
    \begin{CD}
    Z'@>>>Z\\
    @V p VV @VVV\\
    D'@>>>D\\
    \end{CD}
    \]
    such that $Z' \to Z$ and $D' \to D$ are generically finite with degrees relatively prime to $\ell$. 
    
    By Lemma \ref{lem:image_corr}, up to an invertible factor, the action of $Z$ on $A_1(X) \otimes \ZZ_\ell$  agrees with
    \[
    A_1(X)\otimes \ZZ_\ell \to A_0(D')\otimes \ZZ_\ell \xrightarrow{p^*} A_1(Z')\otimes \ZZ_\ell \to A_1(X)\otimes \ZZ_\ell,
    \]
    where the first two maps are pull-backs and the last one is push-forward.
    Similarly, up to an invertible factor, the action on $H^{2d-2}(X, \ZZ_\ell)$ agrees with
    \[
    H^{2d-2}(X, \ZZ_\ell)\to H^{2d-2}(D', \ZZ_\ell) \to H^{2d-2}(Z', \ZZ_\ell) \to H^{2d-2}(X, \ZZ_\ell),
    \]
    where the first two maps are pull-backs and the last one is Gysin push-forward.
    Moreover, there is an obvious commutative diagram of these maps.
    In particular, up to multiplying by an integer relatively prime to $\ell$, the action of $Z$ on the kernel (resp. cokernel) of 
    \[
    A_1(X)\otimes \ZZ_\ell \to H^{2d-2}(X, \ZZ_\ell(d-1)),
    \] 
    factors through the kernel (resp. cokernel) of 
    \[
    A_0(D')\otimes \ZZ_\ell \xrightarrow{cl} H^{2d-2}(D', \ZZ_\ell(d-1)),
    \]
    Since the cycle class map for zero-cycles on $D'$ modulo algebraic equivalence induces 
 an isomorphism in the above cycle class map to cohomology, we see that $Z$ acts as zero.

    The case for the kernel and cokernel of the cycle class map (\ref{eq:highercyc}) is similar. We use the fact that for $D'$, the higher cycle class map induces an isomorphism (Corollary \ref{cor:ch_BM}).
    \[
    CH_0(D', r, \ZZ/\ell^n\ZZ) \xrightarrow[\cong]{cl_{0, j}} H^{2d-2-r}(D', \ZZ/{\ell^n}(d-1))
    \]
    and hence we have an isomorphism after taking inverse limit.

    For the third statement, notice  that $Z_*(H_i(X, \ZZ_\ell))$
    lies in the $\tilde{N}_{1, \cyl}H_i$ part of the strong cylindrical filtration by Corollary \ref{cor:contain}. Thus $Z$ acts by zero on $H_i(X, \ZZ_\ell)/\tilde{N}_{1, \cyl}H_i(X, \ZZ_\ell)$.
    
    Finally, by Lemma \ref{lem:inclusion}, and the fact that $Z_*(H_i(X, \ZZ_\ell))$
    lies in the $\tilde{N}_{1, \cyl}H_i$ part of the strong cylindrical filtration by Corollary \ref{cor:contain}, the correspondence $Z$ acts as $0$ on $H^{2d-3}(X, \ZZ_\ell)/\tilde{N}^{2d-3}H^{2d-3}(X, \ZZ_\ell), H^{2d-3}(X, \ZZ_\ell)/{N}^{2d-3}H^{2d-3}(X, \ZZ_\ell)$.

    So far we have shown that these invariants are birational invariants.
    To prove that these are stable birational invariants, it then suffices to show that for $X$ and $X \times \PP^1$, these invariants are the same.

    The first two statements  follow from the projective bundle formulae for higher Chow groups and \'etale cohomology. Indeed, we have
    \[
    H^{i}(X \times \PP^1)= H^{i-2}(X) \cdot h \oplus H^{i}(X)\cdot 1,
    \]
    \[
    A_1(X \times \PP^1)= A_1(X) \cdot h \oplus A_{0}(X)\cdot 1,
    \]
    \[
    \text{CH}_1(X \times \PP^1, r, \ZZ/m)= \text{CH}_1(X, r, \ZZ/m) \cdot h \oplus CH_0(X, r, \ZZ/m) \cdot 1,
    \]
    where $h$ denotes the class of a hyperplane in cohomology or Chow group.
    In the case of cohomology, the same type of formulae holds for torsion coefficients and $\ell$-adic cohomology. We also use the fact (Corollary \ref{cor:ch_BM}) that
    \[
    CH_0(X, r, \ZZ/m)\xrightarrow{\cong} H^{2d-r}(X, \ZZ/m),
    \]
    and 
    \[
    A_0(X)\otimes \ZZ_\ell \xrightarrow{\cong} H^{2d}(X, \ZZ_\ell).
    \]
    
    For the various filtrations, simply note that the subgroup $H^{i}(X, \ZZ_\ell) \otimes 1$ is in $\tilde{N}_{1, \cyl}H^i(X \times \PP^1, \ZZ_\ell)$, since it lies in the image of \[
    H^{i}(X, \ZZ_\ell)  \xrightarrow{\pi_X^*} H^{i}(X\times \PP^1, \ZZ_\ell)\xrightarrow{id_*} H^i(X \times \PP^1, \ZZ_\ell).\]
    where we consider $X \times \PP^1$ as a correspondence between $X$ and $X \times \PP^1$ via the projection to the first factor and the identity map. 
\end{proof}

\section{Space of one-cycles}\label{sec:space}
In this section, we fix an algebraically closed field $k$ of any characteristic.
We remind the readers that a variety is always assumed to be irreducible,  and in particular, connected, throughout the paper.
Sometimes we add the word irreducible just to emphasize this.
\subsection{Setup and statements}
\begin{defn}[{\cite[Definition 3.1.3, Lemma 3.3.9]{SV_Chow_Sheaves}}]\label{def:RelCycle}
Let $X, S$ be finite type reduced separated $k$-schemes. An equi-dimensional family of relative cycles of dimension $r$ (or simply $r$-cycles) over $S$ is a formal linear combination of integral subschemes $\Gamma_S=\sum m_i Z_i, Z_i \subset S \times X, m_i \in \ZZ$, such that 
\begin{enumerate}
\item Each $Z_i$ dominates one irreducible component of $S$
\item Each non-empty fiber of $Z_i \to S$ has pure dimension $r$.
\item A fat point condition is satisfied (\cite[Definition 3.1.3]{SV_Chow_Sheaves}).
\item A field of definition condition is satisfied (\cite[Lemma 3.3.9 and the ensuing paragraph]{SV_Chow_Sheaves}). 
\end{enumerate} 
Each $Z_i$ is called a \emph{component} of $\Gamma_S$.
We write $\Gamma_{S}^+$ and $\Gamma_{S}^-$ as the \emph{positive and negative parts}, i.e. the sum of $Z_i$'s with positive and negative coefficients. 
\end{defn}

Suslin and Voevodsky consider a more general setup in \cite{SV_Chow_Sheaves}: a finite type scheme $X$ over a noetherian scheme $S$, which could be non-reduced. We only consider families of cycles over a reduced base. To apply their definitions and results to our situation, one considers $X \times S \to S$.

\begin{rem}
For a normal variety $S$ (or more generally, a geometrically unibranch variety), condition $3$ is automatic \cite[Corollary 3.4.3, 3.4.4]{SV_Chow_Sheaves}. Condition $4$ is always satisfied if the base is regular \cite[Corollary 3.4.5, 3.4.6]{SV_Chow_Sheaves}. 
\end{rem}

We refer the interested readers to \cite{SV_Chow_Sheaves} Section 3 for details about the last two  conditions in Definition \ref{def:RelCycle} and the subtle points about these conditions. We only remark that with this definition, one can pull-back families of cycles, that is, we have a presheaf on the category of schemes (\cite[paragraph before Lemma 3.3.10]{SV_Chow_Sheaves}).

\begin{notat}
    Given an equi-dimensional family of cycles $\Gamma_S \to S$ and a morphism $f: T \to S$, we denote by $\Gamma_T=f^*\Gamma_S \to T$, the pull-back family of cycles over $T$. 
If $f$ is a locally closed immersion, we also use the notation $\Gamma_T=\Gamma_S|_T$ to denote the pull-back family.
\end{notat}

The following is well-known. 
\begin{lem}\label{lem:hilb_Chow}
    Given a flat projective family of curves $C \to S$ and an $S$-morphism $F: C\to S\times X$, we can associate to it an equi-dimensional family of one-cycles of $X$ over $S$. 
\end{lem}
\begin{proof}
    We take the graph $C \subset C\times_S(S \times X)\cong C\times X$.
    The subscheme $C \subset C \times X$ is flat and proper over $S$. Applying \cite[Corollary 3.2.5]{SV_Chow_Sheaves}, we get a family of one-dimensional cycles in $C \times X/S$ over the base $S$. Then pushing-forward via the projection to $S\times X$ gives an equi-dimensional family of one-cycles of $X$ over $S$ (\cite[Corollary 3.6.3]{SV_Chow_Sheaves}).
\end{proof}
\begin{defn}\label{def:nodal}
\emph{A nodal family of one-cycles over $S$}, denoted by $S \leftarrow\Gamma_S \to X, \Gamma_S =\sum m_i Z_i (m_i\in \ZZ)$, is a finite $\ZZ$-linear combination of isomorphism classes of $S \leftarrow Z_i \to X$, where $S \leftarrow Z_i \to X$ is a flat projective family of connected nodal curves over $S$ with a morphism to $X$.
If $S=\SP k$, we also call $\Gamma_S$ a \emph{nodal one-cycle}.

Each $Z_i$ is called a \emph{component} of $\Gamma_S$.
We use $[\Gamma_S]\to S$ to denote the associated equi-dimensional family of one-cycles over $S$ as in Lemma \ref{lem:hilb_Chow}.

We write $\Gamma_S^+$ and $\Gamma_S^-$ as the \emph{positive and negative parts}, i.e. the sum of $Z_i$'s with positive and negative coefficients. 
We say that $\Gamma_S$ is \emph{effective} if $\Gamma_S^-$ is $0$.
\end{defn}
Two different nodal families may give the same family of cycles.
This already happens for nodal curves over the field $k$,
and it is one of the main difficulties we have to deal with in the following.

\begin{notat}
    Given a nodal family of one-cycles $S \leftarrow\Gamma_S \to X, \Gamma_S =\sum m_i Z_i$, and a morphism $f: T \to S$, we denote by $\Gamma_T=f^*\Gamma_S \to T$, the pull-back nodal family over $T$, where $f^*\Gamma_S=\sum m_i (Z_i \times_S T)$. 
If $f$ is a locally closed immersion, we also use the notation $\Gamma_T=\Gamma_S|_T$ to denote the pull-back nodal family.
\end{notat}

\begin{rem}
    Given a nodal family of one-cycles $S \leftarrow\Gamma_S \to X$ and a morphism $f: T \to S$, by \cite[Lemma 3.3.10]{SV_Chow_Sheaves} and the definition of associated families of cycles,  we have $[f^*\Gamma_S]=f^*[\Gamma_S]$ as equi-dimensional families of one-cycles over $T$. 
\end{rem}

\begin{defn}[relative projective extension of a nodal family]\label{def:extension}
    Let $f_T: T \to S$ be a morphism between quasi-projective algebraic sets over a field $k$,
    and let $f_V: V\to S$ be a generically finite projective surjective $S$-morphism from a quasi-projective $k$-algebraic set $V$.
    Assume that there is a dense open subset $V^0 \subset V$ which admits a $k$-morphism $g: V^0 \to T$.
    Given a nodal family $T \leftarrow \Gamma_T=\sum_i m_i Z_i \to X$,
    we define a \emph{projective extension to $f_V: V \to S$ (relative to $S$)}, or just \emph{projective extension to $V$} if $S=\SP k$, of the family $\Gamma_T$ to be a nodal family of one-cycles $V \leftarrow \Gamma_V \to X$ such that $\Gamma_V|_{V^0}=g^*\Gamma_T$.
    
    If $T$ is pointed, we require that $V^0$ is pointed, and that $g: V^0 \to T$ is a morphism of pointed schemes.
\end{defn}
By semi-stable reduction \cite[Theorem 5.8]{deJong_alteration}, given $f_T: T \to S$ and a nodal family $T \leftarrow \Gamma_T \to X$, there always exists a projective extension (relative to $S$) of $\Gamma_T$ to some $f_V: V \to S$.

Now we describe some of the basic structures of the space of one-cycles.
\begin{prop}\label{prop:filter_1}
Let $S_1, S_2$ be two connected algebraic sets and $\Gamma_1 \subset S_1 \times X, \Gamma_2 \subset S_2 \times X$ be two equi-dimensional families of $r$-cycles in $X$.
Assume that the cycles in the two families are algebraically equivalent (Definition \ref{def:alg_equiv}).
There is a connected algebraic set $S$ and an equi-dimensional family of $r$-cycles $\Gamma \subset S \times X$, and morphisms $f_1: S_1 \to S, f_2: S_2\to S$ such that $\Gamma_1=f_1^* \Gamma, \Gamma_2=f_2^* \Gamma$. Moreover, if both $S_1$ and $S_2$ are normal/smooth/projective, we may choose $S$ to be normal/smooth/projective.
\end{prop}

\begin{proof}
Take two points $s_1 \in S_2, s_2 \in S_2$. Denote by $\gamma_1, \gamma_2$ the cycle over $s_1, s_2$.
By assumption, $\gamma_1, \gamma_2$ are algebraically equivalent.
Thus there is a smooth projective curve $S_3$, two points $a, b \in S_3$ and a family of $r$-cycles $\Gamma_3 \subset S_3 \times X$ such that the cycle over $a, b$ are $\gamma_1, \gamma_2$.
This follows from the definition of algebraic equivalence. Indeed, by definition of algebraic equivalence, we may find a cycle $\Delta$ and a family of cycles $\Gamma_T$ over a smooth projective curve $T$ and two points $t_1, t_2 \in T$ such that the cycle over $t_1$ (resp. $t_2$) is $\gamma_1+\Delta$ (resp. $\gamma_2+\Delta$) (\cite[Example 10.3.2]{Fulton98}). Then we take $S_3$ to be $T$ and the family of cycles $\Gamma_3$ to be $\Gamma_T-T \times \Delta$.

Define $S=S_1 \times S_2 \times S_3$, with $p_i: S \to S_i, i=1, 2, 3,$ the projections. Define $\Gamma=p_1^*\Gamma_1+p_2^*\Gamma_2-p_3^*\Gamma_3$.
Finally, define 
\[
f_1: S_1 \to S, x \mapsto (x, s_2, b) 
\]
and \[
f_2: S_2 \to S, y \mapsto (s_1, y, a).
\]
It is straightforward to verify that these are the desired morphisms.

Since $S_3$ is a smooth projective curve, if both $S_1, S_2$ are normal, or smooth, or projective, so is $S$.
\end{proof}

Now we can state the main technical results of this section.

\begin{thm}\label{thm:filtered_sm}
Let $X$ be a smooth projective variety defined over an algebraically closed field $k$.
Let $(U, \Gamma_U)$ be an equi-dimensional family of one-cycles over an irreducible variety $U$ and let $u_0, u_1 \in U$ be two points in $U$ such that 
\[
\Gamma_U|_{u_0}= \Gamma_U|_{u_1}=\gamma
\]
as cycles.
Then there is
\begin{enumerate}
    \item a nodal family of one-cycles $V\leftarrow \Gamma_V \to X$ over a normal quasi-projective variety  $V$,  with a morphism $f: V \to U$, and a lifting $v_0, v_1 \in V$ of the points $u_0, u_1 \in U$;
    \item a nodal family of one-cycles $T \leftarrow \Gamma_T \to X$ over a two pointed connected projective curve $(T, t_0, t_1)$, where $t_0, t_1$ are smooth points of $T$;
    \item\label{cond:3} an open subset $U^0\subset U$ such that for any chosen point $u$ in $U^0$, for any pair of points $v, v' \in f^{-1}(u)\subset V$, a nodal family of one-cycles $$(T_{v, v'}, t_{v, v'}^v, t_{v, v'}^{v'})\leftarrow \Gamma_{v, v'} \to X, $$ over a connected two pointed projective curve $(T_{v, v'}, t_{v, v'}^v, t_{v, v'}^{v'})$ where $t_{v, v'}^v$, $t_{v, v'}^{v'}$ are smooth points of $T_{v, v'}$;
    \item  after choosing a point $u \in U^0$, a nodal family of one-cycles $W\leftarrow \Gamma_W \to X$ over a connected projective algebraic set  $W$ (depending on $u$),
\end{enumerate} 
such that 
\begin{enumerate}
    \item[(i)] The morphism $f: V \to U$ is generically finite, surjective and projective.
    
    \item[(ii)]  $f^*\Gamma_U=[\Gamma_V]$ as equi-dimensional family of one-cycles over $V$.
    
    \item[(iii)] There is a morphism
    \[
    \Psi: V\coprod T\coprod (\coprod_{v, v' \in f^{-1}(u)}T_{v, v'})\to W,
    \]
    such that $\Gamma_V \coprod \Gamma_T \coprod (\coprod_{v, v'}\Gamma_{v, v'})=\Psi^*\Gamma_W$.
    
    \item[(iv)] The morphism $\Psi$ maps the pair of points $(v_0, t_0)$ (resp. $(v_1, t_1)$, $(v, t_{v, v'}^v)$, $(v', t_{v, v'}^{v'})$) to the same point in $W$. 
    
    \item[(v)]  The family of cycles over $[\Gamma_T]\to T$ is constant. So necessarily, for each point $t$ of $T$, the cycle class $[\Gamma_T|_{t}]$ over $t$ is $\gamma$.

\item[(vi)] The family of cycles $[\Gamma_{v, v'}] \to T_{v, v'}$ is constant. So necessarily, for each point $t \in T_{v, v'}$, the cycle class $[\Gamma_{v, v'}|_t]$ over $t$ is the same as the cycle $\Gamma_U|_u$.
\end{enumerate}
\end{thm}

\begin{thm}\label{thm:filtered}
Keep the same assumptions as in Theorem \ref{thm:filtered_sm}. 
Also assume that $X$ is separably rationally connected in codimension one.
Then there is
\begin{enumerate}
    \item an equi-dimensional family of one-cycles $V\leftarrow \Gamma_V \to X$ over a normal quasi-projective variety  $V$,  with a morphism $f: V \to U$, and a lifting $v_0, v_1 \in V$ of the points $u_0, u_1 \in U$;
    \item an equi-dimensional family of one-cycles $T \leftarrow \Gamma_T \to X$ over a two pointed connected projective curve $(T, t_0, t_1)$, where $t_0, t_1$ are smooth points of $T$;
    \item\label{cond:3sm} an open subset $U^0\subset U$ such that for any chosen point $u$ in $U^0$, for any pair of points $v, v'$ in the inverse image of $u$ in $V$, an equi-dimensional family of one-cycles $$(T_{v, v'}, t_{v, v'}^v, t_{v, v'}^{v'})\leftarrow \Gamma_{v, v'} \to X, $$ over a connected two pointed projective curve $(T_{v, v'}, t_{v, v'}^v, t_{v, v'}^{v'})$, where $t_{v, v'}^v$, $t_{v, v'}^{v'}$ are smooth points of $T_{v, v'}$;
    \item after choosing a point $u \in U^0$, an equi-dimensional family of one-cycles $W\leftarrow \Gamma_W \to X$ over a normal projective variety $W$ (depending on $u \in U^0$),
\end{enumerate} 
such that 
\begin{enumerate}
    \item[(i)] The morphism $f: V \to U$ is generically finite, surjective and projective.
    
    \item[(ii)]  $f^*\Gamma_U=\Gamma_V$ as equi-dimensional family of one-cycles over $V$.
    
    \item[(iii)] There is a morphism
    \[
    \Psi: V\coprod T\coprod (\coprod_{v, v' \in f^{-1}(u)}T_{v, v'})\to W,
    \]
    such that $\Gamma_V \coprod \Gamma_T \coprod (\coprod_{v, v'}\Gamma_{v, v'})=\Psi^*\Gamma_W$.
    
    \item[(iv)] The morphism $\Psi$ maps the pair of points $(v_0, t_0)$ (resp. $(v_1, t_1)$, $(v, t_{v, v'}^v)$, $(v', t_{v, v'}^{v'})$) to the same point in the smooth locus of $W$. It also maps $v_0$, $v_1$, and every singular point in $T, T_{v, v'}$ to the smooth locus of $W$. 
    
    \item[(v)]  The family of cycles $\Gamma_T\to T$ is constant. So necessarily, for each point $t$ of $T$,  $\Gamma_T|_{t}=\gamma$.

\item[(vi)] The family of cycles $\Gamma_{v, v'}\to T_{v, v'}$ is constant. So necessarily, for each point $t \in T_{v, v'}$, $\Gamma_{v, v'}|_t=\Gamma_U|_u$.
\end{enumerate}
In characteristic $0$, we may take $W$ to be smooth and projective.
\end{thm}

We first make some remarks and give an example in the following before giving the proof.

\begin{rem}
Theorem \ref{thm:filtered} is special to one-cycles on varieties that are SRC in codimension one.
Indeed, if the statements were true for a variety $X$ and families of $r$-dimensional cycles, then the same argument as in Sections \ref{sec:complex} (resp. Section  \ref{sec:general}) would prove that $$\tilde{N}^{d-r-1}H^{2d-2r-1}(X, \ZZ)=N^{d-r-1}H^{2d-2r-1}(X, \ZZ)$$ (resp. the corresponding equality of \'etale cohomology groups with $\ZZ_\ell$-coefficients). But the examples of \cite{BO_Coniveau}  show that this is not true in general.
\end{rem}

\begin{rem}
Even if we start with a family of effective cycles, for the statements to be true, we have to use non-effective cycles.
In the proof, we have to add families of curves to the original family to achieve the desired deformation over the curve $T$. 
 Example \ref{ex} below shows that adding more curves is necessary.

Moreover, the statement is not a simple corollary of the existence of the universal family over the Chow variety (which exists only in characteristic $0$).
This is because we require that the family is parameterized by a normal variety, while the Chow variety is only semi-normal in \cite{Kollar96} by definition or satisfies no such normality condition at all in some other references such as \cite{Friedlander_LawsonHomology}. 

It is possible that the morphism from the normalization of the Chow variety to the Chow variety maps two points to the same point. 
If this happens, we take $U$ to be the normalization of the Chow variety, $u, u'$ to be the two points mapping to the same point in the Chow variety, the existence of $V, T, W$ in this case cannot be deduced from the existence of the universal family over the Chow variety.
\end{rem}

\begin{rem}
    Finally we remark that $U$ being irreducible is not essential in the proof. But it simplifies the argument. If $U$ is reducible and connected, one can use similar argument as in \cite[Section 8]{Kollar_Tian} to find a connected algebraic set $V$. But in this case, we cannot choose $V$ to admit a morphism to $U$. The best one can have is that for each irreducible component of $U$, there is an irreducible component of $V$ with a projective, surjective morphism to this component. 
\end{rem}

\begin{ex}\label{ex}
Take an irreducible nodal curve $C \subset \PP^3$ of degree $d$ and arithmetic genus $g$.
We assume that $d>2g-1$.
Assume that there are $n$ nodes for some $n>1$.
Let $C_i \to C$ be the normalization at one of the nodes for $i=1,\ldots, n$. In this way, we get $n$ morphisms from $n$ different nodal curves (with $(n-1)$-nodes) to $\PP^3$, whose image is $C$.
Call them $f_i: C_i \to \PP^3, i=1, \ldots, n$.

A degree $d$ morphism from a curve of genus $g-1$ to $\PP^3$ is determined by a choice of a line bundle of degree $d$ and $4$ sections without base point.
With our choice of $d, g$, any choice of $4$ general sections works.
Using this, it is easy to show that there is an irreducible family of nodal curves with a morphism to $\PP^3$ over a smooth base $U$, a general member of which is a smooth curve of genus $g-1$, and the family contains all the  morphisms $f_i, i=1, \ldots, n,$ above.

This gives a family of cycles over $U$ as in Theorem \ref{thm:filtered_sm} (choose $u_0, u_1$ to be any two of the points corresponding to the morphisms $f_i, 1 \leq i \leq n$).
A general nodal deformation of the morphism $f_i$ has to have a nodal curve of arithmetic genus $g-1$ as the domain curve. 
If the deformation fixes the image, it is of degree $1$ as a morphism to $C$. 
Thus the deformed morphisms to $C$ (which is of arithmetic genus $g$) has to be the normalization at one of the nodes.
So the morphisms $f_i: C_i \to C, 1 \leq i \leq n,$ are rigid.
As a result, any deformation between the $f_i$'s will have to deform the cycle class as well.
\end{ex}
\subsection{Deforming complete intersection curves}
This section is of auxiliary nature. We collect a few results that are essentially proved in \cite{Kollar_Tian}. 
There are minor twists in the formulation.
So we provide details here.
\begin{defn}\label{def:deformation}
    Let $X$ be a proper scheme over an algebraically closed field. A \emph{deformation equivalence} (of curves mapping to $X$) is a diagram $(B, b_1, b_2) \leftarrow S \to X$, where    
\begin{enumerate}
    \item $(B, b_1, b_2)$ is a two pointed connected curve with smooth marked points $b_i$, 
    \item $S\to B$ is flat, proper, with fibers of pure dimension $1$.
\end{enumerate}

The $b_i$ and the intersection points of different irreducible components of $B$ are the \emph{pivot points}.
\end{defn}

\begin{defn}
    A \emph{$3$-nodal} (curve) singularity is a triple point that is locally isomorphic to $3$ coordinate axes $(xy = yz = zx = 0)\subset \AAA^3$.
\end{defn}

\begin{defn}\label{def:3nodal}
    A deformation equivalence as in Definition \ref{def:deformation} is called \emph{nodal} (or \emph{$2$-nodal}) if all fibers are nodal curves (singularities like $2$ coordinate axes $(xy=0)\subset\AAA^2$).
We call a deformation \emph{$3$-nodal} if all but finitely many fibers are nodal curves, and the pivot fibers have only nodes and $3$-nodal points as singularities.
\end{defn}
\begin{rem}\label{rem:3to2}
    By \cite{Kollar_Tian} Example 19, every $3$-nodal deformation between nodal curves gives rise to a nodal deformation.
\end{rem}

\begin{lem}[{cf. \cite[(42.2)-(42.5)]{Kollar_Tian}}]\label{lem:r-connect}
    Let $X$ be a smooth projective variety defined over an algebraically closed  field $k$. Fix any positive integer $r$. There is a family of smooth curves in $X \times \PP^1$:
    $\mathbf{L}_r \subset S \times X \times \PP^1$ satisfying the following condition: 
    \begin{enumerate}
        \item For any length $r$ subscheme of $X \times \PP^1$ whose embedding dimension is $1$, the subfamily of curves that pass through this subscheme is non-empty and irreducible;
        \item Let $M \leftarrow \mathcal{L}_r \xrightarrow{F} X \times \PP^1, \sigma_i: M \to \mathcal{L}, i=1, \ldots, r,$ be the universal family parameterizing curves in $\mathbf{L}_r$ with $r$-distinct points. Then the evaluation map \[
        ev=\prod_{i=1}^r F\circ \sigma_i: M \to \prod_{i=1}^r(X \times \PP^1) \]
        is flat.
    \end{enumerate}
\end{lem}
\begin{proof}
    Take a sufficiently ample linear system $|A|$ on $X\times \PP^1$ and take the family of curves to be smooth complete intersections of divisors in $|A|$.
\end{proof}
We will call a curve in a chosen family $\mathbf{L}_r\subset S \times X \times \PP^1$ an \emph{$\mathbf{L}_r$-curve} in the following.

\begin{rem}\label{rem:L=hci}
    The $H^\ci$-family  defined in \cite[(28)]{Kollar_Tian} is a special case of $\mathbf{L}_2$-family. The difference is that we do not choose an explicit form of the ample divisor $A$.
\end{rem}
\begin{rem}\label{const:lifting}
     In the following, we often need to lift finitely many families of curves $S \leftarrow Z_i \to X$ to families of curves $S \to {Z}_i \to X \times \PP^1$ such that the images of ${Z}_i$ are disjoint.
     One way to do this is to choose different points $p_{i}$ in $\PP^1$ and use the compositions $Z_i \to X \cong X\times \{p_{i}\} \to X \times \PP^1$).   
\end{rem}

In the following, we will frequently use the construction of gluing families of nodal curves along smooth points. See \cite[Theorem 3.4]{Knudsen_projectivity_2} for detailed information.

\begin{lem}[{cf. \cite[Corollary 29]{Kollar_Tian}}]\label{lem:join}
    Keep the same notations as in Lemma \ref{lem:r-connect}. 
    The family $\mathbf{L}_r \subset S\times X \times \PP^1$ also satisfies the following property.

    Let $U \leftarrow \mathcal{C}_i \xrightarrow{p_i} X, \sigma_i: U \to \mathcal{C}_i, i=1, \ldots, r$ be any $r$ flat projective families of nodal curves with a marked point in the smooth locus.
    Choose a lifting of $U \leftarrow \mathcal{C}_i \to X$ to $U \leftarrow \mathcal{C}_i \xrightarrow{\tilde{p}_i} X\times \PP^1,$ such that the images of $\tilde{p}_i \circ \sigma_i:U \to \mathcal{C}_i \to X \times \PP^1$ are disjoint, $i=1, \ldots, r$ (Remark \ref{const:lifting}).
    Assume that $U$ is quasi-projective.
    Choose finitely many points $u_1, \ldots, u_n$, and curves $L_1, \ldots, L_n$ in the family $\mathbf{L}_r$ such that $L_j (1\leq j \leq n)$ passes through the points $\tilde{p}_1(\sigma_1(u_i)), \ldots, \tilde{p}_1(\sigma_1(u_i)) (1\leq j \leq n)$.
    After a generically finite dominant base change $f: V \to U$, we may construct a family of $\mathbf{L}_r$-curves with $r$-marked points, $V \leftarrow \mathcal{L} \xrightarrow{p} X, \tau_i: V \to \mathcal{L}, i=1, \ldots r,$ such that
    \begin{enumerate}
        \item There are points $v_1, \ldots, v_n \in V$ such that $f(v_j)=u_j$ and $\mathcal{L}_{v_j}=L_j, 1\leq j \leq n$.
        \item $p\circ \tau_j=p_j \circ \sigma_j, 1 \leq j \leq n$.
    \end{enumerate}
    As a result, one can glue the families $\mathcal{C}_i,\mathcal{L}$ over $V$ along the sections $\sigma_i$ and $\tau_i$ to form a family of \emph{connected} nodal curves.
\end{lem}

\begin{proof}
Fix a sufficiently ample divisor $A$ and consider the universal family 
    \begin{align*}
        \textbf{CI}&=\{(u, W)|u \in U, \tilde{p}_i\circ \sigma_i(u)\in \cap_{D \in W}D, W\subset {H^0(X\times \PP^1, \OO_{X\times \PP^1}(A))}, \dim W=d \}\\
        &\subset U \times \text{Grass}(d, H^0(X\times \PP^1, \OO_{X\times \PP^1}(A))), d=\dim X.
    \end{align*}
    Denote by $\textbf{CI}^{\text{sm}}$ the locus where the base locus of the linear system $|W|$ is a smooth complete intersection curve.
    For $A$ sufficiently ample, the fiber $\textbf{CI}^{\text{sm}}_u$ is a dense open subset of the fiber $\textbf{CI}_u$ for each $u\in U$.
  Since we constructed $L_i$ as complete intersection curves, each $L_i$ corresponds to a unique $W_i\subset {H^0(X\times \PP^1, \OO_{X\times \PP^1}(A))}$.
    By taking general hyperplane sections of $\textbf{CI}^{\text{sm}}$ containing the points $(u_i, W_i), 1\leq i \leq n$, we can construct a multisection $V\subset \textbf{CI}^{\text{sm}}\to U$. The morphism $f:V\to U$ is the desired base change, and the family $\mathcal{L}\to V$ is the family of complete intersection curves over $V$.
\end{proof}

\begin{defn}[{cf. \cite[Definition 30]{Kollar_Tian}}]\label{def:join}
    Let $X$ be a smooth projective variety defined over an algebraically closed  field $k$, and let $\mathbf{C}_r=\{\pi_i:C_i \to X, i=1, \ldots, r\}$ be a set of $r$ reduced curves of $X$. 
    
    Fix a family $\mathbf{L}_r \subset S \times X \times \PP^1$ as Lemma \ref{lem:r-connect}.
    Denote by $S \leftarrow \mathbf{L}_r \xrightarrow{F} X$ the family of curves obtained from projection to $X$.

    Define
    $\text{Join}(\mathbf{C}_r, \mathbf{L}_r)$
    to be the scheme parameterizing curves of the form
    \[
    \pi: (\coprod_{i=1}^r C_i) \cup_{\sigma} (L, p_1, \ldots, p_r) \to X,
    \]
    where $(L, p_1, \ldots, p_r)$ is an $\mathbf{L}_r$-curve with \emph{$r$-distinct points} $p_1, \ldots, p_r$,  $$\sigma: \{p_1, \ldots, p_r\} \to \coprod_{i=1}^r C_i$$ is an embedding such that $\sigma(p_i) \in C_i, \pi_i(\sigma(p_i))=F(p_i)$, and the curve
    \[
    \pi: (\coprod_{i=1}^r C_i) \cup_{\sigma} (L, p_1, \ldots, p_r) \to X,
    \]
    is glued from $(\coprod_{i=1}^r C_i) \to X$ and $L \to X$ by identifying the closed subschemes $p_i \cong \sigma(p_i)$.

    We define 
    $\text{Join}^\text{nodal}(\mathbf{C}_r, \mathbf{L}_r)$
    to be the subfamily of curves such that in addition, $\sigma(p_i) \in C_i^{\text{sm}}, i=1, \ldots, r$.    
\end{defn}
\begin{rem}\label{rem:3dal}
    If each $C_i (1 \leq i \leq r)$ is a nodal curve, the curve $(\coprod_{i=1}^r C_i) \cup_{\sigma} (L, p_1, \ldots, p_r)$ is a curve with nodal or $3$-nodal singularities.
    If in addition $\sigma(p_i) \in C_i^{\text{sm}}, i=1, \ldots, r$, the curve
    $(\coprod_{i=1}^r C_i) \cup_{\sigma} (L, p_1, \ldots, p_r)$ is a nodal curve \cite[Theorem 3.4]{Knudsen_projectivity_2}.
\end{rem}

When $r=2$, the following is a special case of \cite[Corollary 31]{Kollar_Tian}.
\begin{lem}\label{lem:deform}
    Let $X$ be a smooth projective variety defined over an algebraically closed field $k$. 
    Let $\mathbf{C}_r=\{\pi_i:C_i \to X, i=1, \ldots, r\} $ be a set of $r$ geometrically connected nodal curves in $X$. 
    Fix a family $\mathbf{L}_r\subset S \times X\times \PP^1$ as in Definition  \ref{def:join}. 
    \begin{enumerate}
        \item $\text{Join}(\mathbf{C}_r, \mathbf{L}_r)$ is connected.
        \item $\text{Join}^\text{nodal}(\mathbf{C}_r, \mathbf{L}_r)$ is an open dense subset of $\text{Join}(\mathbf{C}_r, \mathbf{L}_r)$.
        
        \item For any two members of the family of curves $\text{Join}^{\text{nodal}}(\mathbf{C}_r, \mathbf{L}_r)$, there is a $3$-nodal deformation from one to the other over a connected curve, which is obtained by gluing the constant deformation family of the curves $C_1, \ldots C_r$ to a family of deformations of curves in $\mathbf{L}_r$.
        We can also modify the deformation at pivot points to make it nodal.
    \end{enumerate}
\end{lem}

\begin{proof}[Proof of Lemma \ref{lem:deform}]
    Observe that that is a morphism $$\Phi: \text{Join}(\textbf{C}_r, \mathbf{L}_r)\to \prod_{i=1}^rC_i$$ by mapping a curve $(\coprod_{i=1}^r C_i) \cup_{\sigma} (L, p_1, \ldots, p_r)$ to $\prod_i \sigma(p_i) \in \prod C_i$. 
    The morphism $\Phi$ is flat and has non-empty and geometrically irreducible fiber over any point by assumption on $\mathbf{L}_r$. This proves the first statement.
    The second statement is true because  $\text{Join}^{\text{nodal}}(\mathbf{C}_r, \mathbf{L}_r)=\Phi^{-1}(\prod_{i=1}^rC_i^\sm)$, which is dense in $\prod_{i=1}^rC_i$.

    Finally, the last statement follows from the first two: take a connected curve in $\text{Join}(\mathbf{C}_r, \mathbf{L}_r)$ and use Remarks \ref{rem:3dal} and  \ref{rem:3to2}. 
\end{proof}

\subsection{Proof of Theorems \ref{thm:filtered_sm} and \ref{thm:filtered}}
    In this section, we will often glue (families of) nodal curves together by identifying smooth points on them. 
    That is, given two families of nodal curves $S \leftarrow C \xrightarrow{f} X$ and $S\leftarrow D \xrightarrow{g} X$, and finitely many sections $\sigma_i:S \to C, \tau_i: S \to D, 1=1, \ldots, n$ that lie in the smooth locus of $C \to S$ and $D \to S$ and such that $f\circ \sigma_i=g\circ \tau_i$, we glue $C, D$ by identifying the subschemes $\sigma_i(S)\subset C$ and $\tau_i(S)\subset D$ for each $i$ (\cite[Theorem 3.4]{Knudsen_projectivity_2}).
    If it is not important to keep track of which points are identified, 
    we will simply use notations like $C\cup D$ to denote the curve obtained by gluing $C$ and $D$ together along some smooth points.

    In the proofs, we will usually make base changes several times.
    To simplify the notations, we will use the same letter (usually $U$) to denote the base in the first few steps of base change.

    The following observation is useful.
    \begin{constr}[Creating isomorphic base]\label{samebase}
    Given any two pointed quasi-projective algebraic sets of the same dimension $(B^+, b_0^+, b_1^+), (B^-, b_0^-, b_1^-)$, we may take a general complete intersection $B$ (which is of the same dimension as $B^+$ and $B^-$) of $B^+\times B^-$ that contains the two points $b_0=(b_0^+, b_0^-)$ and $b_1=(b_1^+, b_1^-)$. 
    Then $(B, b_0, b_1)$ maps to both $(B^+, b_0^+, b_1^+), (B^-, b_0^-, b_1^-)$.

    If both $B^+$ and $B^-$ are irreducible, up to replacing $B^{+/-}$ with an alteration, we may always assume that the points $b_i^{+/-}$ are smooth points.
    Thus we may choose $B$ to be smooth and irreducible.
    
    If both $B^+$ and $B^-$ are connected curves, this construction is \cite[26.2]{Kollar_Tian}, and we may take $B$ to be a connected curve.
\end{constr}
    
The following is essentially \cite[Theorem 61]{Kollar_Tian}.
\begin{prop}[{cf. \cite[Theorem 61]{Kollar_Tian}}]\label{prop:addtodeform}
    Keep the same notations as Theorem \ref{thm:filtered_sm}. Assume furthermore that $U$ is quasi-projective and that $\Gamma_U$, the family of cycles over $U$, is an effective nodal family of one-cycles.
    Then there is a generically finite dominant base change $f: V \to U$, a connected two pointed quasi-projective curve $(T, t_0, t_1)$, two families of $\mathbf{L}_s$-curves $\mathcal{L}_T \to T$ and $\mathcal{L}_V \to V$, a connected nodal curve $Z \to X$, two families of $r$-tuple of $L_2$-curves $r|\mathbf{L}_2|_V \to V$, $r|\mathbf{L}_2|_T \to T$, such that the followings hold.
    \begin{enumerate}
        \item There are two points $v_0, v_1 \in V$ such that $f(v_0)=u_0$, $f(v_1)=u_1$.
        \item We may glue the families of curves $\Gamma_V$, $\mathcal{L}_V$, $Z\times V$, $r|\mathbf{L}_2|_V$ together to a family of \emph{connected} nodal curves $C_V$.
        \item We may glue the families $\mathcal{L}_T$, $r|\mathbf{L}_2|_T$, $Z\times T$, and a deformation over the image of $(\Gamma_V\cup \mathcal{L}_V)|_{v_0} $ to form a $3$-nodal deformation over $T$, which gives rise to a family of connected nodal curves $T\leftarrow C_T \to X$ as in Remark \ref{rem:3to2}.
        \item The cycle class of $C_T$ at each point $t \in T$ is $\gamma+[Z]+[\mathcal{L}_T|_t]+[r|\mathbf{L}_2|_T|_t]$.
        \item The fiber of $C_T$ over $t_0$ (resp. $t_1$) is the same as the fiber of $C_V$ over $v_0$ (resp. $v_1$).
     \end{enumerate}
\end{prop}
\begin{proof}
    We write $\Gamma_U=\sum m_i Z_i, U \leftarrow Z_i \to X$.
Take $m_i$ copies of each family of nodal curves $U \leftarrow Z_i \to X$.
Lift these $m_i$ copies of $U \leftarrow Z_i \to X$ for all $i$ to families of nodal curves in $X \times \PP^1$ with disjoint images (Remark \ref{const:lifting}).

We choose a family of $\mathbf{L}_s$-curves and glue them to these lifted families of nodal curves in $X \times \PP^1$ (up to a generically finite base change) as in Lemma \ref{lem:join} to produce a family of \emph{connected} nodal curves. 
We denote the family of $\mathbf{L}_s$-curves by $\mathcal{L}_U \to U$.
We call the resulting family of connected nodal curves ${\Gamma}_U\cup \mathcal{L}_U \to U$.
Moreover, we choose the family $\mathcal{L}_U \to U$ so that the fibers $\mathcal{L}_U|_{u_0}, \mathcal{L}_U|_{u_1}$ are the same.

We use the equality of cycle classes $[({\Gamma}_U\cup \mathcal{L}_U)|_{u_0}]=[({\Gamma}_U\cup \mathcal{L}_U)|_{u_1}]$, and apply \cite[Corollary 35]{Kollar_Tian}, which shows the following.
\begin{enumerate}
\item There is a copy of a fixed connected nodal curve $Z \to X$, connected to $({\Gamma}_U\cup \mathcal{L}_U)|_{u_0}$ by a curve $L_0$ in an $H^\ci$ family (\cite[28]{Kollar_Tian}), which is an $\mathbf{L}_2$-curve (see Lemma \ref{lem:r-connect}, Remark \ref{rem:L=hci}). 
Denote the resulting curve by $({\Gamma}_U\cup \mathcal{L}_U)|_{u_0} \cup L_0 \cup Z$.
And there is another copy of the fixed connected nodal curve $Z \to X$, connected to $({\Gamma}_U\cup \mathcal{L}_U)|_{u_1}$ by a curve $L_1$ in the same $H^\ci$($=\mathbf{L}_2$) family as $L_0$. Denote by $({\Gamma}_U\cup \mathcal{L}_U)|_{u_1} \cup L_1 \cup Z$ the resulting curve.
        
        \item  After further adding a suitable number of $H^\ci$($=\mathbf{L}_2$)-curves to both $({\Gamma}_U\cup \mathcal{L}_U)|_{u_0} \cup L_0 \cup Z$ and $({\Gamma}_U\cup \mathcal{L}_U)|_{u_1} \cup L_1 \cup Z$, there is a $3$-nodal deformation between the resulting two curves over a two-pointed connected curve $({T}, t_0, t_1) $, where $t_0, t_1$ are smooth points of ${T}$.
        We use $C_T\to T$ to denote the corresponding nodal family as in Remark \ref{rem:3to2}. 
    \end{enumerate}
    Finally, using Lemma \ref{lem:join}, after a generically finite base change, denoted by $f: V \to U$, we glue a family $r|\mathbf{L}_2|_V \to V$ of r-tuple of $\mathbf{L}_2$-curves over $V$, $Z \times V$, $(\Gamma_U \cup \mathcal{L}_U)\times_U V$ together to form a family of connected nodal curves $C_V \to V$ so that $C_V|_{v_0}=C_T|_{t_0}$ and $C_V|_{v_1}=C_T|_{t_1}$.
    The statements (3) and (4) follow from \cite[Corollary 35.4, Complement 36]{Kollar_Tian}
\end{proof}

\begin{proof}[Proof of Theorem \ref{thm:filtered_sm}]
    First, observe that if $U' \to U$ is a birational morphism (not necessarily proper) and the statement holds for $U'$ (with any lifting of $u_0, u_1$), then by the existence of a projective extension (Definition \ref{def:extension} and the remark after), the statement holds for $U$.
    
    Based on this observation, using Nagata's compactification and Chow lemma, we can assume that $U$ is projective.

    We write $\Gamma_U=\sum_i m_i Z_i, Z_i \subset U \times X$.
    Denote by $\tilde{Z}_i$ the normalization of the component $Z_i$. 
    Up to a purely inseparable base change, we may assume that $\tilde{Z}_i \to Z_i \to U$ is generically smooth for all $i$ (apply \cite[Lemma 2.8]{deJong_families_of_curves_and_alteration} to the generic fiber of each component $Z_i$ of the cycle $\Gamma_U$).

    We make a generically finite projective base change $f': U' \to U$ and semistable reduction \cite[Theorem 5.8]{deJong_alteration} to each component of $\Gamma_U$,  
    and thus produce a nodal family of one-cycles $U' \leftarrow {\Gamma}_{U'} \to X$ such that $f'^*(\Gamma_U)=[\Gamma_{U'}]$ as equi-dimensional family of one-cycles.
    We may assume that $U'$ is smooth by the existence of an alteration.
    The general fibers of each $\tilde{Z}_i \to U$ are smooth and thus for any general point $u \in U$, and any point $u' \in U'$ lying over $u$, the fibers $\tilde{Z}_i\times_U U'|_{u'}$ and $\tilde{Z}_i|_{u}$ are isomorphic.
    Denote by $U^0$ the open subset of $U$ consisting of points $u\in U$ satisfying the above condition.
     Then for all points $x_1, x_2, \ldots$ in $U'$ that is in the inverse image of $u \in U^0$, we have an equality of nodal one-cycles $\Gamma_{U'}|_{x_1}=\Gamma_{U'}|_{x_2}=\ldots$.
     
   We write $\Gamma_{U'}=\Gamma^+_{U'}-\Gamma^-_{U'}$, with $\Gamma^+_{U'}$ (resp. $\Gamma^-_{U'}$) as its positive (resp. negative) part. We write
    \[
\gamma_0^+=\Gamma^+_{U'}|_{u_0},\quad \gamma_0^-=\Gamma^-_{U'}|_{u_0},\quad
    \gamma_1^+=\Gamma^+_{U'}|_{u_1},\quad
    \gamma_1^-=\Gamma^-_{U'}|_{u_1}.
    \]
    By assumption, 
    $[\gamma_0^+]-[\gamma_0^-]=[\gamma_1^+]-[\gamma_1^-]$. 

    We take a general complete intersection $V'$ (of the same dimension as $U'$) containing $v_0'=(u_0, u_1)$ and $v_1'=(u_1, u_0)$ in the product $U' \times U'$.
    Note that $V'$ is smooth, irreducible and projective (since $U'$ is smooth, irreducible and projective).
    There are two nodal families of one-cycles $$\Gamma_p=p^*\Gamma_{U'}=\Gamma_p^+-\Gamma_p^-, \quad\Gamma_q=q^*\Gamma_{U'}=\Gamma_q^+-\Gamma_q^-$$ over $V'$ induced by pulling back $\Gamma_{U'}$ along the the two projections $p, q: V' \to U'$, where $\Gamma_p^+$ (resp. $\Gamma_q^+$) is the positive part of $\Gamma_p$ (resp. $\Gamma_q$), and $\Gamma_p^-$ (resp. $\Gamma_q^-$) is the negative part of $\Gamma_p$ (resp. $\Gamma_q$). 
    We construct two new nodal families of one-cycles over $V'$:
    \[
{\Gamma}_{V'}^+=\Gamma_p^++\Gamma_q^-,\quad \Gamma_{V'}^-=\Gamma_p^-+\Gamma_q^-.
    \]
    We remark that $\Gamma_p=p^*\Gamma_{U'}=\Gamma_{V'}^+-\Gamma_{V'}^-$.
    The nodal one-cycle $\Gamma_{V'}^+|_{v_0'}$ is $\gamma_0^++\gamma_1^-$, and the nodal one-cycle $\Gamma_{V'}^+|_{v_1'}$ is $\gamma_1^++\gamma_0^-$. 
    Thus their cycle classes are the same. 
    It follows from a similar computation that the cycle class of the restriction of $\Gamma_{V'}^-$ to $v_0', v_1'$ are the same.

    So now we have two \emph{effective} nodal families of one-cycles $\Gamma_{V'}^+, \Gamma_{V'}^-$, and the restriction of each family to $v_0', v_1'$ has the same cycle class.
    We write $\tilde{p}:V' \xrightarrow{p} U' \to U$.

    We apply Proposition \ref{prop:addtodeform} to $\Gamma_{V'}^+$ and $\Gamma_{V'}^-$.
    Using Construction \ref{samebase}, we may take a single normal variety $V$ and a connected two pointed projective curve $(T, t_0, t_1)$ in Proposition \ref{prop:addtodeform} that works for both families $\Gamma_{V'}^+$ and $\Gamma_{V'}^-$.
    Denote by $f: (V, v_0, v_1) \to (U, u_0, u_1)$ the base change. 
    For the family $\Gamma_{V'}^+$ (resp. $\Gamma_{V'}^-$), denote by $Z^+$ (resp. $Z^-$) the constant nodal curve, $\mathcal{L}_V^+$, $\mathcal{L}_T^+$ (resp. $\mathcal{L}_V^-$, $\mathcal{L}_T^-$) the family of $\mathbf{L}_s$-curves over $V$ and $T$, $r|\mathbf{L}_2|_V^+$,  $r|\mathbf{L}_2|_T^+$ (resp. $r|\mathbf{L}_2|_V^-$, $r|\mathbf{L}_2|_T^-$) the family of $\mathbf{L}_2$-curves over $V$ and $T$, $C_V^+$ (resp. $C_V^-$) the family of nodal curves  over $V$ glued from above families, $C_T^+$ (resp. $C_T^-$) the family of nodal curves over $T$ glued from above families and deformation over the image of $(\Gamma_{V}^+\cup \mathcal{L}_V^+)|_{v_0}$ (resp. $(\Gamma_{V}^-\cup \mathcal{L}_V^-)|_{v_0}$).

    We define 
    \[
    \Gamma_V=(C_V^+-\mathcal{L}_V^+-r|\mathbf{L}_2|_V^+-Z^+\times V)-(C_V^--\mathcal{L}_V^--r|\mathbf{L}_2|_V^--Z^-\times V),
    \]
    \[
    \Gamma_T=(C_T^+-\mathcal{L}_T^+-r|\mathbf{L}_2|_T^+-Z^+\times T)-(C_T^--\mathcal{L}_T^--r|\mathbf{L}_2|_T^--Z^-\times T).
    \]
    By Proposition \ref{prop:addtodeform}.(4), the family of cycles $[C_T^+]$ is a fixed cycle plus deformations in the families $\mathbf{L}_s$ and $r|\mathbf{L}_2|$.
    In the definition of $\Gamma_T$, we subtract the deformations of cycles in the families $\mathbf{L}_s$ and $r|\mathbf{L}_2|$ in $[C_T^+]$.
    The same is true for the family $[C_T^-]\to T$.
    Therefore, $[\Gamma_T] \to T$ is a family of constant cycles.

    By abuse of notations, we still denote the pull-back families of $\Gamma_p, \Gamma_q$ along the morphism $V \to V'$ by $\Gamma_p$ and $\Gamma_q$.

    Choose a point $u \in U^0$ and let $v, v'$ be two points in $f^{-1}(u)$.
    Our next goal is to construct the family of constant cycles over a curve connecting $v, v'$.
    
    Choose an irreducible smooth two pointed curve $(B_{v, v'}, b, b')$ mapping to  $(V, v, v')$ and denote the pull-back of the family $\Gamma_q^- \to V$ by $\Gamma_{B_{v, v'}}\to B_{v, v'}$.
    
    The fiber of $C_V^+$ over $v$  (resp. $v'$) is a nodal curve of the form 
     \[
     {\Gamma}_p^+|_{v}\cup \Gamma_q^-|_v \cup L_{v} \cup Z^+ \cup r|\mathbf{L}_2|_{v}
     \]
     (resp.
     ${\Gamma}_p^+|_{v'}\cup \Gamma_q^-|_{v'} \cup L_{v'} \cup Z^+ \cup r|\mathbf{L}_2|_{v'}$).  
     Here ${\Gamma}_p^+|_{v}$ and ${\Gamma}_p^+|_{v'}$ are by definition effective nodal one-cycles. 
     But we abuse the notations by treating them as disconnected nodal curves consisting of the disjoint union of $m_i$ copies of the nodal curve $\tilde{Z}_i|_{u}$.
     We think of the fiber of $C_V^+$ over $v$  (resp. $v'$) as nodal curves $Z^+$ and ${\Gamma}_{p}^+|_{v}, \Gamma_q^-|_{v}$ (resp. ${\Gamma}_{p}^+|_{v'}, \Gamma_q^-|_{v'}$),  connected by two types of curves: 
     \begin{itemize}
         \item an $\mathbf{L}_{s}$-curve $L_v$ (resp. $L_{v'}$) connecting ${\Gamma}_{p}^+|_{v}$ and $\Gamma_q^-|_{v}$ (resp. ${\Gamma}_{p}^+|_{v'}$ and $\Gamma_q^-|_{v'}$),
         \item $r$-tuple of $\mathbf{L}_2$-curves $r|\mathbf{L}_2|_{v}$ (resp. $r|\mathbf{L}_2|_{v'}$) connecting ${\Gamma}_p^+|_{v}\cup \Gamma_q^-|_v \cup L_{v}$ and $Z^+$
     (resp.
     ${\Gamma}_p^+|_{v'}\cup \Gamma_q^-|_{v'} \cup L_{v'}$ and $Z^+$).
     \end{itemize}  
     We first apply Lemma \ref{lem:join} (once for the family $\mathbf{L}_s$ and $r$-times for the family $\mathbf{L}_2$) to the families $\Gamma_{B_{v, v'}}\to B_{v, v'}$, $\Gamma_p^+|_v \times B_{v, v'} \to B_{v, v'}$, $Z^+ \times B_{v, v'} \to B_{v, v'}$ to construct a deformation $C^{q+}_{v, v'} \to (B^+_{v, v'}, b_{v, v'}^v, b_{v, v'}^q)$ from the curve
     \[
     C^{q+}_{v, v'}|_{b_{v, v'}^v}={\Gamma}_p^+|_{v}\cup \Gamma_q^-|_v \cup L_{v} \cup Z^+ \cup r|\mathbf{L}_2|_{v}
     \]
     to a curve of the form
     \[
     C_q^+=C^{q+}_{v, v'}|_{b_{v, v'}^q}={\Gamma}_p^+|_{v}\cup \Gamma_q^-|_{v'} \cup L \cup Z^+ \cup L'.
     \]
     where $L$ is an $\mathbf{L}_s$-curve and  $L'$ consists of r-tuple of $|\mathbf{L}_2|$-curves.

     Recall that by the choice of $u$, $\Gamma_p^+|_v=\Gamma_p^+|_{v'}$.
     We deform the curve $L$ and each irreducible component of $L'$ over a connected curve using Lemma \ref{lem:deform} (first apply for the family $\mathbf{L}_s$ and then apply r times for the family $\mathbf{L}_2$) to construct a nodal deformation $C^{L+}_{v, v'} \to (T^+_{v, v'}, t_{v, v'}^q, t_{v, v'}^{v'})$ from  
     \[
     C_q^+={\Gamma}_p^+|_{v}\cup \Gamma_q^-|_{v'} \cup L \cup Z^+ \cup L'
     \]
     to
     \[
     {\Gamma}_p^+|_{v'}\cup \Gamma_q^-|_{v'} \cup L_{v'} \cup Z^+ \cup r|\mathbf{L}_2|_{v'}.
     \]
     Here we first produce a $3$-nodal deformation
     constructed by gluing the constant families $({\Gamma}_p^+|_{v}\cup \Gamma_q^-|_{v'})\times T_{v, v'}^+, Z\times T_{v, v'}^+$, and two deformations $r|\mathbf{L}_2|_{v, v'} \to T_{v, v'}^+, \mathcal{L}_{v, v'} \to T_{v, v'}'^+$ of the curves in the families $r|\mathbf{L}_2|$ and $\mathbf{L}_{s}$.
     Then we use Remark \ref{rem:3to2} to turn it into a nodal deformation.

     We then do the same type of deformation for the curves $C_V^-|_v, C_V^-|_{v'}$ using the family $\Gamma_{B_{v, v'}}\times_{B_{v, v'}} B^+_{v, v'}\to B^+_{v, v'}$ and deformations of curves in the family $\mathbf{L}_s$ and $\mathbf{L}_2$.
     Again the deformation consists of two steps: first a deformation over an irreducible curve $C^{q-}_{v, v'}\to B^-_{v, v'}$ and then a deformation over a connected curve $C^{L-}_{v, v'}\to T^-_{v, v'}$.
     Note that there is also a deformation $C^{q+}_{v, v'}\times_{B^+_{v, v'}}B^-_{v, v'} \to B^-_{v, v'}$.

     By Construction \ref{samebase}, we may assume that  $T_{v, v'}^-$ maps to $T^+_{v, v'}$ and thus we have the pull-back deformation $C_{v, v'}^{L+}\times_{T_{v, v'}^+} T_{v, v'}^- \to T_{v, v'}^-$. 

     We glue $T_{v, v'}^-$ and $B_{v, v'}^-$ together by identifying the two points $b_{v, v'}^q$ and $t_{v, v'}^q$, the fiber over which is the nodal curve $C_q^+$.
     We call the resulting curve $T_{v, v'}$
     and we take $t_{v, v'}^v$ to be $b_{v, v'}^v$.
     We glue the two families $C^{q+}_{v, v'}\times_{B^+_{v, v'}}B^-_{v, v'}$ and $C_{v, v'}^{L+}\times_{T_{v, v'}^+} T_{v, v'}^-$ together (along the nodal curve $C^+_q$) to a nodal deformation family over the connected two pointed curve $(T_{v, v'}, t_{v, v'}^{v}, t_{v, v'}^{v'})\leftarrow C^{+}_{v, v'} \to X$.
     In a similar way, by gluing the two families $C^{q-}_{v, v'}$ and $C_{v, v'}^{L-}$ together (along the nodal curve $C^-_q$), we construct a nodal deformation family over the connected two pointed curve $(T_{v, v'}, t_{v, v'}^{v}, t_{v, v'}^{v'})\leftarrow C_{v, v'}^- \to X$.

     We take the nodal family of one-cycles over $T_{v, v'}$ to be
    \[
    \Gamma_{v, v'}=(C_{v, v'}^+-r|\mathbf{L}_2|_{v, v'}^+-Z^+ \times T_{v, v'}-\mathcal{L}_{v, v'})-(C_{v, v'}^--r|\mathbf{L}_2|_{v, v'}^--Z^- \times T_{v, v'}-\mathcal{L}_{v, v'}^-).
    \] 
    Over the subcurve $B^{-}_{v, v'}\subset T_{v, v'}$, the family of cycles $\Gamma_{v, v'}^+$ is
    \[
    [\Gamma_p^+|_v+\Gamma_q^-|_b+Z^++\mathcal{L}_{v, v'}^+|_b+r|\mathbf{L}_2|_{v, v'}^+|_b].
    \]
    The family of cycles $\Gamma_{v, v'}^-$ is
    \[
    [\Gamma_p^-|_v+\Gamma_q^-|_b+Z^-+\mathcal{L}_{v, v'}^-|_b+r|\mathbf{L}_2|_{v, v'}^-|_b].
    \]
    The varying family of cycles $[\Gamma_q^-|_b]$ appears in both the positive and negative side. 
    Thus they do not contribute to the family of cycles $[\Gamma_{v, v'}]$.
    In the construction of $\Gamma_{v, v'}$, we subtract the varying families of cycles $[\mathcal{L}_{v, v'}^+|_b+r|\mathbf{L}_2|_{v, v'}^+|_b]$ (resp. $[\mathcal{L}_{v, v'}^-|_b+r|\mathbf{L}_2|_{v, v'}^-|_b]$) from $[C_{v, v'}^+]$ (resp. $[C^-_{v, v'}]$). 
    Therefore the family of cycles $[\Gamma_{v, v'}]$ is constant.
    By a similar computation, for the subcurve $T_{v, v'}^-\subset T_{v, v'}$,
    the family of cycles $[\Gamma_{v, v'}]$ is constant.

    Now we apply projective extension (\ref{def:extension}) to the nodal families over $V, T, T_{v, v'}$. By abuse of notations, we use the same letters to denote the extensions.
    Then we achieve that $f: V \to U$ is projective, and that $T, T_{v, v'}$ are connected projective curves.
    
    To construct $W$, we note that the nodal families over the pair of points $(v_0, t_0)$ (resp. $(v_1, t_1)$, $(v, t_{v, v'}^v)$, $(v', t_{v, v'}^{v'})$) are the same. 
    So we glue the schemes $V, T, T_{v, v'}$ together by identifying the pair of points $(v_0, t_0)$ (resp. $(v_1, t_1)$, $(v, t_{v, v'}^v)$, $(v', t_{v, v'}^{v'})$).
    We take $W$ to be the resulting scheme and it carries a nodal family obtained by gluing the nodal families over $V, T, T_{v, v'}$ together.
\end{proof}

\begin{rem}
    We observe the following interesting phenomenon.
    Over the curves $T, T_{v, v'}$, the families of cycles are constant, but both the positive and negative part are varying as families of cycles.
    Over the curve $T$, this is due to the fact that we deform curves in $\mathbf{L}_s, r|\mathbf{L}_2|$ to construct the family of curves $C_T, C_{v, v'}$, and thus the positive part of the family of cycles is a constant one-cycle plus deformation of cycles in the family $\mathbf{L}_{s}, r|\mathbf{L}_2|$. Similarly, the negative part is also a constant one-cycle plus deformation of cycles in the families $\mathbf{L}_{s}, r|\mathbf{L}_2|$.
    Over the curve $T_{v, v'}$, there is also a deformation in the family $\Gamma_q^-$ on both the positive and negative side.
\end{rem}

\begin{proof}[Proof of Theorem \ref{thm:filtered}]
     We first apply Theorem \ref{thm:filtered_sm}.
     We use $V^s$ (resp. $T^s$, $T^s_{v,v'}$, $U^s$, $W^s$) to denote $V$ (resp. $T$, $T_{v, v'}$, $U^0$, $W$) in the statement of Theorem \ref{thm:filtered_sm}.
     Choose an embedding of each component of $\Gamma_{W^s}$ into $\PP^M$.
     This gives a morphism
     \[
     \Phi: V^s \coprod T^s \coprod (\coprod_{f^s(v)=f^s(v')=u} T^s_{v,v'}) \to W^s \to  \prod_i \text{Hilb}^\text{nodal}(X \times \PP^M),
    \]
    where $\text{Hilb}^\text{nodal}(X \times \PP^M)$ is the open locus of the Hilbert scheme of $X \times \PP^M$ parameterizing curves with at worst nodal singularities.
    Denote by $\pi_i$ the  projection from $\text{Hilb}^\text{nodal}(X \times \PP^M)$ to the $i$-th factor and $\Phi_i=\pi_i \circ \Phi$.
    Denote by $W_i$ the image of $\Phi_i$. 
    Each $W_i$ is connected.
    
    Denote by $\mathbf{T} \subset \text{Hilb}^\text{nodal}(X \times \PP^M)$ the family of curves constructed in \cite[Theorem 44]{Kollar_Tian}.
    Since $\mathbf{T}$ parameterizes free curves (\cite[(41.2)]{Kollar_Tian}), it is contained in the smooth locus of a unique irreducible component of the Hilbert scheme, denoted by $H_\mathbf{T}$ (\cite[Claim 41.4]{Kollar_Tian}).
    
     Let $\ecomb^{d-\free}(W_i, r \mathbf{T})$ be space parameterizing family of embedded combs with handles parameterized by the image $W_i$ and $r$-teeth in $T$, and that are $d$-free along the handles and free \cite[Definition 45, (41.2)]{Kollar_Tian}. 
    Each $\ecomb^{d-\free}(W_i, r \mathbf{T})$ admits a morphism to the $r$-fold product $H_\mathbf{T}^r$.
     There is also a morphism $\Pi_i: \ecomb^{d-\free}(W_i, r \mathbf{T}) \to W_i$, which, by \cite[Theorem 48]{Kollar_Tian}, is surjective, has connected fibers, and satisfies the curve lifting property (\cite[Definition 47]{Kollar_Tian}).
     As a result, we may find a quasi-projective variety $\tilde{V}$ (resp. two pointed connected quasi-projective curve $(\tilde{T}, \tilde{t}_0, \tilde{t}_1)$, $(\tilde{T}_{v, v'}, \tilde{t}_{v, v'}^v, \tilde{t}_{v, v'}^{v'})$)  with a generically finite dominant morphism $\tilde{V} \to V^s$ (resp. a morphism of two pointed curves $(\tilde{T}, \tilde{t}_0, \tilde{t}_1) \to ({T}^s, {t}^s_0, {t}^s_1)$, $(\tilde{T}_{v, v'}, \tilde{t}_{v, v'}^v, \tilde{t}_{v, v'}^{v'})\to ({T}^s_{v, v'}, {t}_{v, v'}^{s, v}, {t}_{v, v'}^{s, v'}) ),$ which induces a nodal family of one-cycles $\tilde{
V} \leftarrow \Gamma_{\tilde{V}} \to X$ (resp. $\tilde{T} \leftarrow \Gamma_{\tilde{T}_{v, v'}} \to X$, $\tilde{T}_{v, v'}\leftarrow \Gamma_{\tilde{T}}\to X$), and a morphism 
     $\tilde{\Phi}: \tilde{V} \coprod \tilde{T} \coprod (\coprod_{v, v'} \tilde{T}_{v,v'}) \to \prod_i \ecomb^{d-\free}(W_i, r \mathbf{T})$
     such that the following diagram commutes
     \[
     \begin{CD}
         \tilde{V} \coprod \tilde{T} \coprod (\coprod_{v, v'} \tilde{T}_{v,v'}) @>\tilde{\Phi}>> \prod_i \ecomb^{d-\free}_i(W_i, r \mathbf{T})\\
         @VVV @VVV\\
         V^s \coprod T^s \coprod (\coprod_{v, v'} T^s_{v,v'}) @>(\prod \Phi_i) >> \prod_i W_i 
     \end{CD}
     \]
     
    Since $W_i$ is connected, there is a unique geometrically irreducible component $H_i$ of $\text{Hilb}(X \times \PP^M)$ containing $\ecomb^{d-\free}(W_i, r \mathbf{T})$ (\cite[Corollary 49]{Kollar_Tian}).
    
    In positive characteristic, we take $W$ to be the normalization of $\prod_i (H_i \times \prod_{j=1}^r H_\mathbf{T})$. 
    In characteristic $0$, we take a resolution of singularities $W \to \prod_i (H_i \times \prod_{j=1}^r H_\mathbf{T})$ that is an isomorphism over the smooth locus. 
    The normalization/resolution of singularities is an isomorphism in a neighbourhood of the image of $\ecomb^{d-\free}_i(W_i, r \mathbf{T})$.
 The universal family of subschemes induces an equi-dimensional family of one-cycles over $H_i, H_\mathbf{T}^r$.
    Denote by $H_i\leftarrow \Gamma_i \to X$ and $H_\mathbf{T}\leftarrow \Gamma_\mathbf{T} \to X$ the corresponding equi-dimensional family of one-cycles.
    Then $W$ carries an equi-dimensional family of one-cycles
    $\Gamma_W=\sum m_i p_i^*\Gamma_i-\sum_{j=1}^rp_{ij}^*\Gamma_\mathbf{T}$, where $p_i: W \to H_i$ and $p_{ij}: W \to H_{\mathbf{T}}$ are the natural morphisms to each factor of $\prod_i (H_i \times \prod_{j=1}^r H_\mathbf{T})$.
    The cycle class of the nodal family of one-cycles on $\tilde{V} \coprod \tilde{T} \coprod (\coprod_{v, v'} \tilde{T}_{v,v'})$ equals the pull-back of $\Gamma_W$.
    
    Denote by $\tilde{f}$ the morphism $\tilde{V} \to U$.
     We take ${V}$ to be a normal partial compactification of $\tilde{V}$ such that the morphism $\tilde{f}:\tilde{V} \to V^s$ extends to a projective morphism $f_{V/V^s}: {V} \to V^s$, and such that the morphism $\tilde{V} \to W$ extends to $V \to W$ (e.g. take $V$ to be the normalization of the graph closure of $\tilde{V} \subset V^s \times W \times \PP^N$ after choosing a locally closed embedding $\tilde{V} \subset \PP^N$), as shown in the following diagram.
     \[
     \xymatrix{
     {V} \ar@/_/[ddr]_{\text{projective morphism } f_{V/V^s}} \ar@/^/[drrr] \\
&\tilde{V} \ar[d]^{\tilde{f}} \ar[r] \ar@{_{(}->}[ul]& \prod_i \ecomb^{d-\free}(W_i, r \mathbf{T})\ar[d] \ar[r] &W \\
&V^s \ar[r] &\prod_i W_i}
     \]
     We write $f: V \to U$ for the composition $V \to V^s \to U$.
     We take ${T}$ (resp. ${T}_{v, v'}$) to be the projective compactification of $\tilde{T}$ (resp. $\tilde{T}_{v, v'}$) that only adds smooth points.
    Then the morphism $\tilde{\Phi}$ extends to a morphism
    \[
    \Psi: {V}\coprod {T}\coprod (\coprod_{v, v'}{T}_{v, v'}) \to W.
    \] 
    We set $U^0$ to be an open subset of $U^s$ where $f: f^{-1}(U^0) \to U^0$ factors as $f^{-1}(U^0)\to \tilde{f}^{-1}(U^0) \to U^0$.
\end{proof}

\section{Lawson homology and the filtrations on cohomology}\label{sec:complex}
\subsection{Lawson homology}
Let $X$ be a complex projective variety and we fix a very ample line bundle $\OO(1)$.
All the degrees are taken with respect to this line bundle.
Let $\ch_{r, d}(X)$ be the Chow variety parameterizing degree $d$, $r$-dimensional cycles of $X$ and 
\[
\ch_r(X)=\coprod_{d\geq 0} \ch_{r, d}(X),
\]
where $\ch_{r, 0}(X)$ is defined to be a single point corresponding to the  zero-cycle.
We give the set $\ch_r(X)(\CC)$ the structure of a topological monoid, where the topological structure comes from the analytic topology on $\ch_{r, d}(X)(\CC)$ and the monoid structure is the sum of cycles.
Define $Z_r(X)$ to be the group completion of $\ch_r(X)(\CC)$.
It has a topological group structure. 
The topology can be described in several equivalent ways.
These were studied by Lima-Filho \cite{LF_fl_eq_chow}.

\begin{defn}\label{def:eq}
We first define the category $I^\eq$. The objects are pairs $(S, \Gamma)$ consisting of a normal variety $S$ and a family of equi-dimensional $r$-dimensional cycles $\Gamma$ (Definition \ref{def:RelCycle}), and whose morphisms between $(S, \Gamma)$ and $(S', \Gamma')$ are all the morphisms $f: S \to S'$ such that $\Gamma=f^*\Gamma'$. This category is essentially small.

Define the topological space $Z_r(X)^\eq$ as the colimit of all the topological spaces $S(\CC)$ over the category $I^\eq$.

More precisely, each $(S, \Gamma)$ in $I^\eq$ gives a map of sets $\phi_{(S, \Gamma)}: S(\CC) \to Z_r(X)$. The topology of $Z_r(X)^\eq$ is defined in such a way that a subset $T \subset Z_r(X)$ is closed if and only if $\phi_{(S, \Gamma)}^{-1}(T)$ is closed for all $(S, \Gamma)$.
\end{defn}

\begin{lem}\label{lem:smproj}
In the definition of $ Z_r(X)^\eq$, we may take a set consisting of family of equi-dimensional cycles over normal projective varieties  (or smooth projective varieties).
\end{lem}

\begin{proof}
Given any family of equi-dimensional cycles $\Gamma \to S$, we may find a normal projective variety (resp. smooth projective variety) $T$, a family $\Gamma_T \to T$, and an open subset $T^0$ of $T$ such that there is a surjective proper map $p: T^0 \to S$ and $\Gamma_T|_{T^0}$ is $\Gamma \times_S T^0$. 

Note that we have a factorization $T^0(\CC) \to S(\CC) \to Z_r(X)$.
A set in $S(\CC)$ is closed if and only if its inverse image under $p^{-1}$ in $T^0(\CC)$ is closed.
 That is, the topology of $S(\CC)$ is the quotient topology coming from $T^0(\CC) \to S(\CC)$.  

Thus the topology on $Z_r(X)^\eq$ is determined by families over normal varieties (resp. smooth varieties) such that the family has an extension over a normal (resp. smooth) projective compactification.

Therefore, when defining $Z_r(X)^\eq$ as a colimit, we may take only normal (resp. smooth) projective varieties.
\end{proof}

\begin{defn}\label{def:chow}
Define the topological space $Z_r(X)^\ch$ as the quotient of 
\[
\ch_r(X)(\CC) \times \ch_r(X)(\CC)
\]
by $\ch_r(X)(\CC)$, where the action is $(a, b) \mapsto (a+c, b+c)$ for $c \in \ch_r(X) (\CC)$.
\end{defn}

\begin{thm}[\cite{LF_fl_eq_chow}, Theorem 3.1, Theorem 5.2, Corollary 5.4]
The identity map induces a homeomorphism
\[
 Z_r(X)^\eq \cong Z_r(X)^\ch.
\]
\end{thm}

Here is the definition of Lawson homology, first studied in \cite{Lawson_LawsonHomology}.
\begin{defn}\label{def:lawson}
Let $X$ be a complex projective variety. Define the Lawson homology $L_rH_{n+2r}(X)$ as the homotopy group $\pi_n(Z_r(X))$.
\end{defn}

\begin{ex}[Dold-Thom isomorphism]\label{doldthom}
Consider $Z_0(X)$, the group of zero-cycles on $X$. The classical Dold-Thom theorem implies that there is an isomorphism
\[
L_0H_n(X) \cong H_n(X, \ZZ).
\]
\end{ex}

\begin{ex}[Hurewicz map]
The diagonal $\Delta\subset X\times X$ induces an inclusion $X \to Z_0(X)$. The Hurewicz map is induced by this inclusion: 
\[
\pi_k(X) \to \pi_k(Z_0(X)) \cong H_k(X, \ZZ).
\]
\end{ex}

Now we introduce another ingredient.
\begin{lem}\cite[Page 709, 1.2.1]{Friedlander_Mazur_AIF}\label{s-map}
There is a continuous map, the \emph{s-map}: $Z_r(X) \wedge \PP^1 \to Z_{r-1}(X)$ inducing the s-map on Lawson homology $s: L_rH_k(X) \to L_{r-1}H_k(X)$.
\end{lem}
\begin{rem}
The construction of the s-map depends on a deep result: Lawson's algebraic suspension theorem. A geometric way of describing the s-map is the following. Given a cycle $\Gamma$, take a general pencil of divisors $D_t (t \in \PP^1)$ that intersect $\Gamma$ properly, and the s-map sends $([\Gamma], t)$ to the cycle $\Gamma \cdot D_{t}-\Gamma \cdot D_0$.
\end{rem}

\begin{defn}
Let $Y$ be a semi-normal variety. Let $Z \subset Y \times X$ be a family of $r$-dimensional cycle over $Y$ corresponding to a morphism $f: Y \to Z_r(X)$. We define the \emph{correspondence homomorphism}
\[
\Phi_f: H_k(Y, \ZZ) \to H_{k+2r}(X, \ZZ)
\]
as the composition
\[
H_k(Y, \ZZ) \cong \pi_k(Z_0(Y)) \to \pi_k(Z_r(X)) \xrightarrow{s^{\circ k}} \pi_{k+2r}(Z_0(X)) \cong H_{k+2r}(X, \ZZ),
\]
where the map $\pi_k(Z_r(X)) \to \pi_{k+2r}(Z_0(X))$ is induced by $k$-th iterations of the $s$-map.
\end{defn}
\begin{thm}[\cite{Friedlander_Mazur_AIF} Theorem 3.4]\label{thm:corr}
Let $Y$ be a smooth projective variety and $\Gamma \to Y$ be an equi-dimensional family of $r$-cycle over $Y$ corresponding to a morphism $f: Y \to Z_r(X)$.
We have 
\[
\Phi_f=\Gamma_*: H_k(Y, \ZZ) \to H_{k+2r}(X, \ZZ),
\]
where $\Gamma_*$ is the map defined using $\Gamma$ as a correspondence.
\end{thm}

\subsection{Comparing the filtrations}
First we prove the main technical theorem of the section.
\begin{thm}\label{thm:surjectivity}
Let $X$ be a complex smooth projective variety.  Then for any loop $L$ in $Z_1(X)$ (i.e. a continuous map from the circle to $Z_1(X)$), there is a connected projective algebraic set $W$ with a nodal family of one-cycles $W \leftarrow \Gamma \to X$ over $W$ such that the map
\[
\Phi: L_0H_1(W)=\pi_1(Z_0(W)) \to L_1H_3(X)=\pi_1 (Z_1(X))
\] 
induced by the family $\Gamma$ contains the class $[L]$ in $L_1 H_3(X)$.

If  $X$ is separably rationally connected in codimension one, there is an equi-dimensional family $W \leftarrow \Gamma \to X$ over a smooth projective variety $W$ such that the induced map
\[
\Phi: L_0H_1(W)=\pi_1(Z_0(W)) \to L_1H_3(X)=\pi_1 (Z_1(X))
\] 
contains the class $[L]$ in $L_1 H_3(X)$.
\end{thm}

We first introduce some notations. Given a projective algebraic set $S$ parameterizing an equi-dimensional family of one-cycles of $X$, there is an induced continuous map between topological groups:
\[
Z_0(S) \to Z_1(X).
\]
We denote by $I(S)$ the image of this map, i.e. the closed subgroup of $Z_1(X)$ generated by the cycles over $S$, and $K(S)$ the kernel of this map.

The first observation in the proof of Theorem \ref{thm:surjectivity} is the following.
\begin{lem}\label{lem:single}
Let $X$ be a complex smooth projective variety.
For any class $[L]$ in $L_1H_3(X) =\pi_1(Z_1(X))$,
there is a normal projective variety $U$ and a family of equi-dimensional one-cycles $\gamma_U$ over $U$ such that $[L]$ is represented by a continuous map
\[
I=[0, 1] \to U \to Z_1(X).
\]
\end{lem}

Note that in general the map $I \to U$ does not map $0, 1$ to the same point.

\begin{proof}
Denote by $Z_1(X)^0$ the neutral component  of the topological group $Z_1(X)$, i.e., the connected component containing the identity.
We may assume $L$ lies in $Z_1(X)^0$.
Cycles in $Z_1(X)^0$ are precisely the cycles algebraically equivalent to $0$.
By Proposition \ref{prop:filter_1}, the topological group $Z_1(X)^0$ is a filtered colimit over closed subgroups generated by one-dimensional cycles parameterized by normal projective varieties.

Homotopy groups commutes with filtered colimits.
Thus there is an irreducible normal projective variety $S$ with a family of one dimensional cycles over $S$ such that the image of the induced map 
\[
 \pi_1(I(S)) \to \pi_1(Z_1(X)) \cong L_1H_{3}(X) 
\] 
contains the class $[L]$ in $\pi_1(Z_1(X))$.

The fibration $K(S) \to Z_0(S) \to I(S)$ gives a long exact sequence of homotopy groups:
\[
\ldots \to \pi_1(Z_0(S)) \to \pi_1(I(
S)) \to \pi_0(K(S)) \to \ldots.
\]
A loop in $I(S)$ lifts to a continuous map from the unit interval $I=[0, 1]$ to $Z_0(S)$, such that $0, 1$ map to two points in $Z_0(S)$ that parameterize the same cycle in $X$.
 
We may assume that the family over $I$ is the difference of two families of effective $0$-cycles of degree ${(d+)}/{(d-)} $ in $S$. That is, it corresponds to the difference of two continuous maps $f^+: I \to S^{(d+)}, f^-: I \to S^{(d-)}$, which is the same as a continuous map $f: I \to S^{(d+)} \times S^{(d-)}$ with $0$ mapping to a point $x=(x^+, x^-)$ and $1$ mapping to a point $y=(y^+, y^-)$. 

A family of one-cycles over $S$ induces a  family of one-cycles over $S^{(d+)}$ and $S^{(d-)}$. 
Let us denote them by $\Gamma_{d+}, \Gamma_{d-}$.

The loop is the composition $ I \to S^{(d+)} \times S^{(d-)} \to Z_0(S) \to Z_1(X)$, where the middle map is taking the difference.

Let us use a different family of cycles $\pi_+^* \Gamma_{d+}-\pi_{d-}^*\Gamma_{d-}$ on the product
$S^{(d+)} \times S^{(d-)}$, where $\pi_{+/-}$ is the projection to $S^{(d+)/(d-)}$.
This family of cycles induces a continuous map $S^{(d+)} \times S^{(d-)} \to Z_1(X)$ such that the composition $I \to S^{(d+)} \times S^{(d-)} \to Z_1(X)$ is the loop $L$.

We take $U$ to be $S^{(d+)} \times S^{(d-)}$ and $\gamma_U$ to be $\pi_+^* \Gamma_{d+}-\pi_{d-}^*\Gamma_{d-}$.
\end{proof}

\begin{proof}[Proof of Theorem \ref{thm:surjectivity}]
By Lemma \ref{lem:single}, there is a normal projective variety $U$ and an equi-dimensional family of one-cycles $\gamma_U$ over $U$ such that $[L]$ is represented by a continuous map
\[
f: I=[0, 1] \to U \to Z_1(X).
\]
In particular, $[L]$ comes from $\pi_1(I(U))$.
Denote by $L_U$ the class in $\pi_1(I(U))$ represented by $f$.
Denote by $x, y \in U$ the image of $0, 1$ by $f$. The cycle over $x, y$ are the same by assumption. 

\textit{1. Recap of Theorem \ref{thm:filtered_sm} and \ref{thm:filtered}.}

We apply Theorem \ref{thm:filtered_sm} (Theorem \ref{thm:filtered} if $X$ is SRC in codimension one) to the family of cycles $\gamma_U$ over $U$, with $u_0=x, u_1=y$. 
This yields
\begin{itemize}
    \item a normal projective variety $V$;
    \item  a connected projective algebraic set $W$ (a smooth projective variety $W$ if $X$ is SRC in codimension one);
    \item a nodal family (an equi-dimensional family if $X$ is SRC in codimension one) of one-cycles $W \leftarrow \Gamma \to X$ over $W$;
    \item a surjective projective morphism $p: V \to  U$ and a morphism $F: V \to W$;
    \item liftings $x_V, y_V \in V$ of the points $x, y$,
\end{itemize}   such that
\begin{enumerate}
    \item \label{c1} the image $F(x_V)$ and $F(y_V)$ are joined by a connected curve $T_W$ in $W$ parameterizing constant one-cycles, which is the image of the curve $T$ in the statement of Theorem \ref{thm:filtered_sm} (or \ref{thm:filtered} if $X$ is SRC in codimension one);
    
    \item \label{c3} there is a general point $u$ in $U$, such that for any two points  $a, b \in p^{-1}(u)$, $F(a), F(b)$ are joined by a connected curve $C_{a, b}$ in $W$ parameterizing constant one-cycles,  which is the image of the curve $T_{a, b}$ in the statement of Theorem \ref{thm:filtered_sm} (or \ref{thm:filtered} if $X$ is SRC in codimension one). 
\end{enumerate}

Our goal is to prove that $[L]$ lies in the image of 
\[
\pi_1(Z_0(W)) \to \pi_1(Z_1(X))
\]
induced by the nodal/equi-dimensional family of one-cycles $W \leftarrow \Gamma \to X$ over $W$.

\textit{2. The commutative diagram and the obstruction class.}

The morphism $p: V \to U$ induces a continuous surjective map between topological groups $p_*: Z_0(V) \to Z_0(U)$. 
Denote by $K$ the kernel topological group. As a group, $K$ is generated by elements of the form $(a-b)$, where $a, b$ are points in the same fiber of $V \to U$.
It follows that the group $\pi_0(K)$ is generated by classes of the form $[a-b]$ for $a, b$ as above.

Note that $I(V)=I(U)$. Thus we have a fibration sequence of topological groups:
\[
0 \to K \to K(V) \to K(U) \to 0.
\]

We have a commutative diagram of exact sequences, which are part of the long exact sequence of homotopy groups of fibrations:
\[
\begin{CD}
\pi_1(Z_0(W)) @>>> \pi_1(I(W)) @>>> \pi_0(K(W))\\
@AAF_*A @AAF_*A @AAF_*A\\
\pi_1(Z_0(V)) @>>> \pi_1(I(V)) @>>> \pi_0(K(V)) \\
@VVp_*V @VV\cong V @VVp_*V\\
\pi_1(Z_0(U)) @>>> \pi_1(I(U)) @>>> \pi_0(K(U))\\
\end{CD}
\]
Denote by $L_V$ the image of $L_U$ under the isomorphism $\pi_1(I(U)) \cong \pi_1(I(V))$.
Denote by $\mathfrak{o}_V$ (resp. $\mathfrak{o}_U$) the image of the class $L_V$ (resp. $L_U$) in $\pi_0(K(V))$ (resp. $\pi_0(K(U))$).
Clearly $p_*(\mathfrak{o}_V)=\mathfrak{o}_U$ by the above commutative diagram.

Denote by $L_W \in \pi_1(I(W))$ the push-forward class $F_*(L_V)$, and denote by $\mathfrak{o}_W$ the image of $L_W$ in $\pi_0(K(W))$, which is just the push-forward $F_*(\mathfrak{o}_V)$.

Clearly, the class $L_W$ comes from $\pi_1(Z_0(W))$ if and only if $\mathfrak{o}_W=0$. We call $\mathfrak{o}_W$ the obstruction class.

\textit{3. Computation of the obstruction class.}

We first compute the class of $\mathfrak{o}_U$ 
in $\pi_0(K(U))$.

We have already seen that $L_U$ comes from $f:I \to U \subset Z_0(U)$, where the inclusion is induced by the diagonal.  
It follows that its image in $\pi_0(K(U))$ is the class $[f(1)-f(0)]$, i.e. $[y-x]$.

The fibration of topological groups 
\[
0 \to K \to K(V) \to K(U) \to 0.
\]
induces a long exact sequence of homotopy groups, part of which is
\[
\pi_0(K) \to \pi_0(K(V)) \xrightarrow{p_*} \pi_0(K(U)) \to 0.
\]
We have $$p_*([y_V-x_V])=[p(y_V)-p(x_V)]=[y-x]=\mathfrak{o}_U,$$ and thus  
\[
[y_V-x_V]\equiv\mathfrak{o}_V \pmod{\pi_0(K)}.
\]

We claim the following.
\begin{claim}\label{claim:pi_0}
     Let $p: V \to U$ be a generically finite surjective morphism between normal projective complex varieties. Denote by $K$ the kernel topological group of $Z_0(V) \to Z_0(U)$. The connected components of $K$, i.e. $\pi_0(K)$, is finitely generated as an abelian group by classes of the form $[a-b]$, where $a, b$ are points in the fiber over any chosen general point in $U$.
\end{claim} 

We assume this claim for the moment and give the proof afterwards.
This claim implies that 
\[
\mathfrak{o}_W\equiv F_*([y_V-x_V]) \pmod{\langle F_*([a-b])| a, b \in p^{-1}(u)\rangle}.
\]
Here we take the general point $u$ as in (\ref{c3}).

\textit{4. Vanishing of the obstruction class $\mathfrak{o}_W$.}

 We first prove that for the any two points $a, b \in p^{-1}(u)$ as in (\ref{c3}), the push-forward $F_*([a-b])$ is zero. 
 
By (\ref{c3}), there is a family of constant one-cycles over $C_{a, b}\subset W$. 
This induces a continuous map $C_{a, b} \to K(W)$, which maps $c \in C_{a, b}$ to the  $0$-cycle $c-F(b)$ in $W$.  
Since $C_{a, b}$ is connected and $F(a), F(b)\in C_{a, b}$,
\[
F_*([a-b])=[F(a)-F(b)]\cong [F(b)-F(b)]=0 \in \pi_0(K(W)).
\]
So $$\mathfrak{o}_W=F_*(\mathfrak{o}_V)=F_*([y_V-x_V])=[F(y_V)-F(x_V)].$$

 Next we prove that $$F_*([x_V-y_V])=0 \in \pi_0(K(W)).$$
 
 This follows from a similar argument as above. By (\ref{c1}), there is a family of constant one-cycles over $T_W$ joining $F(x_V)$ and $F(y_V)$.   
So we have a continuous map $T_W \to K(W)$, which maps $t \in T_W$ to $t-F(x_V) \in K(W)$.
Since $T_W$ is connected, we have
\[
F_*([y_V-x_V])=[F(y_V)-F(x_V)]\cong [F(x_V)-F(x_V)]=0 \in \pi_0(K(W)).
\]

Combining the above two vanishing results, we have $$\mathfrak{o}_W=F_*(\mathfrak{o}_V)=0.$$
Thus the class $L_W\in \pi_1(I(W))$ comes from $\pi_1(Z_0(W))$.
\end{proof}
As promised, we provide the proof of Claim \ref{claim:pi_0} now. One should compare this to Lemma \ref{lem:move}, which is similar in both the statement and the proof.

First, we treat a special case.
\begin{lem}[Compare with Lemma \ref{lem:flat_moving}]\label{lem:easy_pi_0}
    Let $X \to Y$ be a flat and finite morphism between projective complex varieties, where $Y$ is normal (but $X$ is not necessarily normal). 
Denote by $K$ the kernel topological group of the map $Z_0(X) \to Z_0(Y)$.
Then $\pi_0(K)$ is finitely generated as an abelian group by class of the form $[t_1-t_2]$, for $t_1, t_2$ in the fiber of any chosen general point in $Y$.
\end{lem}

\begin{proof}
Clearly $\pi_0(K)$ is generated by class of the form $[x_1-x_2]$ for all the pairs of points $x_1, x_2 \in X$ with the same image in $Y$.
We will show that for any chosen general point $t \in Y$, there is a continuous path in $K$ connecting $[x_1-x_2]$ to a point $[t_1-t_2]$ for some points $t_1, t_2$ in the fiber over $t$. 

Consider the correspondence $X \times_Y X\subset X \times X$.
Since $X \to Y$ is assumed to be flat and finite, $X \times_Y X \to X$ is flat and finite.
We take an irreducible component $D$ containing $(x_1, x_2)$, which dominates (and thus surjects onto) $X$. 
Denote by $f, f_1, f_2$ the morphisms $D \subset X \times_Y \times X \to Y, D \subset X \times_Y X \xrightarrow{p_1} X, D \subset X \times_Y X \xrightarrow{p_2} X$, where $p_1, p_2$ are projections to the first and second factor.
 There are two points $x_D, t_D$ in $D$ such that \[
 f(x_D)=y, f(t_D)=t, f_1(x_D)=x_1, f_2(x_D)=x_2.
 \] 
This implies that we can find a continuous path in $K$ starting from $[x_1- x_2]$ and ending at $[f_1(t_D)-f_2(t_D)]$, Therefore, the class $[x_1-x_2]$ is the same  as $[f_1(t_D)-f_2(t_D)]$.
\end{proof}

\begin{proof}[Proof of Claim \ref{claim:pi_0}]
    By the flattening lemma \ref{lem:flattening}, 
    there is a birational projective  morphism $U' \xrightarrow{q} U$ such that the strict transform $V'$ of $V$ in $V \times_U U'$ is flat over $U'$. That is, we have a commutative diagram:
\[
\xymatrix{
V' \ar[rd]_{p'} \ar@{^{(}->}[r]_i \ar@/^1pc/[rr]^{q'}
&V\times_U U' \ar[d]\ar[r]
&V \ar[d]^p\\
&U' \ar[r]^q &U
}
\]
where the square is cartesian, $i$ is a closed immersion, $q, q'$ are birational and projective, and $p'$ is flat.
Note that flatness of $p'$ implies that it is quasi-finite and thus also finite.

We may replace $U'$ with its normalization and replace $V'$ with the base change of $V'$ along the normalization of $U'$. So in the following we assume that $U'$ is normal.

We use the same notations $p, p'$ etc. to denote the induced map on the topological group of zero-cycles (e.g. $p:Z_0(V) \to Z_0(U)$). We use $K(p)$ etc. to denote the kernel topological group of $p: Z_0(V) \to Z_0(U)$ etc.. 
We have a commutative diagram
\[
\begin{CD}
    \pi_1(Z_0(U')) @>>> \pi_0(K(p'))@>>> \pi_0(Z_0(V')) @>>> \pi_0(Z_0(U'))\\
    @VVq_*V @VVq_*'V @V\cong Vq_*' V @V\cong Vq_*V\\
    \pi_1(Z_0(U)) @>>> \pi_0(K(p))@>>> \pi_0(Z_0(V)) @>>> \pi_0(Z_0(U))\\
\end{CD}
\]

We first show that 
\[
q_*: \pi_1(Z_0(U'))  \to \pi_1(Z_0(U)) 
\]
is surjective. 
For this, it suffices to show that $\pi_0(K(q))$ is trivial.
 The group $\pi_0(K(q))$ is generated by classes of the form $[a-b]$ for $a, b$ points in a fiber of $q$. Since $q$ is a birational projective morphism and $U$ is normal, $q$ has connected fibers. We conclude that the class $[a-b]$ is connected to $0=[b-b]$ by a continuous path in $K(q)$. That is, $K(q)$ is connected and $\pi_0(K(q))$ is trivial as desired.

Since the Dold-Thom theorem (Example \ref{doldthom}) gives isomorphisms
\[
\pi_0(Z_0(U')) \xrightarrow{\cong} H_0(U', \ZZ) \cong \ZZ,  \pi_0(Z_0(U))\xrightarrow{\cong}H_0(U', \ZZ) \cong \ZZ, 
\]
\[
\pi_0(Z_0(V')) \xrightarrow{\cong}H_0(V', \ZZ) \cong \ZZ, \pi_0(Z_0(V)) \xrightarrow{\cong}H_0(V, \ZZ) \cong \ZZ, 
\]
and since these isomorphisms are compatible with push-forward along the continuous maps $q': Z_0(V') \to Z_0(V), q: Z_0(U') \to Z_0(U)$,
we conclude that
\[
\pi_0(K(p')) \to \pi_0(K(p))
\]
is surjective by a diagram chasing for the above commutative diagram.
Thus the statement follows from Lemma \ref{lem:easy_pi_0}.
\end{proof}

Now we can prove the following results over complex numbers.

\begin{thm}\label{thm:image}
Let $X$ be a complex smooth projective variety of dimension $d$. 
There is a projective curve $C$ with a nodal family of one-cycles $\Gamma \subset C \times X$ inducing a map $f: C \to Z_1(X)$ such that
\[
\Phi_f: H_1(C, \ZZ) \to H_{3}(X, \ZZ) 
\]
has the same image as the s-map $s: L_1H_{3}(X) \to H_{3}(X, \ZZ)\cong H^{2d-3}(X, \ZZ)$, which is the same as the subgroup $N^{d-2}H^{2d-3}(X, \ZZ)$.

If $X$ is separably rationally connected in codimension one, the following two subgroups on $H^{2d-3}(X, \ZZ)$ agree
\[\tilde{N}^{d-2}H^{2d-3}(X, \ZZ)=N^{d-2}H^{2d-3}(X, \ZZ).\]
\end{thm}
\begin{proof}
Recall that there is an isomorphism $L_0H_k(S) \cong H_k(S)$ for any projective algebraic set $S$ by the Dold-Thom theorem.
We have a commutative diagram:
\[
\begin{CD}
L_0H_1(Y) @>f_*>> L_1H_3(X)\\
@V\cong VV @V s VV\\
H_1(Y) @>\Phi_f>> H_3(X)
\end{CD}
\]

The image of the s-map is finitely generated. Thus we may find finitely many  projective algebraic sets $Y_i$ and families of semistable curves $\Gamma_i \to Y_i \times X$ such that the generators of the image of s-map are contained in the image of the correspondence homomorphisms $\Phi_{i*}: H_1(Y_i) \to H_3(X)$.
Then we take $Y$ to be the product $\Pi_i Y_i$ and $\Gamma=\sum_i \pi_i^* \Gamma_i$.
Clearly $\Phi_*$ contains the image of all the $\Phi_{i*}$.

By taking general hyperplane sections containing   all the singularities and  all the one dimensional irreducible components, we may find a  projective curve $C \subset Y$ such that $\pi_1(C_1) \to \pi_1(Y)$ is surjective. Then we simply restrict the family of cycles to $C$. 

If $X$ is separably rationally connected in codimension one, we may take the families of cycles over smooth projective varieties which generate the image of s-map by Theorem \ref{thm:surjectivity}. 
We then take $C$ to be a one-dimensional general complete intersection of very ample divisors, which is a smooth projective curve.

Finally, the image of the s-map $s: L_1H_3(X) \to H_3(X)$ is $N^{d-2}H_3(X, \ZZ)$ by \cite[Proposition 2.8]{Walker_Morphic_AJ}. Then by Lemma \ref{lem:inclusion}, we have the equality of the two subgroups on $H^{2d-3}(X, \ZZ)$.
\end{proof}

On any smooth threefold $X$, the quotient $H^3(X, \ZZ)/N^1H^3(X, \ZZ)$ is torsion free (Corollary \ref{cor:torsionfree}). 
On the other hand, we know that it is torsion if $X$ is SRC in codimension one (Corollary \ref{cor:2d-3torsion}). So we have the following.
\begin{thm}\label{thm:surj_dim3}
Let $X$ be a complex smooth projective $3$-fold that is separably rationally connected in codimension one. Then the following subgroups of $H^{3}(X, \ZZ)$ introduced in Definition \ref{def:filtration} equal the whole cohomology group:
\[
 \tilde{N}_{1, \cyl}H^{3}(X, \ZZ) = \tilde{N}^{1}H^{3}(X, \ZZ) =N^1H^{3}(X, \ZZ)=H^{3}(X, \ZZ).
\]
\end{thm}

\section{Chow sheaves and the filtrations on cohomology}\label{sec:general}
In this section we work over an algebraically closed field $k$. 

Sometimes we may ``invert $p$", by taking the tensor product of a sheaf with $\ZZ[\frac{1}{p}]$. In this scenario, we understand that $p$ is the exponential characteristic of the field.
We use $\ell$-adic \'etale cohomology for $\ell$ a prime number, non-zero in the field $k$. 

One can also define Lawson homology with $\ZZ_\ell$-coefficients in this context \cite{Friedlander_LawsonHomology}. But the construction of the analogue of $Z_r(X)$ is more complicated. 
Also Lawson homology in this context is much less studied.
Many of the known results for complex varieties have not been explicitly stated to hold, even though one can imagine that they are still true.
For example, the author does not know a reference for the construction of the s-map. 
Neither could the author find the analogue of Friedlander-Mazur's result (Theorem \ref{thm:corr}) explicitly stated.
If one had developed all the necessary results in this general context, presumably the argument in last section works in a similar way.

So we use another approach, i.e. the theory of Chow sheaves as introduced by Suslin and Voevodsky in \cite{SV_Chow_Sheaves}, for the general case.

\subsection{Chow sheaves}

\begin{defn}
A finite correspondence from a scheme $Y$ to a scheme $X$ is an equi-dimensional family of relative cycles of dimension $0$ with proper support over $Y$.
\end{defn}
\begin{defn}
Let $\text{Sch}_k$ be the category of finite type separated $k$-schemes.
Let $\text{Cor}_k$ be category whose objects are separated finite type $k$-schemes and morphisms finite correspondences.
Let $\text{SmCor}_k$ be the full subcategory whose objects are smooth.
In this subcategory a finite correspondence between from $X$ to  $Y$ is a linear combination of closed subvarieties of $X \times Y$ that are finite surjective onto one of the irreducible components of $X$.
A \emph{presheaf with transfers} is a contravariant functor on $\text{SmCor}_k$.
\end{defn}
\begin{defn}[{\cite[Section 4.1]{SV_Chow_Sheaves}}]
    A morphism of schemes $p:X\to Y$ is called a \emph{topological epimorphism} if the underlying Zariski topological space of $Y$ is a quotient space of the underlying Zariski topological space of $X$ (i.e. $p$ is surjective and a subset $A$ of $Y$ is open if and only if $p^{-1}(A)$ is open in $X$), and $p$ is called a \emph{universal topological epimorphism} if for any scheme $Z$ over $Y$, the morphism $p_Z : X \times_Y Z \to Z$ is a topological epimorphism.

An h-cover of a scheme $X$ is a finite family of morphisms of finite type $\{p_i : X_i \to X\}$ such that $\coprod p_i: \coprod X_i \to X$ is a universal topological epimorphism.

A qfh-cover is an h-cover that is also quasi-finite.

The qfh topology (resp. h topology) is the topology generated by qfh-covers (resp. h-covers).
\end{defn}

Later we will only use the fact that a surjective proper morphism is an h-cover (since proper morphisms are universally closed). 

\begin{defn}
We define the presheaf $Z_\eq(X, r)$ on the category $\text{Sch}_k$ whose value on  a scheme $S$ is the group of equi-dimensional families of cycles in $X$ of dimension $r$ over $S$ (Definition \ref{def:RelCycle}).

We define $Z_r(X)$ as the h-sheaf associated to $Z_\eq(X, r)$.
\end{defn}

Given a presheaf $\mathcal{F}$ on $\text{Sch}_k$, we use $C_*(\mathcal{F})$ to denote the Suslin complex of presheaves (with non-positive degrees) \cite[Definition 2.14]{Mazza_Voevodsky_Weibel_Lecture_motivic_cohomology}.
That is, for each $i\geq 0$, the degree $-i$ term of the complex is $$C_{i}(\mathcal{F})(S)=\mathcal{F}(S \times \Delta^i),$$ where $\Delta^i=\SP k[t_0, \ldots, t_i]/\langle \sum_j t_j=1 \rangle$ is the algebraic $i$-simplex.

If $\mathcal{F}$ is a torsion, homotopy invariant \'etale sheaf with transfers, or a qfh, or h-sheaf on the category of schemes over $k$, with torsion order prime to $p$, the Suslin rigidity theorem (\cite[Theorem 4.5]{Suslin_Voevodsky_homology}, \cite[Theorem 7.20]{Mazza_Voevodsky_Weibel_Lecture_motivic_cohomology}) implies that $\mathcal{F}$ is a locally constant sheaf. Since we work over an algebraically closed field, locally constant is the same as constant. For any torsion h-sheaf (resp. qfh-sheaf, \'etale sheaf with transfer) $\mathcal{G}$ with torsion order prime to $p$, the presheaf $X \mapsto h^{-i}(C_*\mathcal{G}(X)), i\geq 0,$ of the cohomology of its Suslin complex $C_*(\mathcal{G})$ is homotopy invariant (\cite[Corollary 7.5]{Suslin_Voevodsky_homology}, \cite[Corollary 2.19]{Mazza_Voevodsky_Weibel_Lecture_motivic_cohomology}), and thus the Suslin complex $C_*(\mathcal{G})$ is quasi-isomorphic to a complex of constant h-sheaves (resp. qfh-sheaves, \'etale sheaves). 
Moreover if $\mathcal{F}$ is a constant \'etale sheaf, we have isomorphisms (\cite[Theorem 10.2, 10.7]{Suslin_Voevodsky_homology})
\[
H^i_\text{\'et}(X, \mathcal{F}) \cong H^i_\text{qfh}(X, \mathcal{F^{\text{qfh}}}) \cong H^i_\text{h}(X, \mathcal{F^\text{h}}).
\]
Since we assume that $k$ is algebraically closed, $\SP k$ has no higher cohomology for any sheaf in any of these three topologies. 
Therefore for any complex $K^*$ that is quasi-isomorphic to a complex of constant sheaves, we also have the isomorphism of hypercohomology for $X=\SP k$ by the degeneration of the spectral sequence $$E_2^{p, q}=H^p(\SP k, H^q(K^*))\implies \mathbb{H}^{p+q}(\SP k, K^*)$$ at $E_2$-page.
In  particular, above discussions apply to  the complex $C_*(Z_\eq(X, r))\otimes \ZZ/N \ZZ$.
We may identify the cohomology of this complex.
\begin{thm}\label{thm:computation}
Let $X$ be a  quasi-projective variety defined over an algebraically closed field $k$. Let $N$ be an integer, non-zero in the field. We have the following isomorphisms for each non-negative integer $i$.
\begin{align*}
    &H^{-i}_\text{h}(\SP k, C_*(Z_\eq^\text{h}(X, r)\otimes \ZZ/N \ZZ)) \cong H^{-i}_\text{qfh}(\SP k, C_*(Z_\eq^\text{qfh}(X, r)\otimes \ZZ/N \ZZ))\\
\cong &H^{-i}_\text{\'et}(\SP k, C_*(Z_\eq(X, r)\otimes \ZZ/N \ZZ)) \cong H^{-i}_\text{Ab}(C_*(Z_\eq(X, r)\otimes \ZZ/N \ZZ)(\SP k))\\
\cong &\text{CH}_{r}(X, i, \ZZ/N \ZZ).
\end{align*}
\end{thm}

Note that the Suslin complex is a cochain complex in negative degrees, while the usual convention for higher Chow groups is homological in positive degrees $i$. 
From now on, we write hypercohomology the same way as cohomology.
\begin{proof}
The first three cohomology groups are isomorphic as discussed above. They are all equal to the cohomology of the complex $C_*(Z_\eq(X, r)\otimes \ZZ/N \ZZ)(\SP k))$, since this complex computes the cohomology group in the qfh and \'etale topology.

Finally,  under the hypothesis that resolution of singularities holds, Suslin \cite[Theorem 3.2]{Suslin_Higher_Chow} proves that for any quasi-projective variety $X$, there is a quasi-isomorphism between 
$C_*(Z_\eq(X, r)\otimes \ZZ/N \ZZ)(\SP k)$ and the complex defining higher Chow groups (Definition \ref{def:higherchow}), which induces
an isomorphism
\[
H^{-i}_\text{Ab}(C_*(Z_\eq(X, r)\otimes \ZZ/N \ZZ)(\SP k))
\cong \text{CH}_{r}(X, i, \ZZ/N \ZZ).
\]
Using Gabber's refinement of de Jong's alteration theorem, Kelly \cite[Theorem 5.6.4]{Kelly} removed the resolution of singularities hypothesis.
\end{proof}

\begin{cor}\label{cor:finiteness}
Let $X$ be a  quasi-projective variety defined over an algebraically closed field $k$. Let $N$ be an integer, non-zero in the field. We have the following isomorphisms. 
\begin{align*}
    &H^{-i}_\text{\'et}(\SP k, C_*(Z_\eq(X, 0)\otimes \ZZ/N \ZZ)) \cong H^{-i}_\text{qfh}(\SP k, C_*(Z_\eq(X, 0)\otimes \ZZ/N \ZZ))\\
\cong &H^{-i}_\text{h}(\SP k, C_*(Z_\eq(X, 0)\otimes \ZZ/N \ZZ)) \cong H^{-i}_\text{Ab}(C_*(Z_\eq(X, 0)\otimes \ZZ/N \ZZ)(\SP k))\\
\cong &\text{CH}_{0}(X, i, \ZZ/N \ZZ) \cong H_{i}^{\text{BM}}(X, \ZZ/N \ZZ).
\end{align*}
In particular, all the groups are finite. 
\end{cor}

\begin{proof}
The last equality follows from Corollary \ref{cor:ch_BM}. Clearly the Borel-Moore homology group is finite.
\end{proof}

\subsection{Some easy lemmas}
We first make one simple but useful remark. 
\begin{rem}\label{rem:ch0}
    Let $X$ be a connected projective algebraic set over an algebraically closed field. Then the degree map gives an isomorphism $\deg: CH_0(X)\otimes \ZZ/N\cong \ZZ/N$. 
To see this, one simply note that any two points are algebraically equivalent and $CH_0(X)_\text{alg}$ is divisible. 
Moreover, the class of any point generates this Chow group.  
So any morphism $p: X \to Y$ between projective algebraic sets induces an isomorphism $p_*:CH_0(X)\otimes \ZZ/N\cong CH_0(Y)\otimes \ZZ/N$.

\end{rem}
\begin{lem}\label{lem:curve_vanishing}
Let $T$ be a connected projective algebraic set over an algebraically closed field $k$, and let $x, y$ be two points in $T$. Let $\mathcal{F}$ be a sheaf of abelian groups in the qfh or h topology, or an \'etale sheaf with transfers.
Fix an integer $N$ invertible in $k$.
Write $F_x$ (resp. $F_y$) for the pull-back of $F \in \mathcal{F} \otimes \ZZ/N \ZZ (T)$ to $x$ (resp. $y$) via the inclusion. Then $[F_x]=[F_y]$ in $H^0(\SP k, C_*(\mathcal{F})\otimes \ZZ/N \ZZ)$, where the cohomology is taken in the \'etale topology, qfh topology or h topology.
\end{lem}
\begin{proof}
Every element in $\mathcal{F} \otimes \ZZ/N\ZZ (T)$ induces a unique morphism of presheaves
\[
Z_\eq(T, 0) \otimes \ZZ/N \to \mathcal{F}\otimes \ZZ/N \ZZ.
\]
If $\mathcal{F}$ is an \'etale sheaf with transfers, this is the Yoneda Lemma. 
If $\mathcal{F}$ is a qfh-sheaf or h-sheaf, this follows from Yoneda Lemma and the fact that the qfh sheafification of $Z_\eq(T, 0)[\frac{1}{p}]$ is the free sheaf $\ZZ [\frac{1}{p}] [T]$ generated by the presheaf of sets $\text{Hom}(\cdot, T)$ (\cite[Theorem 6.7]{Suslin_Voevodsky_homology}).

In the following, we consider the case that $\mathcal{F}$ is an h-sheaf. 
We have induced morphism of h-sheaves:
\[
Z_0(T)\otimes \ZZ/N\ZZ=Z_\eq^h(T, 0) \otimes \ZZ/N \to \mathcal{F}\otimes \ZZ/N \ZZ,
\]

The class $[\mathcal{F}_x]$ (resp. $[\mathcal{F}_y]$) is the image of $[x]$ (resp. $[y]$) under the map
\[
H^0(\SP k, C_*(Z_0(T))\otimes \ZZ/N \ZZ) \to H^0(\SP k, C_*(\mathcal{F})\otimes \ZZ/N \ZZ).
\]
So it suffices to show that $[x]=[y]$ in $H^0(\SP k, C_*(Z_0(T)\otimes \ZZ/N \ZZ))$.
But the latter cohomology group is $CH_0(T) \otimes \ZZ/N \ZZ \cong \ZZ/N\ZZ$ by Corollary \ref{cor:finiteness} and the isomorphism is given by the degree map (Remark \ref{rem:ch0}).
Any two points $x, y$ give the same class in $H^0(\SP k, C_*(Z_0(T))\otimes \ZZ/N \ZZ)$.

To prove the other cases, one only needs to replace the sheaf $Z_0(T)$ by the sheafification of $Z_\eq(T, 0)$ in the corresponding topology (in fact, $Z_\eq(T, 0)\otimes \ZZ[\frac{1}{p}]$ is already a qfh, and thus \'etale, sheaf \cite[Proposition 4.2.5]{SV_Chow_Sheaves}) and the cohomology group in h topology by the cohomology group in the corresponding topology.
\end{proof}

\begin{lem}\label{lem:flat_moving}
Let $X \to Y$ be a flat and finite morphism defined over an algebraically closed field $k$, where $Y$ is a normal variety (but $X$ is not necessarily normal). Let $N$ be an integer invertible over $k$.
Denote by $K$ the kernel h-sheaf of $Z_0(X) \to Z_0(Y)$.
Then $H_h^{0}(\SP k, C_*(K)\otimes \ZZ/N\ZZ)$ is generated by classes of the form $[t_1-t_2]$, for $t_1, t_2$ in the fiber of any chosen general point $t$ in $Y$.
\end{lem}

\begin{proof}
Clearly $H_h^{0}(\SP k, C_*(K)\otimes \ZZ/N\ZZ)$ is generated by classes of the form $[x_1-x_2]$ for all the pairs of points with the same image in $Y$.
Let $x_1, x_2$ be two points in the fiber over $y \in Y$.
We will show that for any chosen general point $t \in Y$, the class $[x_1-x_2]$ is equivalent to a class $[t_1-t_2]$ for some points $t_1, t_2$ in the fiber over $t$. 

Consider the correspondence $X \times_Y X\subset X \times X$.
Since $X \to Y$ is assumed to be flat and finite, $X \times_Y X \to X$ is flat and finite.
We take an irreducible component $D \subset X\times_Y X$ containing $(x_1, x_2)$, which dominates (and thus surjects onto) $X$. 
 Denote by $f, f_1, f_2$ the morphisms $D \subset X \times_Y \times X \to Y, D \subset X \times_Y X \xrightarrow{p_1} X, D \subset X \times_Y X \xrightarrow{p_2} X$, where $p_1, p_2$ are projections to the first and second factor.
 There are two points $x_D, t_D$ in $D$ such that \[
 f(x_D)=y,\quad f(t_D)=t,\quad f_1(x_D)=x_1,\quad f_2(x_D)=x_2.
 \] 
 We have an element $D_1-D_2\in K(D)$ where $D_i=\{(d, f_i(d))| d \in D\}\subset D \times X$, $i=1, 2$.
Then by Lemma \ref{lem:curve_vanishing}, the class $[x_1-x_2]=[f_1(x_D)-f_2(x_D)]$ is the same  as $[f_1(t_D)-f_2(t_D)]$.
\end{proof}

\begin{lem}\label{lem:move}
Let $p: X \to Y$ be a generically finite surjective morphism between normal projective varieties over an algebraically closed field $k$.
 Let $N$ be an integer invertible over $k$.
Denote by $K$ the kernel h-sheaf of $p_*: Z_0(X) \to Z_0(Y)$.
Then the cohomology group $H_h^{0}(\SP k, C_*(K)\otimes \ZZ/N\ZZ)$ is generated by classes of the form $[t_1-t_2]$, for $t_1, t_2$ in the fiber of any chosen general point in $Y$.
\end{lem}
\begin{proof}
By the flattening lemma \ref{lem:flattening}, there is a birational projective  morphism $Y' \xrightarrow{q} Y$ such that the strict transform $X'$ of $X$ in $X \times_Y Y'$ is flat over $Y'$. That is, we have a commutative diagram:
\[
\xymatrix{
X' \ar[rd]_{p'} \ar@{^{(}->}[r]_i \ar@/^1pc/[rr]^{q'}
&X\times_Y Y' \ar[d]\ar[r]
&X \ar[d]^p\\
&Y' \ar[r]^q &Y
}
\]
where the square is cartesian, $i$ is a closed immersion, $q, q'$ are birational and projective, and $p'$ is flat.
Note that flatness of $p'$ implies that it is quasi-finite and thus also finite (since it is proper).
Up to replacing $Y'$ with its normalization and $X'$ with its base change along the normalization of $Y'$, we may assume that $Y'$ is normal.

We use the same notations $p, p'$ etc. to denote the induced map on the h-sheaves of zero-cycles (e.g. $p:Z_0(X) \to Z_0(Y)$).
We use $K(p)$ etc. to denote the kernel h-sheaf of $p: Z_0(X) \to Z_0(Y)$ etc.. 
There is a commutative diagram of short exact sequences of h-sheaves
\[
\begin{CD}
0 @>>> K(p') @>>> Z_0(X') @>p'>> Z_0(Y') @>>> 0\\
@VVq'V @VVV @Vq'VV @VqVV @VVV\\
0 @>>> K(p) @>>> Z_0(X) @>p>> Z_0(Y) @>>> 0,
\end{CD}
\]
which also gives a commutative diagram of exact sequences after tensoring with $\ZZ/N\ZZ$, since the sheaves $Z_0(X)$ etc. are torsion free.
Then we have a commutative diagram of long exact sequences, part of which is the following:
\[
\begin{CD}
     CH_0(Y', 1) @>>> H_h^{0}(\SP k, C_*(K(p'))/N)
    @>>> CH_0(X')  @>p_*'>> CH_0(Y')\\
    @VVq_* V @VVq_*' V @V\cong Vq_*'V @V\cong Vq_* V\\
     CH_0(Y, 1) @>>> H_h^{0}(\SP k, C_*(K(p))/N)
    @>>>CH_0(X)  @>p_*>> CH_0(Y),
\end{CD}
\]
where all the (higher) Chow groups are
with $\ZZ/N$-coefficients.

We first show that 
\[
q_*: CH_0(Y', 1, \ZZ/N\ZZ)  \to CH_0(Y, 1, \ZZ/N\ZZ) 
\]
is surjective. For this, it suffices to show that $H_h^{0}(\SP k, C_*(K(q))\otimes \ZZ/N\ZZ)$ vanishes.
The morphism $Y' \to Y$ has connected fibers. 
So for any two points $y_1', y_2'$ in the same fiber of $Y' \to Y$, by Lemma \ref{lem:curve_vanishing}, the class of the difference $[y_1'-y_2']$ is the same as $[y_1'-y_1']=0$.

Since
\[
CH_0(X', 0, \ZZ/N)  \xrightarrow{q_*'} CH_0(X, 0, \ZZ/N\ZZ), \quad
CH_0(Y', 0, \ZZ/N)  \xrightarrow{q_*} CH_0(Y, 0, \ZZ/N\ZZ) 
\]
are isomorphisms (Remark \ref{rem:ch0}), a simple diagram chasing of the above commutative diagram shows that
\[
H_h^{0}(\SP k, C_*(K(p'))\otimes \ZZ/N\ZZ) \to H_h^{0}(\SP k, C_*(K(p))\otimes \ZZ/N)
\]
is surjective.
Thus the statement follows from Lemma \ref{lem:flat_moving}.
\end{proof}

\begin{lem}\label{lem:}
Let $p: X \to Y$ be a generically finite surjective morphism between normal projective varieties over an algebraically closed field $k$.
 Let $N$ be an integer invertible over $k$.
Assume that $\deg p$ is relatively prime to $N$. Then we have a surjection
\[
p_*: CH_0(X, 1, \ZZ/N \ZZ) \to CH_0(Y, 1, \ZZ/N \ZZ).
\]
Moreover, denote by $K$ the kernel h-sheaf of $Z_0(X) \to Z_0(Y)$.
We have the vanishing of cohomology: $H^0_h(\SP k, C_*(K)/N)=0$.
\end{lem}
\begin{proof}
We have the long exact sequence
\begin{align*}
    &CH_0(X, 1, \ZZ/N \ZZ) \xrightarrow{p_*} CH_0(Y, 1, \ZZ/N \ZZ) \to H^0_h(\SP k, C_*(K)/N)\\
    \to &CH_0(X, 0, \ZZ/N) \xrightarrow[\cong]{p_*} CH_0(Y, 0, \ZZ/N).
\end{align*}
Since $p_*$ induces an isormorphism $CH_0(X, 0, \ZZ/N) \xrightarrow[\cong]{p_*} CH_0(Y, 0, \ZZ/N)\cong \ZZ/N$ (Remark \ref{rem:ch0}), the two statements in the lemma are equivalent.

By Lemma \ref{lem:move}, it suffices to show that for a general point $y \in Y$ and any two points $x_1, x_2$ in the fiber of $y$, the class $[x_1-x_2]$ is zero in $H^0_h(\SP k, C_*(K)/N)$.

By the Lefschetz theorem for \'etale fundamental groups, there is a general complete intersection curve $H$ such that the inverse image $H'$ in $X$ is irreducible. 
For $H$ general, the morphism $H' \to H$ is flat and finite of degree prime to $N$.
We have a long exact sequence
\begin{align*}
&CH_0(H', 1, \ZZ/N\ZZ)  \to CH_0(H, 1, \ZZ/N\ZZ) \to H_h^{0}(\SP k, C_*(K_H)\otimes \ZZ/N\ZZ)\\
  &\to CH_0(H', \ZZ/N\ZZ) \to CH_0(H, \ZZ/N\ZZ),
\end{align*}
where $K_H$ is the kernel h-sheaf of $Z_0(H') \to Z_0(H)$.
The map
\[
CH_0(H', \ZZ/N\ZZ) \to CH_0(H, \ZZ/N\ZZ) 
\]
is an isomorphism (Remark \ref{rem:ch0}).
On the other hand, since $p: H'\to H$ is flat and finite, we have pull-back and push-forward on all the higher Chow groups. 
The composition of pull-back and push-forward
\[
CH_0(H, 1, \ZZ/N\ZZ) \xrightarrow{p^*} CH_0(H', 1, \ZZ/N\ZZ)  \xrightarrow{p_*} CH_0(H, 1, \ZZ/N\ZZ)
\]
is multiplication by $\deg p$.
Since the degree of the map is relatively prime to $N$,
\[
 CH_0(H', 1, \ZZ/N\ZZ)  \xrightarrow{q_*} CH_0(H, 1, \ZZ/N\ZZ)
\]
is surjective.
Thus for any two points $t_1, t_2$ over a general point $t \in H \subset Y$, the class $[t_1-t_2]$ vanishes in $ H_h^{0}(\SP k, C_*(K_H)\otimes \ZZ/N\ZZ)$. So does its push-forward in  $H_h^{0}(\SP k, C_*(K)\otimes \ZZ/N\ZZ)$.
\end{proof}

\subsection{The key exact sequence}
Recall that we use $A_i(X)$ to denote the group of $i$-dimensional cycles in $X$ modulo algebraic equivalence (Definition \ref{def:alg_equiv}). 

For any integer $N$ invertible in the field $k$, 
Lemma \ref{lem:modN} applied to the exact sequence $0 \to \ZZ \xrightarrow{\times N} \ZZ \xrightarrow{\mod N} \ZZ/N \to 0$ gives a long exact sequence:
\begin{align*}
    \ldots \to &\text{CH}_i(X, 1, \ZZ) \xrightarrow{\times N} \text{CH}_i(X, 1, \ZZ) \xrightarrow{\mod N} \text{CH}_i(X, 1, \ZZ/N)\\
    \to &\text{CH}_i(X, 0, \ZZ)\xrightarrow{\times N} \text{CH}_i(X, 0, \ZZ)\ldots.
\end{align*}
So we have a surjective homomorphism
$\text{CH}_i(X, 1, \ZZ/N\ZZ) \to \text{CH}_i(X)[N]$.
Since the group $\text{CH}_i(X)_\text{alg}$ consisting of cycles that are algebraically trivial is divisible, we have a surjective map
$\text{CH}_i(X, 0, \ZZ)[N] \to A_i(X)[N]$.
Composing these two maps, we have a surjective homomorphism for $i=1$:
$\text{CH}_1(X, 1, \ZZ/N \ZZ) \to A_1(X)[N]$.

Given an equi-dimensional family of one-cycles $\Gamma\subset S \times X$ parameterized by a smooth projective variety $S$, we have the following commutative diagram, where the vertical maps are induced by the correspondence $\Gamma_*$:
\[
\xymatrix{
 CH_0(S,1, \ZZ/N)\ar[r]\ar[d]^{{\Gamma}_*} & CH_0(S)[N] \ar@{->>}[r]\ar[d]^{{\Gamma}_*} &A_0(S)[N] \ar[d]^{{\Gamma}_*}\\
  \text{CH}_1(X, 1, \ZZ/N) \ar[r] & \text{CH}_1(X)[N] \ar@{->>}[r] & A_1(X)[N].}
\]
Since $A_0(S)[N]=0$, the image ${\Gamma}_*(CH_0(S, 1, \ZZ/N))$ is contained in the kernel of 
\[
\text{CH}_1(X, 1, \ZZ/N) \to A_1(X)[N].
\]
Now we can state the main technical result of this section, which says that when $X$ is separably rationally connected in codimension one, the kernel of the above map is precisely generated by the image of such correspondences.

\begin{thm}\label{thm:sheaf_surj}
Let $X$ be a smooth projective variety defined over an algebraically closed field $k$ of characteristic $p$.
 Fix an integer $N$ that is invertible in $k$.
For any class $[L]$ in the kernel of the map
\[
H_h^{-1}(\SP k, C_* (Z_1(X))\otimes \ZZ/N\ZZ) \cong \text{CH}_1(X, 1, \ZZ/N) \to A_1(X)[N],
\]
there is a connected projective algebraic set $Z$ and a nodal family of one-cycles over $Z$ such that the class $[L]$ is in the image of 
\[
H_h^{-1}(\SP k, C_*(Z_0(Z))\otimes \ZZ/N \ZZ) \to H_h^{-1}(\SP k, C_* (Z_1(X)) \otimes \ZZ/N \ZZ)
\]
induced by this family of cycles.

Assume furthermore that $X$ is separably rationally connected in codimension one.
We have the following exact sequence
\[
\oplus_{(S, \Gamma_S)} \text{CH}_0(S,1, \ZZ/N)\xrightarrow{\oplus {\Gamma_S}_*} \text{CH}_1(X, 1, \ZZ/N)  \to A_1(X)[N] \to 0,
\]
where the direct sum is taken over isomorphism classes of all smooth projective varieties $S$ and equi-dimensional families of one-cycles $\Gamma_S$.
\end{thm}

\begin{rem}
This result is a priori weaker than Theorem \ref{thm:surjectivity} over complex numbers.
We have a short exact sequence of Lawson homology
\[
0 \to L_1H_3(X)/N \to L_1H_3(X, \ZZ/N) \to L_1H_2(X)[N]\to 0.
\]
Since $L_1H_3(X, \ZZ/N) \cong \text{CH}_1(X, 1, \ZZ/N)$ (\cite[Theorem 9.1]{Suslin_Voevodsky_homology}) and  $L_1H_2(X)[N] \cong A_1(X)[N]$, Theorem \ref{thm:sheaf_surj} only says that for $X$ SRC in codimension one, classes in $L_1H_3(X)/N$ comes from a smooth projective variety.

If we know that $L_1H_3(X)$ is finitely generated, then we can find the lift. Conjecture \ref{conj:tian_lawson} predicts that this group is isomorphic to $H_3(X, \ZZ)$ for $X$ SRC in codimension one, thus finitely generated.
\end{rem}

The proof of Theorem \ref{thm:sheaf_surj} is analogous to that of Theorem \ref{thm:surjectivity}.
We first prove the analogue of Lemma \ref{lem:single}.
\begin{lem}\label{lem:single_alg}
Let $X$ be a smooth projective variety defined over an algebraically closed field $k$ of characteristic $p$. Fix an integer $N$ that is invertible in $k$.
For any class $[L]$ in
\[
H_h^{-1}(\SP k, C_* (Z_1(X))\otimes \ZZ/N\ZZ) \cong \text{CH}_1(X, 1, \ZZ/N \ZZ),
\]
there is a normal projective variety $S$, an equi-dimensional family of one-cycles $\gamma_S$ over $S$ and a morphism $f: \Delta^1 \to S$ such that $[L]$ is represented by $f^*\gamma_S$ over $\Delta^1$.
\end{lem}

\begin{proof}
We could translate the proof of Lemma \ref{lem:single} in the context of h-sheaves.
But here is an easier argument using Hilbert schemes.

The class $[L]$ is represented by a family of cycles $\sum_i m_i \Gamma_i, m_i \in \ZZ/N$ over $\Delta^1$, where $\Gamma_i \subset \Delta^1 \times X$ is an integral subvariety. Since $\Delta^1$ is one dimensional, the projection $\Gamma_i \to \Delta^1$ is flat. 
Thus we get a morphism $f_i$ from  $\Delta^1$ to the Hilbert scheme of $X$.
The universal subscheme over the Hilbert scheme gives a family of cycles over the Hilbert scheme.
Therefore, we may take $S$ to be the normalization of products of irreducible components of the Hilbert scheme and $\gamma_S$ to be the family of cycles (with appropriate multiplicity) coming from universal subschemes.
\end{proof}

Now we begin the proof of Theorem \ref{thm:sheaf_surj}.
\begin{proof}[Proof of Theorem \ref{thm:sheaf_surj}]
\textit{1. Setup and notations.}

Given a projective algebraic set $S$ and an equi-dimensional family of one-cycles  $\Gamma_S$ over $S$, there is an induced morphism of h-sheaves:
$Z_0(S) \to Z_1(X)$.
We denote by $I(S)$ (resp. $K(S)$) the image h-sheaf (resp. kernel h-sheaf) of this map.

By Lemma \ref{lem:single_alg}, there is a normal projective variety $S$, an equi-dimensional family of one-cycles $\gamma_S$ over $U$ and a morphism $f: \Delta^1 \to S$ such that $[L]$ is represented by $f^*\gamma_S$ over $\Delta^1$.

Denote by $\gamma_0, \gamma_1$ the restriction of the family of cycles $\gamma_S$ over $S$ to $0, 1 \in \Delta^1$. 
Then $\gamma_0-\gamma_1=N(\gamma_{0, 1})$ for some cycle $\gamma_{0, 1}$. 
The image of $[L]$ under the map
\[
\text{CH}_1(X, 1, \ZZ/N) \to A_1(X)[N]
\]
is the class  $[\gamma_{0, 1}]\in A_1(X)[N]$ .

If $[\gamma_{0, 1}]$ is zero in $A_1(X)[N]$, that is, if $\gamma_{0, 1}$ is algebraically equivalent to $0$, 
then by Proposition \ref{prop:filter_1}, 
there is a smooth projective curve $D$ with a family of cycles $\gamma_D$ and two points $d_1, d_2\in D$ such that $\gamma_D|_{d_1}$ is $0$ and $\gamma_D|_{d_2}$ is $\gamma_{0, 1}$. 

Consider the product $U=S \times D$.
We have a family of cycles $\gamma=\pi_S^* \gamma_S+N \pi_D^* \gamma_D$.

Denote by $g$ the composition of morphisms $$g:  \Delta^1 \xrightarrow{f} S \cong S \times \{d_1\}\to S \times D=U.$$ Then  $[L]$ is represented by $g^*\gamma$ over $\Delta^1$.
That is, the class $[L]$ is contained in the image of
$$H^{-1}_h(\SP k, C_*(I(U))\otimes \ZZ/N\ZZ) \to H^{-1}_h(\SP k, C_*(Z_1(X)\otimes\ZZ/N))\otimes \ZZ/N\ZZ).$$
Denote by $L_U \in H^{-1}_h(\SP k, C_*(I(U))\otimes \ZZ/N\ZZ)$ the class induced by $g^*\gamma$.

Similar to the proof of Theorem \ref{thm:surjectivity}, our goal is to find a family of cycles connecting the two points $x=g(0), y=g(1)$ that is a family of constant cycles modulo $N$.
But unlike the situation in Theorem \ref{thm:surjectivity}, we cannot directly apply Theorem \ref{thm:filtered_sm}.
This is because the cycles over $x, y$ are only assume to be equal modulo $N$, and thus not necessarily the same as cycles.
So we have to introduce an auxiliary point $z$.

There are three points in $U=S \times D$, 
\begin{equation}\label{xyz}
    x=g(0)=(f(0), d_1), y=g(1)=(f(1), d_1), z=(f(1), d_2)
\end{equation}
such that
\begin{itemize}
   \item $\gamma_x=\gamma_z=\gamma_0$.
    \item There is an irreducible curve $D_U(=\{f(1)\}\times D)$ containing $y, z$ such that for every point $d \in D_U$, the cycle $\gamma_d$ equals $\gamma_1$ in $Z_1(X) \otimes \ZZ/N (\SP k)$ (but $\gamma_d$ may of course vary as one-cycles with $\ZZ$-coefficient). 
\end{itemize}

\textit{2. Recap of Theorem \ref{thm:filtered_sm} and \ref{thm:filtered}.}

We apply Theorem \ref{thm:filtered_sm} (Theorem \ref{thm:filtered} if $X$ is SRC in codimension one) to the family of cycles $\gamma$ over $U$, with $u_0=x, u_1=z$.
So we are in the following situation.
\begin{sit}\label{situ}
There is
    \begin{itemize}
    \item a normal projective variety $V$;
    \item a connected projective algebraic set $W$ (a normal projective variety $W$ if $X$ is SRC in codimension one);
    \item a nodal family of one-cycles $W\leftarrow \Gamma_W \to X$ over $W$ (an equi-dimensional family of one-cycles $W\leftarrow \Gamma_W \to X$ over $W$ if $X$ is SRC in codimension one);
    \item a surjective projective morphism $p: V \to  U$, a morphism $F: V \to W$, and liftings $x_V, y_V, z_V \in V$ of the points $x, y, z \in U$ (defined in (\ref{xyz})),
\end{itemize} 
such that the followings hold.
\begin{enumerate}
    \item[(i)]  The image $F(x_V)$ and $F(z_V)$ are joined by a connected curve $T_W$ in $W$ parameterizing constant one-cycles, which is the image of the curve $T$ in the statement of Theorem \ref{thm:filtered_sm} (Theorem \ref{thm:filtered} if $X$ is SRC in codimension one).
    
    \item[(ii)]  For a general point $u$ in $U$, and for any two points  $a, b \in p^{-1}(u)$, $F(a), F(b)$ are joined by a connected curve $C_{a, b}$ in $W$ parameterizing constant one-cycles,  which is the image of the curve $T_{a, b}$ in the statement of Theorem \ref{thm:filtered_sm} (Theorem \ref{thm:filtered} if $X$ is SRC in codimension one).

    \item[(iii)]  In the algebraic set $p^{-1}(D_U)$, we choose a connected projective curve $D_V$, which surjects onto $D_U$, and which contains the point $z_V$.
    We also choose the lifting $y_V \in D_V$ of $y$. 
    As a result, $F(x_V)$ and $F(y_V)$ are connected by the connected curve $D_W=F(D_V)$.
    By construction, both of the curves $D_V$ and $D_W$ parameterize a families of cycles that becomes constant modulo $N$. 
\end{enumerate}
\end{sit}

    Denote by $K$ the kernel of the morphism between h-sheaves 
$p_*: Z_0(V) \to Z_0(U)$.
Since $p: V \to U$ is proper and surjective, the above morphism of h-sheaves is surjective. 
Then we have an equality of sub-h-sheaves of $Z_1(X)$:
$p_*: I(V) \cong I(U)$.

    We use $L_V$ to denote the image of the class $L_U$ under the isomorphism $$p^{-1}_*:H^{-1}_h(\SP k, C_*(I(U))/N) \xrightarrow{\cong} H^{-1}_h(\SP k, C_*(I(V))/N)$$ induced by the isomorphism $p^{-1}_*:I(U)\cong I(V)$.
Define $$L_W=F_*(L_V) \in H^{-1}_h(\SP k, C_*(I(W))/N).$$
Denote by $\mathfrak{o}_V$ (resp. $\mathfrak{o}_U$, $\mathfrak{o}_W$) the image of the class $L_V$ (resp. $L_U$, $L_W$) under the map $$H^{-1}_h(\SP k, C_*(I(V))/N) \to H^0_h(\SP k, C_*(K(V))/N)$$ (resp. $$H^{-1}_h(\SP k, C_*(I(W))/N) \to H^0_h(\SP k, C_*(K(U))/N),$$
$$H^{-1}_h(\SP k, C_*(I(W))/N) \to H^0_h(\SP k, C_*(K(W))/N)).$$

Our first goal is to prove that $L_W$ lies in the image of 
\[
H_h^{-1}(\SP k, C_*(Z_0(W))\otimes \ZZ/N \ZZ) \to H_h^{-1}(\SP k, C_* (I(W)) \otimes \ZZ/N \ZZ)
\]
induced by the family of cycles $[\Gamma_W]\to W$.
This will prove the part for all smooth projective varieties. 
The case for varieties that are SRC in codimension one requires one further step to be discussed at the end of this proof.

\textit{3. The commutative diagram and the obstruction class.}

Since we have an equality of sub-h-sheaves of $Z_1(X)$:
$p_*: I(V) \cong I(U)$,
there is a short exact sequence of h-sheaves:
\[
0 \to K \to K(V) \xrightarrow{p_*} K(U) \to 0.
\]

We have a commutative diagram of long exact sequences, part of which is the following (where all hypercohomology groups are for $\SP k$):
\[
\begin{CD}
H^{-1}_h( C_*(Z_0(W))/N) @>>> H^{-1}_h( C_*(I(W) /N)) @>>> H^{0}_h( C_*(K(W))/N)\\
@AA F_*A @AAF_*A @AA F_* A\\
H^{-1}_h( C_*(Z_0(V))/N) @>>> H^{-1}_h( C_*(I(V))/N) @>>> H^{0}_h( C_*(K(V))/N) \\
@VV p_* V @V\cong Vp_*V @VV p_* V\\
H^{-1}_h(C_*(Z_0(U))/N) @>>> H^{-1}_h(C_*(I(U))/N) @>>> H^{0}_h(C_*(K(U))/N)\\
\end{CD}
\]

Clearly $p_*(\mathfrak{o}_V)=\mathfrak{o}_U$ and $F_*(\mathfrak{o}_V)=\mathfrak{o}_W$ by the above commutative diagram.

The class $L_W$ comes from $H^{-1}_h(\SP k, C_*(Z_0(W))/N)$ if and only if $\mathfrak{o}_W=0$. We call $\mathfrak{o}_W$ the obstruction class.

\textit{4. Computation of the obstruction class.}

We first compute the class of $\mathfrak{o}_U$ 
in $H^0_h(\SP k, C_*(K(U))/N)$.

We have already seen that $L_U$ comes from $g^*\gamma$ over $\Delta^1$ induced by the morphism $g: \Delta^1 \to U$. That is, it lifts to an element in $C_1(Z_0(U))(\SP k)\otimes\ZZ/N \ZZ$. 
So its image in $H^0_h(\SP k, C_*(K(U))/N)$ is the class $[g(1)-g(0)]$, i.e. $[y-x]$.

The short exact sequence 
$0 \to K \to K(V) \to K(U) \to 0$
induces a long exact sequence of hypercohomology of Suslin complex, part of which is
\[
H^0_h(\SP k, C_*(K)/N) \to H^0_h(\SP k, C_*(K(V))/N) \xrightarrow{p_*} H^0_h(\SP k, C_*(K(U))/N)
\]
We have $$p_*([y_V-x_V])=[p(y_V)-p(x_V)]=[y-x]=\mathfrak{o}_U,$$ and thus  
\[
[y_V-x_V]\equiv \mathfrak{o}_V \pmod{H^0_h(\SP k, C_*(K)/N)}.
\]

By Lemma \ref{lem:move}, the group $H^{0}_h(\SP k, C_*(K)\otimes \ZZ/N\ZZ)$ is generated by classes of the form $[a-b]$, where $a, b$ are points in the fiber over any fixed general point in $U$.
So we have
\[
\mathfrak{o}_W\equiv F_*([y_V-x_V]) \pmod{\langle F_*([a-b])| a, b \in p^{-1}(u)\rangle}.
\]
Here we take the general point $u$ as in (\ref{situ}.(ii)).

\textit{5. Vanishing of the obstruction class $\mathfrak{o}_W$.}

 We first prove that for the any two points $a, b \in p^{-1}(u)$ as in (\ref{situ}.(ii)), the push-forward $F_*([a-b])$ is zero. 
 
By (\ref{situ}.(ii)), there is a family of constant one-cycles over $C_{a, b}\subset W$. 
So we have an element of $K(W)(C_{a, b})$, which corresponds to the family of $0$-cycles (in $W$) $\{[c-F(b)], c \in C_{a, b}\}$ over $C_{a, b}$.  By Lemma \ref{lem:curve_vanishing}, we have
\[
F_*([a-b])=[F(a)-F(b)]= [F(b)-F(b)]=0 \in H^0 _h(\SP k, C_*(K(W))/N).
\]
So $$\mathfrak{o}_W=F_*(\mathfrak{o}_V)=F_*([y_V-x_V])=[F(y_V)-F(x_V)].$$

 Next we prove that $$F_*([x_V-y_V])=0 \in H^0_h(\SP k, C_*(K(W))/N).$$
 
 This follows from a similar argument as above. By (\ref{situ}.(i)), there is a family of constant one-cycles over $T_W$ joining $F(x_V)$ and $F(z_V)$.  
 By (\ref{situ}.(iii)), there is a family of one-cycles over $D_W$ joining $F(y_V)$ and $F(z_V)$. Moreover, this family of one-cycles over $D_W$ are constant modulo $N$.
So we have an element of $K(W)\otimes \ZZ/N (D_W \cup T_W)$, which corresponds to the family of $0$-cycles (in $W$) $\{[w-F(x_V)], w \in T_W\cup D_W\}$ over $T_W\cup D_W \subset W$. 
By Lemma \ref{lem:curve_vanishing} (and the fact that $T_W\cup D_W \subset W$ is a connected closed subset connecting $F(y_V)$ and $F(x_V)$), we have
$F_*([y_V-x_V])=[F(y_V)-F(x_V)]\cong [F(x_V)-F(x_V)]=0$.

Combining the above two vanishing results, we have $$\mathfrak{o}_W=F_*(\mathfrak{o}_V)=0.$$
Thus the class $L_W$ comes from $H^{-1}_h(\SP k, C_*(Z_0(W))/N).$

For a smooth projective variety $X$, we take $Z$ in the statement of the theorem to be $W$ and use the nodal families of one-cycles $W \leftarrow \Gamma_W \to X$ on $W$. 
This proves the statement about a smooth projective variety.

\textit{6. The case of varieties that are SRC in codimension one.}

If $X$ is separably rationally connected in codimension one, the previous argument only proves that the variety $W$ is normal and projective. 
To find a smooth projective variety as required in the statement, we use Gabber's refinement of de Jong's alteration and proceed in the following way.

Let $N=\prod_i \ell_i^{m_i}$ be the factorization of $N$ into distinct prime factors. 
Denote by $[L_i]$ the class of $[L]$ in $\text{CH}_1(X, 1, \ZZ/\ell_i^{m_i})$ under the decomposition
$$\text{CH}_1(X, 1, \ZZ/N)\cong \oplus_i \text{CH}_1(X, 1, \ZZ/\ell_i^{m_i}).$$
For each prime factor $\ell_i$ of $N$, there is a smooth projective variety $Z_i$ and a projective alteration $p_i:Z_i \to W$ whose degree is relatively prime to $\ell_i$. Then 
\[
p_{i*}: CH_0(Z, 1, \ZZ/\ell_i^{m_i}\ZZ)  \to CH_0(W, 1, \ZZ/\ell_i^{m_i}\ZZ) 
\]
is surjective by Lemma \ref{lem:}.
Pulling back the families of cycles over $W$ gives a family of cycles $\Gamma_i=p_i^*\Gamma_W$ over $Z_i$. 
The composition
\[
CH_0(Z_i, 1, \ZZ/\ell_i^{m_i})  \xrightarrow{p_{i*}} CH_0(W, 1, \ZZ/\ell_i^{m_i}) \xrightarrow{\Gamma_*} \text{CH}_1(X, 1, \ZZ/\ell_i^{m_i})
\]
is just $\Gamma_{i*}$.
Therefore $[L_i]$ is in the image of $\Gamma_{i*}$.

We take $Z=\prod Z_i$ with the family of cycles $\Gamma_Z$ over $Z$ to be the sum of the pull-back of $\Gamma_i$'s.
By Corollary \ref{cor:ch_BM}, we have an isomorphism
\[
CH_0(Z, 1, \ZZ/N)\cong H_1(Z, \ZZ/N)\cong \oplus_i H_1(Z_i, \ZZ/N) \cong \oplus_i CH_0(Z_i, 1, \ZZ/N).
\]
Then $\Gamma_{Z*}(CH_0(Z, 1, \ZZ/\ell_i^{m_i}))\subset \text{CH}_1(X, 1, \ZZ/\ell_i^{m_i})$ contains the image of $\Gamma_{i*}$, in particular $[L_i]$.
Since $$\text{CH}_1(X, 1, \ZZ/N)\cong \oplus_i \text{CH}_1(X, 1, \ZZ/\ell_i^{m_i}), $$ we are done.
\end{proof}

\subsection{Comparing the filtrations}
In this section, we work over an algebraically closed field and omit the Tate twists (most of the time) for simplicity of notations.
Our goal is to prove the following theorem.

\begin{thm}\label{thm:image_sheaf}
Let $X$ be a $d$-dimensional smooth projective variety defined over an algebraically closed field, which is separably rationally connected in codimension one. Fix a prime number $\ell$ different from the characteristic of $k$. 
Then the following two subgroups of $H^{2d-3}_{\et}(X, \ZZ_\ell)$ agree:
\[
\tilde{N}^{d-2}H_{\et}^{2d-3}(X, \ZZ_\ell)=N^{d-2}H^{2d-3}_\et(X, \ZZ_\ell).
\]
\end{thm}
We first observe the following.
\begin{lem}\label{lem:surface_CH_1_1}
    For any reduced finite type two dimensional scheme $\Sigma$ defined over an algebraically closed field $k$, the higher cycle class map induces an isomorphism
\[
cl_{1, 1}: \text{CH}_1(\Sigma, 1, \ZZ/{\ell^n}) \cong H_3^{\text{BM}}(\Sigma, \ZZ/\ell^n).
\]
\end{lem}
\begin{proof}
    When this scheme $\Sigma$ is smooth, we argue in the following way. We have an isomorphism (by Corollary \ref{cor:chow_zar})
\[
\text{CH}_1(\Sigma, 1, \ZZ/\ell^n) \xrightarrow{\cong} H^1(\Sigma, \tau^{\leq 1}R\pi_* \ZZ/\ell^n),
\]
where $\pi$ is the morphism between \'etale and Zariski site of $\Sigma$, $\pi: \Sigma_{\et}\to \Sigma_{\text{Zar}}$, and $\tau$ is the good truncation functor.
Thus we have the desired isomorphism
(keeping in mind that $H_3^{\text{BM}}\cong H^1_\et$ for smooth surfaces):
\[
\text{CH}_1(\Sigma, 1, \ZZ/\ell^n) \xrightarrow{\cong} H^1(\Sigma, \tau^{\leq 1}R\pi_* \ZZ/\ell^n) \xrightarrow{\cong} H^1(\Sigma, R\pi_* \ZZ/\ell^n)\xrightarrow{\cong} H^1_\et(\Sigma, \ZZ/\ell^n).
\]

For the general case, we use the localization sequence for higher Chow groups and Borel-Moore homology and the following commutative diagram (where the vertical maps are higher cycle class maps and the commutativity follows from Proposition \ref{prop:func}):
\[
\xymatrix{
\text{CH}_1(Z, 1, \ZZ/\ell^n) \ar[r] \ar[d] &\text{CH}_1(\Sigma, 1, \ZZ/\ell^n) \ar[r] \ar[d] &\text{CH}_1(U, 1, \ZZ/\ell^n) \ar[r] \ar[d]&\text{CH}_1(Z, \ZZ/\ell^n) \ar[d]\\   
    H_3^{\text{BM}}(Z, \ZZ/\ell^n) \ar[r]  &H_3^{\text{BM}}(\Sigma, \ZZ/\ell^n) \ar[r] &H_3^{\text{BM}}(U, \ZZ/\ell^n) \ar[r] &H_2^{\text{BM}}(Z, \ZZ/\ell^n).}
\]
Here $U$ is the smooth locus of $\Sigma$ and $Z=\Sigma\backslash U$ the singular locus with the reduced scheme structure. Note that $Z$ is at most one dimensional. So the cycle class map $\text{CH}_1(Z, \ZZ/\ell^n) \to H_2^{\text{BM}}(Z, \ZZ/\ell^n)$ is an isomorphism and 
\[
\text{CH}_1(Z, 1, \ZZ/\ell^n)\xrightarrow{\cong} H_3^{\text{BM}}(Z, \ZZ/\ell^n)=0.
\]
Then the isomorphism follows from the five lemma.
\end{proof}

\begin{proof}[Proof of Theorem {\ref{thm:image_sheaf}}]
In the following, we use Borel-Moore homology (with coefficient $A$). 
For simplicity of notations, we only write them as $H_i(\cdot, A)$.
Denote by $\text{NH}_3(X, \ZZ/\ell^n)$ the group $N^{d-2}H_3(X, \ZZ/\ell^n)$.

Given an equi-dimensional family of one-cycles $\Gamma_S$ over a smooth projective variety $S$, we have a commutative diagram
\begin{equation}\label{eq:cd1}
\xymatrix{
 CH_0(S,1, \ZZ/\ell^n) \ar[r]^{ \Gamma_{S*}}\ar[d]^{ \Gamma_{S*}\circ cl^S_{1, 1}} &\text{CH}_1(X, 1, \ZZ/\ell^n) \ar@{->>}[r]\ar[d]^{cl_{1, 1}^X} &A_1(X)[\ell^n] \ar[d]^{cl_{1}}\\
   H_3(X, \ZZ/\ell^n)  \ar[r]^= & H_3(X, \ZZ/\ell^n) \ar[r] & H_2(X, \ZZ_\ell)[\ell^n]}.
\end{equation}

The left vertical map is the composition of cycle class map $cl^S_{1, 1}$ for $S$ and action of the correspondence $\Gamma_S$ on cohomology.
The middle and the right vertical maps are (higher) cycle class maps. 

The left square is commutative by Proposition \ref{class_cor}. The right square is commutative by Proposition \ref{prop:bockstein} and the fact that cycle class map is trivial on cycles that are algebraically equivalent to $0$.

The first observation is that the image of left and middle vertical maps actually lands in a smaller group.

\begin{claim}\label{claim:mid}
    The image of the higher cycle class map
\[
cl_{1, 1}^X: \text{CH}_1(X, 1, \ZZ/\ell^n) \to H_3(X, \ZZ/\ell^n)
\]
is $\text{NH}_3(X, \ZZ/\ell^n)\subset H_3(X, \ZZ/\ell^n)$.
\end{claim}
\begin{proof}
We first note that \[
cl_{1, 1}^X: \text{CH}_1(X, 1, \ZZ/\ell^n) \to H_3(X, \ZZ/\ell^n)
\]
factors through $\text{NH}_3(X, \ZZ/\ell^n)\subset H_3(X, \ZZ/\ell^n)$.
This is because any element in $\text{CH}_1(X, 1, \ZZ/\ell^n)$ is represented by a family of  one-cycles over $\Delta^1\cong \AAA^1$, and thus comes from $\text{CH}_1(\Sigma, 1, \ZZ/\ell^n)$ for some surface $\Sigma$ (which can be taken as the closure in $X$ of the locus swept out by the support of this family of one-cycles over $\Delta^1$). Therefore the higher cycle class map factors through $\text{NH}_3(X, \ZZ/\ell^n)$ by the functoriality of higher cycle class maps under proper push-forward (Proposition \ref{prop:func}).

Since the higher cycle class map is compatible with proper push-forward (Proposition \ref{prop:func}), we have a commutative diagram 
\[
\begin{CD}
    \oplus_\Sigma \text{CH}_1(\Sigma, 1, \ZZ/\ell^n) @>cl_{1,1}^\Sigma>\cong> \oplus_\Sigma H_3(\Sigma, \ZZ/\ell^n)\\
    @VVV @VVV\\
    \text{CH}_1(X, 1, \ZZ/\ell^n) @>cl_{1, 1}^X>> \text{NH}_3(X, \ZZ/\ell^n),
\end{CD}
\]
where the direct sum is over all $2$-dimensional finite type connected reduced schemes $\Sigma$ with a proper morphism to $X$.
In the commutative diagram,
the upper horizontal arrow is an isomorphism induced by the higher cycle class map (Lemma \ref{lem:surface_CH_1_1}),
and the right vertical arrow is surjective by the definition of $\text{NH}_3(X, \ZZ/\ell^n)$. Therefore the map
\[
\text{CH}_1(X, 1, \ZZ/\ell^n) \to \text{NH}_3(X,\ZZ/\ell^n)
\]
is surjective.
\end{proof}
\begin{claim}\label{claim:left1}
The map 
\[
\Gamma_{S*}\circ cl^S_{1, 1}: CH_0(S, 1, \ZZ/\ell^n) \to H_3(X, \ZZ/\ell^n)
\]
factors through $H_3(X, \ZZ_\ell)/\ell^n \subset H_3(X, \ZZ/\ell^n)$.
\end{claim}
\begin{proof}
 For any variety $Y$,
we have a short exact sequence
\[
0 \to H_1(Y, \ZZ_\ell)/\ell^n \to H_1(Y, \ZZ/\ell^n) \to H_0(Y, \ZZ_\ell)[\ell^n] \to 0.
\]
Therefore $H_1(Y, \ZZ_\ell)/\ell^n \xrightarrow{\cong} H_1(Y, \ZZ/\ell^n)$,
since $H_0(Y, \ZZ_\ell)\xrightarrow{\cong} \ZZ_\ell$ is torsion free.

The cycle class map induces an isomorphism (Corollary \ref{cor:ch_BM})$$cl_{0, 1}:CH_0(S,1, \ZZ/\ell^n) \xrightarrow{\cong} H_1(S, \ZZ/\ell^n)\xrightarrow{\cong} H_1(S, \ZZ_\ell)/\ell^n.$$

This fact, combined with the following commutative diagram:
\[
\begin{CD}
    H_1(S, \ZZ_\ell)@>\mod \ell^n >> H_1(S, \ZZ_\ell)/\ell^n @>\cong >> H_1(S, \ZZ/\ell^n)\\
    @VV\Gamma_{S*}V @VV\Gamma_{S*}V @VV\Gamma_{S*} V\\
    H_3(X, \ZZ_\ell)@>\mod \ell^n>> H_3(X, \ZZ_\ell)/\ell^n @>>> H_3(X, \ZZ/\ell^n),
\end{CD}
\]
shows that the left vertical arrow in the commutative diagram (\ref{eq:cd1}) is the composition
\[
    CH_0(S, 1, \ZZ/\ell^n) \xrightarrow{\cong} H_1(S, \ZZ/\ell^n)\cong H_1(S, \ZZ_\ell)/\ell^n \xrightarrow{\Gamma_{S*}} H_3(X, \ZZ_\ell)/\ell^n\to H_3(X, \ZZ/\ell^n).
\]
This proves the claim.
\end{proof}

Furthermore, the image of
\[
\Gamma_{S*}\circ cl^S_{1, 1}: CH_0(S, 1, \ZZ/\ell^n) \to H_3(X, \ZZ/\ell^n)
\]
is also contained in $\text{NH}_3(X, \ZZ/\ell^n)$ since the middle vertical map has image contained in it (Claim \ref{claim:mid}).
So we have an induced commutative diagram:
\begin{equation}\label{eq:cd}
\xymatrix{
\oplus_{(S, \Gamma_S)} CH_0(S,1, \ZZ/\ell^n)\ar[r]^{\oplus \Gamma_{S*}}\ar[d]^{\oplus \Gamma_{S*}\circ cl_{1, 1}^S} &\text{CH}_1(X, 1, \ZZ/\ell^n) \ar@{->>}[r]\ar[d]^{cl_{1, 1}^X} &A_1(X)[\ell^n] \ar[d]^{cl_{1}} \to 0\\
  0 \to H_3(X, \ZZ_\ell)/\ell^n \cap \text{NH}_3(X, \ZZ/\ell^n) \ar@{^{(}->}[r] & \text{NH}_3(X, \ZZ/\ell^n) \ar[r] & H_2(X, \ZZ_\ell)[\ell^n]},
\end{equation}
where the direct sum is taken over isomorphism classes of equi-dimensional families of one-cycles over smooth projective varieties. 

By Theorem \ref{thm:sheaf_surj}, the upper row is exact. The lower row is also exact, since it comes from
\[
0 \to H_3(X, \ZZ_\ell)/\ell^n \to H_3(X, \ZZ/\ell^n) \to H_2(X, \ZZ_\ell)[\ell^n] \to 0.
\]

We can be even more precise about the image of the left vertical map.
\begin{claim}\label{claim:left}
    The image of the left vertical map in the commutative diagram (\ref{eq:cd}) equals the image of $\tilde{N}H_3(X, \ZZ_\ell)/\ell^n \to H_3(X, \ZZ_\ell)/\ell^n \cap \text{NH}_3(X, \ZZ/\ell^n)$.
\end{claim}
\begin{proof}[Proof of Claim \ref{claim:left}]
   By the proof of Claim \ref{claim:left1}, the left vertical arrow in the commutative diagram (\ref{eq:cd}) is the direct sum of the  compositions
\[
CH_0(S, 1, \ZZ/\ell^n) \xrightarrow{\cong} H_1(S, \ZZ/\ell^n)\cong H_1(S, \ZZ_\ell)/\ell^n \xrightarrow{\Gamma_{S*}} H_3(X, \ZZ_\ell)/\ell^n.\]

Therefore, by Lemma \ref{cyl_equiv}, the left vertical map in the commutative diagram (\ref{eq:cd}) has the same image as $$\tilde{N}_{1, \cyl}H_3(X, \ZZ_\ell)/\ell^n \to H_3(X, \ZZ_\ell)/\ell^n \cap \text{NH}_3(X, \ZZ/\ell^n).$$
 By Lemma \ref{lem:inclusion}, $\tilde{N}_{1, \cyl}H_3(X, \ZZ_\ell)=\tilde{N}^{d-2}H_3(X, \ZZ_\ell)$. This proves the claim.
\end{proof}

Note that $\tilde{N}H_3(X, \ZZ_\ell)\subset \text{NH}_3(X, \ZZ_\ell)\subset H_3(X, \ZZ_\ell)$ are finitely generated $\ZZ_\ell$-modules. Thus we have isomorphisms 
\[
\tilde{N}H_3(X, \ZZ_\ell) \xrightarrow{\cong}\lim \limits_{\xleftarrow[n]{}} \tilde{N}H_3(X, \ZZ_\ell)/\ell^n,\quad
 \text{NH}_3(X, \ZZ_\ell) \xrightarrow{\cong}
\lim \limits_{\xleftarrow[n]{}} \text{NH}_3(X, \ZZ_\ell)/\ell^n.
\]

We have a factorization
\[
\tilde{N}H_3(X, \ZZ_\ell)/\ell^n \to {N}H_3(X, \ZZ_\ell)/\ell^n \to H_3(X, \ZZ_\ell)/\ell^n \cap \text{NH}_3(X, \ZZ/\ell^n) \subset H_3(X, \ZZ_\ell)/\ell^n.
\]
Taking inverse limit, we get maps
\begin{equation*}
    \tilde{N}H_3(X, \ZZ_\ell) \to \text{NH}_3(X, \ZZ_\ell) \to \lim \limits_{\xleftarrow[n]{}} (H_3(X, \ZZ_\ell)/\ell^n \cap \text{NH}_3(X, \ZZ/\ell^n)) \subset H_3(X, \ZZ_\ell).
\end{equation*}
Here the last inclusion is preserved due to left exactness of inverse limit.

Since the compositions
\[
\tilde{N}H_3(X, \ZZ_\ell) \to {N}H_3(X, \ZZ_\ell) \to \lim \limits_{\xleftarrow[n]{}} H_3(X, \ZZ_\ell)/\ell^n \cap \text{NH}_3(X, \ZZ/\ell^n) \to H_3(X, \ZZ_\ell)
\]
are the inclusions of subgroups, we have the following claim.
\begin{claim}\label{claim:inj}
    We have inclusions
    \[
\tilde{N}\text{H}_3(X, \ZZ_\ell) \subset \text{NH}_3(X, \ZZ_\ell) \subset \lim \limits_{\xleftarrow[n]{}} H_3(X, \ZZ_\ell)/\ell^n \cap \text{NH}_3(X, \ZZ/\ell^n) \subset H_3(X, \ZZ_\ell)
\]
\end{claim}

Now our next claim is:
\begin{claim}\label{claim:torfree}
    The cokernel of the inclusion
\begin{equation}\label{eq:torsionfree}
     \tilde{N}H_3(X, \ZZ_\ell) \to \lim \limits_{\xleftarrow[n]{}} H_3(X, \ZZ_\ell)/\ell^n \cap \text{NH}_3(X, \ZZ/\ell^n)
\end{equation}
is torsion free. 
So is the cokernel of the inclusion \[
\tilde{N}^{d-2}H_3(X, \ZZ_\ell)\to N^{d-2}H_3(X, \ZZ_\ell).\]
\end{claim}
\begin{proof}[Proof of Claim \ref{claim:torfree}]

By the snake lemma and Claims \ref{claim:mid}, \ref{claim:left}, we have an exact sequence:
\begin{equation}\label{eq:exact}
    \tilde{N}H_3(X, \ZZ_\ell)/\ell^n \to H_3(X, \ZZ_\ell)/\ell^n \cap \text{NH}_3(X, \ZZ/\ell^n) \to C_n \to 0,
\end{equation}
where 
$C_n$ is the cokernel of
\begin{equation}\label{eq:alpha}
    \alpha_n: \text{Ker}(\text{CH}_1(X, 1, \ZZ/\ell^n) \to \text{NH}_3(X, \ZZ/\ell^n)) \to \text{Ker}(A_1[\ell^n] \to H_2(X, \ZZ_\ell)[\ell^n]).
\end{equation}

Since $C_n$ is a quotient of $\text{Ker}(A_1[\ell^n] \to H_2(X, \ZZ_\ell)[\ell^n])$, the kernel of connecting maps $C_{n+m} \to C_n$ contains the $\ell^m$-torsions $C_{n+m}[\ell^m]$. 
Therefore the inverse limit $\lim \limits_{\xleftarrow[n]{}} C_n$ is torsion free. 

Every term involved in the exact sequence (\ref{eq:exact}) is a finite abelian group, and thus the exact sequence (\ref{eq:exact}) satisfies the Mittag-Leffler condition and the inverse limit of (\ref{eq:exact}) is exact. 
So the cokernel of 
\[
\lim \limits_{\xleftarrow[n]{}}\tilde{N}H_3(X, \ZZ_\ell)/\ell^n \to \lim \limits_{\xleftarrow[n]{}} H_3(X, \ZZ_\ell)/\ell^n \cap \text{NH}_3(X, \ZZ/\ell^n)
\]
is $\lim \limits_{\xleftarrow[n]{}} C_n$ and torsion free.

In view of the inclusions in Claim \ref{claim:inj}, the cokernel of $\tilde{N}H_3 \to \text{NH}_3$ is a sub-$\ZZ_\ell$-module of the cokernel of $\tilde{N}H_3(X, \ZZ_\ell) \to \lim \limits_{\xleftarrow[n]{}} H_3(X, \ZZ_\ell)/\ell^n \cap \text{NH}_3(X, \ZZ/\ell^n)$, and thus also torsion free. This proves the claim. 
\end{proof}

On the other hand, we know that the quotient $N^{d-2}H_3(X, \ZZ_\ell)/\tilde{N}^{d-2}H_3(X, \ZZ_\ell)$ is torsion since the strong coniveau filtration and coniveau filtration agree for $\QQ_\ell$-coefficient. 
So it has to be zero. 
That is, the two subgroups coincide:
\[
\tilde{N}^{d-2}H_3(X, \ZZ_\ell)=N^{d-2}H_3(X, \ZZ_\ell).\]
\end{proof}

For later use, we note the following. 
\begin{claim}\label{claim:equal}
    Assume that $X$ is SRC in codimension one. The inclusions
\begin{equation}\label{eq:equality}
\tilde{N}^{d-2}H_3(X, \ZZ_\ell) \to N^{d-2}H_3(X, \ZZ_\ell) \xrightarrow{\psi} \lim \limits_{\xleftarrow[n]{}} H_3(X, \ZZ_\ell)/\ell^n \cap \text{NH}_3(X, \ZZ/\ell^n)  
\end{equation}
are isomorphisms.
\end{claim}
\begin{proof}
    The first isomorphism is already shown above.
We have already shown that the second map is injective (Claim \ref{claim:inj}).
The cokernel of $\psi$ is torsion since  it is a subgroup of the cokernel of $N^{d-2}H_3(X, \ZZ_\ell) \to H_3(X, \ZZ_\ell)$, which is torsion by Corollary \ref{cor:2d-3torsion}. By the first isomorphism and the fact that the composition has torsion free cokernel (Claim \ref{claim:torfree}), the cokernel of $\psi$ is also torsion free, and thus zero.
\end{proof}

\begin{rem}
    If $X$ is a threefold that is SRC in codimension one, the proof of Theorem \ref{thm:image_sheaf} can be simplified. In fact, in this case, we know that $A_1(X)[\ell^n] \to H_2(X, \ZZ_\ell)[\ell^n]$ is injective by Bloch-Srinivas \cite[Theorem 1, (ii)]{BlochSrinivas}. So 
    \begin{equation*}
\tilde{N}^{1}H_3(X, \ZZ_\ell)/\ell^n \to  H_3(X, \ZZ_\ell)/\ell^n \cap \text{NH}_3(X, \ZZ/\ell^n)  
\end{equation*}
 is surjective by the snake lemma and Claim \ref{claim:mid}. The surjectivity is preserved when taking inverse limit since the groups are finite and satisfy the Mittag-Leffler condition. Then
    \begin{equation*}
\tilde{N}^{1}H_3(X, \ZZ_\ell) \to N^{1}H_3(X, \ZZ_\ell) \to \lim \limits_{\xleftarrow[n]{}} H_3(X, \ZZ_\ell)/\ell^n \cap N^1H_3(X, \ZZ/\ell^n)  
\end{equation*}
are isomorphisms (by claim \ref{claim:inj}).
\end{rem}

\begin{lem}\label{3fold}
    Let $X$ be a smooth projective $3$-fold defined over an algebraically closed field that is separably rationally connected in codimension one. Fix a prime number $\ell$ different from the characteristic of $k$.  Then 
    \[
    N^1H^3_\et(X, \ZZ_\ell)=H^3_\et(X, \ZZ_\ell)
    \]
    and the higher cycle class map induces a surjective map
    \[
    \lim \limits_{\xleftarrow[n]{}}\text{CH}_1({X}, 1, \ZZ/\ell^n \ZZ) \to H^{3}_\text{\'et}({X}, \ZZ_\ell(2)).
    \]
\end{lem}
\begin{proof}
By Corollary \ref{cor:torsionfree}, the quotient $H^3_\et(X, \ZZ_\ell)/N^1H^3_\et(X, \ZZ_\ell)$ is torsion free. 
On the other hand, by Corollary \ref{cor:2d-3torsion}, the quotient $H^3_\et(X, \ZZ_\ell)/N^1H^3_\et(X, \ZZ_\ell)$ is torsion. So it must be zero.

    Claim \ref{claim:mid} shows that the higher cycle class map $$cl_{1, 1}^X: \text{CH}_1(X, 1, \ZZ/\ell^n) \to N^1H^3_\et(X, \ZZ/\ell^n)$$ is surjective.
 By Corollary \ref{cor:chow_zar}, we have 
\[
\text{CH}_1(X, 1, \ZZ/\ell^n) \cong H^3(X, \tau^{\leq 2}R\pi_*{\mu_{\ell^n}^{\otimes 2}}),
\]
where $\pi: X_\et \to X_\text{Zar}$ is the morphism from the \'etale site to the Zariski site and $\tau$ is the good truncation functor.
It follows that the higher cycle class map 
\[
\text{CH}_1(X, 1, \ZZ/\ell^n) \xrightarrow{\cong} H^3(X, \tau^{\leq 2}R\pi_*{\mu_{\ell^n}^{\otimes 2}}) \to  H^3(X, R\pi_*{\mu_{\ell^n}^{\otimes 2}})\xrightarrow{\cong} H^3_\et(X, \mu_{\ell^n}^{\otimes 2})
\]
is an injection,
since $H^2_\et(X, R^3\pi_*{\mu_{\ell^n}^{\otimes 2}}[-3])=0$.
Therefore the higher cycle class map induces an isomorphism $$cl_{1, 1}^X: \text{CH}_1(X, 1, \ZZ/\ell^n) \xrightarrow{\cong} N^1H^3_\et(X, \ZZ/\ell^n).$$

We have maps between finite abelian groups
\[
N^1H^3(X, \ZZ_\ell)/\ell^n \to N^1H^3_\et(X, \ZZ/\ell^n) \to H^3_\et(X, \ZZ/\ell^n).
\]
Taking inverse limits, we get
\[
N^1H^3(X, \ZZ_\ell) \to \lim \limits_{ \xleftarrow[n]{}} N^1H^3_\et(X, \ZZ/\ell^n) \cong \lim \limits_{\xleftarrow[n]{}} \text{CH}_1(X, 1, \ZZ/\ell^n)  \to H^3_\et(X, \ZZ_\ell).
\]

Since $N^1H^3_\et(X, \ZZ_\ell)=H^3_\et(X, \ZZ_\ell)$, we have the desired surjection
\[
 \lim \limits_{\xleftarrow[n]{}} \text{CH}_1(X, 1, \ZZ/\ell^n)  \to H^3_\et(X, \ZZ_\ell(2)).
\]
\end{proof}
 
Combined with Theorem \ref{thm:image_sheaf}, this implies the following.
 \begin{thm}\label{thm:surj_dim3_chow}
 Let $X$ be a smooth projective $3$-fold over an algebraically closed field. 
 Fix a prime number $\ell$ different from the characteristic of $k$. 
 Assume that $X$ is separably rationally connected in codimension one. Then the following subgroups on $H^{3}_\et(X, \ZZ_\ell)$ introduced in Definition \ref{def:filtration} equal the whole cohomology group:
\[
 \tilde{N}_{1, \cyl}H^{3}_\et(X, \ZZ_\ell) = \tilde{N}^{1}H^{3}_\et(X, \ZZ_\ell) =N^1H^{3}_\et(X, \ZZ_\ell)=H^{3}_\et(X, \ZZ_\ell).
\]
 \end{thm}

\subsection{Arithmetic part of the cycle class map}
Over finite fields, we have the following corollary of Theorem \ref{thm:image_sheaf}.

\begin{cor}\label{cor:IHC}
Let $X$ be a smooth projective  variety of dimension $d$ defined over a finite field $\mathbb{F}_q$, that is separably rationally connected in codimension one.
Fix a prime number $\ell$ different from the characteristic of $k$. 
Assume one of the followings
\begin{enumerate}
    \item $N^{d-2}H_\et^{2d-3}(\overline{X}, \ZZ_\ell(d-1))=H_\et^{2d-3}(\overline{X}, \ZZ_\ell(d-1))$.
    \item The higher cycle class map induces a surjective map
\[
cl: \lim \limits_{\xleftarrow[n]{}}\text{CH}_1(\overline{X}, 1, \ZZ/\ell^n) \to H^{2d-3}_{\text{\'et}}(\overline{X}, \ZZ_\ell(d-1)).
\]

\end{enumerate}
Then every class in $H^1(\FF_q, H^3(\overline{X}, \ZZ_\ell(d-1)))$ is the class of an algebraic cycle defined over $\FF_q$.
In particular, this holds if $X$ has dimension $3$.
\end{cor}

\begin{proof}[Proof of Corollary \ref{cor:IHC}]
We first show that  the surjectivity of the map 
\[
cl: \lim \limits_{\xleftarrow[n]{}}\text{CH}_1(\overline{X}, 1, \ZZ/\ell^n) \to H^{2d-3}_{\text{\'et}}(\overline{X}, \ZZ_\ell(d-1))
\]
implies that $N^{d-2}H_\et^{2d-3}(\overline{X}, \ZZ_\ell(d-1))=H_\et^{2d-3}(\overline{X}, \ZZ_\ell(d-1))$.
In fact, we have a factorization of the above map
\begin{align*}
    & \lim \limits_{\xleftarrow[n]{}}\text{CH}_1(\overline{X}, 1, \ZZ/\ell^n) \to \lim \limits_{\xleftarrow[n]{}} N^{d-2}H_\et^{2d-3}(\overline{X}, \ZZ/\ell^n(d-1))\\
    \to & \lim \limits_{\xleftarrow[n]{}}H^{2d-3}_\et(\overline{X}, \ZZ/\ell^n(d-1))= H^{2d-3}_{\text{\'et}}(\overline{X}, \ZZ_\ell(d-1)).
\end{align*}

Therefore the map
$\lim \limits_{\xleftarrow[n]{}} N^{d-2}H_\et^{2d-3}(\overline{X}, \ZZ/\ell^n(d-1)) \to \lim \limits_{\xleftarrow[n]{}}H^{2d-3}_\et(\overline{X}, \ZZ/\ell^n(d-1))$
is surjective. On the other hand, since $N^{d-2}H_\et^{2d-3}(X, \ZZ/\ell^n)$ is a subgroup of $H_\et^{2d-3}(X, \ZZ/\ell^n)$, the inverse limit is injective, hence an isomorphism.
We have an exact sequence
\begin{align*}
 & 0 \to \lim \limits_{\xleftarrow[n]{}} H_\et^{2d-3}(\overline{X}, \ZZ_\ell(d-1))/\ell^n \cap N^{d-2}H_\et^{2d-3}(\overline{X}, \ZZ/\ell^n(d-1))\\
\xrightarrow{\phi} &\lim \limits_{\xleftarrow[n]{}} N^{d-2}H_\et^{2d-3}(\overline{X}, \ZZ/\ell^n(d-1)) \to \lim \limits_{\xleftarrow[n]{}}H^{2d-2}_\et(\overline{X}, \ZZ_\ell(d-1))[\ell^n],
\end{align*}
where the first inverse limit is $N^{d-2}H^{2d-3}_\et(\overline{X}, \ZZ_\ell(d-1))$ by Claim \ref{claim:equal}, and the last inverse limit is torsion free.
Since the quotient
\[
H_\et^{2d-3}(\overline{X}, \ZZ_\ell(d-1))/N^{d-2}H_\et^{2d-3}(\overline{X}, \ZZ_\ell(d-1))
\]
is torsion by Corollary \ref{cor:2d-3torsion}, we know that $\phi$ is an isomorphism and thus 
\[
N^{d-2}H^{2d-3}_\et(\overline{X}, \ZZ_\ell) \to H^{2d-3}_\et(\overline{X}, \ZZ_\ell)
\]is an isomorphism.

By Theorem \ref{thm:image_sheaf}, the two subgroups $N^{d-2}H^{2d-3}_\et(\overline{X}, \ZZ_\ell), \Tilde{N}^{d-2}H^{2d-3}_\et(\overline{X}, \ZZ_\ell)$ agree. So they equal the whole cohomology group $H^{2d-3}_\et(\overline{X}, \ZZ_\ell)$ by the previous paragraph. 
Then the corollary follows from \cite[Proposition 5.6]{ScaviaSuzuki2023coniveau}, the statement of which is recalled in Proposition \ref{thm:SS3}.

The $3$-fold case follows from Theorem \ref{thm:surj_dim3_chow}.
\end{proof}


\section{The integral Tate conjecture, the local-global principle, and the unramified cohomology}\label{sec:ITateHasse}

For general discussions on the relation between the integral Tate conjecture, unrammified cohomology, and the local-global principle for zero-cycles over global function fields, one can consult \cite{CTKahnCycleCodim2,CT_Scavia_ITC}.
\subsection{The integral Tate conjecture and the local-global principle for zero-cycles}
Let $X$ be a smooth projective geometrically irreducible variety of dimension $d$ defined over a finite field $\FF$. We have the cycle class maps:
\begin{equation}\label{int_Tate_2}
\text{CH}^r(X) \otimes \ZZ_\ell \to H^{2r}(X, \ZZ_\ell(r)),
\end{equation}
\[
\text{CH}^r(X) \otimes \ZZ_\ell \to H^{2r}_{\text{{\'e}t}}(X, \ZZ_\ell(r)) \to H^{2r}_{\text{{\'e}t}}(\overline{X}, \ZZ_\ell(r))^G.
\]

We also have the corresponding cycle class maps after tensoring with $\QQ_\ell$.
The Tate conjecture predicts that the cycle class map on codimension $r$ cycles
\[
\text{CH}^r(X) \otimes \mathbb{Q}_\ell \to H^{2r}_{\text{\'et}}(X, \mathbb{Q}_\ell(r))
\]
is surjective for any smooth projective variety $X$ defined over a finite field.
While the cycle class map (\ref{int_Tate_2}) is in general not surjective for $\ZZ_\ell$ coefficients (\cite[Th\'eor\`eme 2.1]{CTSzamuelyIntegralTate}), one is still interested in knowing in which cases surjectivity still hold.
This is usually called the integral Tate conjecture (even though it is not true in general).

Of particular interest to this paper is the following.
\begin{conj}\label{q:SRCTate}
For a smooth projective variety $X$ defined over $\FF$,  the cycle class map for one-cycles (i.e. (\ref{int_Tate_2}) for $r=d-1$) is surjective.
\end{conj}

See Section \ref{sec:questions} for the explanation why we expect this integral Tate conjecture for one-cycles to be true for varieties that are SRC in codimension one.

The connection between the integral Tate conjecture for one-cycles and Conjectures \ref{conj:CT1}, \ref{conj:CT2} is the following.
\begin{thm}[{\cite[Proposition 3.2]{CTLocalGlobalChow}}, {\cite[Corollary (8-6)]{SaitoMotivicCoh_ArithmeticScheme}}]\label{thm:TateImpiesCT}
Let $\FF$ be a finite field, $C$ a smooth projective geometrically connected curve over $\FF$, and $K$ the function field of $C$. 
Let $\mcX$ be a smooth projective geometrically connected variety of dimension $d+1$ defined over $\FF$, equipped with a morphism $p: \mathcal{X} \to C$, whose generic fiber $X$ is smooth and geometrically irreducible. Let $l$ be a prime different from the characteristic.
\begin{enumerate}
\item \label{itc1} If the cycle class map
\[
\text{CH}^d(\mcX) \otimes \ZZ_\ell \to H^{2d}_{\text{{\'e}t}}(\mcX, \ZZ_\ell(d))
\]
is surjective, Conjectures \ref{conj:CT1} and \ref{conj:CT2} are true for the generic fiber $X$.
\item \label{itc2} If the cycle class map
\[
\text{CH}^d(\mcX) \otimes \ZZ_\ell \to H^{2d}_{\text{{\'e}t}}(\mcX, \ZZ_\ell(d)) \to H^{2d}_{\text{{\'e}t}}(\overline{\mcX}, \ZZ_\ell(d))^G
\]
is surjective, or if 
\[
\text{CH}^d(\mcX) \otimes \ZZ_\ell \to H^{2d}_{\text{{\'e}t}}(\mcX, \ZZ_\ell(d))
\]
is surjective modulo torsion,
Conjecture \ref{conj:CT2} is true for the generic fiber $X$. 
\end{enumerate}
\end{thm}

\begin{rem}
The cited references only contain a proof of the first statement. But the second statement follows from the same proof. The general result of Saito produces a cohomology class $\xi \in H^{2d}(\mcX, \ZZ_\ell(d))$ whose restriction to each local place coincide with the class of $z_\mu$(\cite[Proposition 3.1]{CTLocalGlobalChow}). The various types of integral Tate conjecture are simply used to find a global cycle whose class agrees with $\xi$ in various cohomology groups. See also Page 19 of the slide of Colliot-Th\'el\`ene's lecture at Cambridge in 2008 (available at \url{https://www.imo.universite-paris-saclay.fr/~jean-louis.colliot-thelene/expocambridge240809.pdf}).
\end{rem}

We mention another closely related question.
\begin{ques}\label{q:CH0}
Let $X$ be a smooth projective variety defined over a henselian local field with finite residue field. Is the cycle class map
\[
CH_0(X)\hat{\otimes} \ZZ_\ell \to H^{2d}(X, \ZZ_\ell(d))
\]
injective? Here $\ell$ is a prime number invertible in the residue field.
\end{ques}
One should expect a negative answer in general. But the author is not aware of an explicit example in the literature.
\begin{rem}\label{rem:WE}
Question \ref{q:CH0} has a positive answer if $X$ is a geometrically rational surface, and has a regular model with SNC central fiber (\cite[Theorem 3.1]{WittenbergEsnault_0_cycle} in general and \cite[Theorem A]{Saito_torsion_codim_2} for the case of $p$-adic fields without requiring having SNC central fiber).
In this case, the proof in \cite{WittenbergEsnault_0_cycle} also shows that the closed fiber also satisfies a version of the integral Tate conjecture.
For $X$ defined over a Laurent field $\FF_q \Semr{t}$, a regular model with SNC central fiber always exists since we have resolution of singularities for $3$-folds.
\end{rem}
If Question \ref{q:CH0} has a positive answer for the generic fiber $X$, then Conjectures \ref{conj:CT1} and \ref{conj:E} are equivalent for $X$.

\subsection{Degree 3 unramified cohomology}
By Lefschetz hyperplane theorem, one reduces the integral Tate conjecture for one-cycles for all varieties to the same conjecture for threefolds. In this case, the conjecture is about codimension $2$ cycles, and thus closely related to the degree $3$ unramified cohomology by the work of \cite{CTKahnCycleCodim2}. 

 We summarize the relevant points. Denote by $\pi: X_\et \to X_\text{Zar}$ the morphism from the \'etale site to the Zariski site of $X$. For a prime number $\ell$ relatively prime to the characteristic of the base field, and an integer $n$, define $\mathcal{H}^i(\mu_{\ell^n}^{\otimes j})=R^i\pi_* \mu_{\ell^n}^{\otimes j}$.

 We define the degree $i$ unramified cohomology as $H^i_{\text{nr}}(X, \mu_{\ell^n}^{\otimes j})=H^0(X, \mathcal{H}^i(\mu_{\ell^n}^{\otimes j}))$ and $H^i_{\text{nr}}(X, \QQ_{\ell}/\ZZ_\ell(j))$ as the inductive limit $\lim \limits_{\xrightarrow[n]{}} H^0(X, \mathcal{H}^i(\mu_{\ell^n}^{\otimes j}))$.

 \begin{thm}[{\cite[Th\'eor\`eme 2.2, Proposition 3.2]{CTKahnCycleCodim2}}]\label{thm:ctk}
 Let $X$ be a smooth projective geometrically irreducible variety defined over a field $k$.
 \begin{enumerate}
     \item Assume that $k$ has finite $\ell$-cohomological dimension. Let $M$ be the cokernel of the cycle class map $\text{CH}^2(X) \otimes \ZZ_\ell \to H^4(X, \ZZ_\ell(2))$. Then the torsion part of $M$ is isomorphic to the quotient of $H^3_{\text{nr}}(X, \QQ_\ell/\ZZ_\ell(2))$ by its maximal divisible subgroup.
     \item Assume that $k$ is finite or separably closed. If the Chow group of zero-cycles with rational coefficients is universally supported in a surface, then the group $H^3_{\text{nr}}(X, \QQ_\ell/\ZZ_\ell(2))$ is finite.
 \end{enumerate}
     
 \end{thm}
  
 One should also note that the Tate conjecture for codimension $2$ cycles would imply that the cokernel $M$ as above has to be torsion.
 
 In \cite{CTKahnCycleCodim2}, the authors deduced a short exact sequence relating various Chow groups of codimension $2$ cycles and degree $3$ unramified cohomology. 
 Their short exact sequence for varieties over finite fields reads the following (\cite[Th\'eor\`eme 6.8]{CTKahnCycleCodim2}):
 \begin{align}\label{eq:CTK}
    \nonumber 0 \to &\text{Ker}(\text{CH}^2(X) \to \text{CH}^2
(\overline{X})) \to H^1(\FF, \oplus_{\ell} H^3_\text{\'et}(X, \ZZ_\ell(2))_\text{tors})\\
 \to &\text{Ker}(
H^3_\text{nr}(X, \QQ/\ZZ(2)) \to H^3_\text{nr}(\overline{X}, \QQ/\ZZ(2)))\\
\nonumber
\to &\text{Coker}(\text{CH}^2
(X) \to \text{CH}^2(\overline{X})^G)\to 0
 \end{align}
There is a similar exact sequence for $\ell$-primary torsions (\cite{CTKahnCycleCodim2} proved this result by first proving the exact sequence for each prime $\ell$ separately).

\subsection{Proof of main theorems}\label{proof}
\begin{proof}[Proof of Theorem \ref{thm:integralTate}]
Recall that $G=\text{Gal}(\bar{\FF}_q/\FF_q)$ is the absolute Galois group.
By Theorem \ref{thm:G_inv_cycle} (\cite[Theorem 7]{Kollar_Tian}), we have an isomorphism
\[
A_1(\mcX) \cong A_1(\overline{\mcX})^G.
\]
Under the assumptions (A), (B) of Theorem \ref{thm:integralTate}, we know that there is an isomorphism of $\text{Gal}(\bar{\FF}_q/\FF_q)$-modules:
\[
A_1(\overline{\mcX}) \otimes \ZZ_\ell \cong H^{2d}(\overline{\mcX}, \ZZ_\ell(d)).
\]

Note that $G$ is generated by the Frobenius $F$ and $A_1(\overline{\mcX})^G$ is the kernel of $F^*-\text{id}$. Since $\ZZ_\ell$ is a flat $\ZZ$-module, we have $A_1(\overline{\mcX})^G\otimes \ZZ_\ell$ is the kernel of $$(F^*-\text{id})\otimes \text{id}_{\ZZ_\ell}: A_1(\overline{\mcX})\otimes \ZZ_\ell \to A_1(\overline{\mcX})\otimes \ZZ_\ell.$$ That is,
\[
A_1(\overline{\mcX})^G \otimes \ZZ_\ell \cong (A_1(\overline{\mcX})\otimes \ZZ_\ell)^G \cong H^{2d}(\overline{\mcX}, \ZZ_\ell(d))^G,
\]
where $\ZZ_\ell$ is equipped with the trivial action of $G$.

This proves part (\ref{1.10.1}) of the theorem. Part (\ref{1.10.2}) of the theorem is just Corollary \ref{cor:IHC}.
\end{proof}
For an algebraic subvariety $Z$ of dimension $d$ in a variety $X$, we have its fundamental class $\eta_Z \in H_{2d}^{\text{BM}}(Z, \ZZ_\ell(d))\cong \ZZ_\ell$. 
See the last line of \cite[Example 2.1]{BlochOgus} for the case of $\ZZ/\ell^m$-coefficient. The $\ZZ_\ell$-coefficient case follows by taking inverse limit.
The $\ell$-adic cycle class map is defined by $Z \mapsto i_*(\eta_
Z)$, where $i_*: H_{2d}^{\text{BM}}(Z, \ZZ_\ell(d))\to H_{2d}^{\text{BM}}(X, \ZZ_\ell(d))$ is the push-forward of Borel-Moore homology via $i:Z \to X$.

We will use the following.
\begin{lem}[{\cite[First paragraph of Section 4.3]{WittenbergEsnault_0_cycle}}]\label{lem:surj_fiber}
    Let $\mathfrak{X} \to \SP R$ be a regular scheme of relative dimension $2$ over a henselian DVR with separably closed residue field, such that the central fiber is a simple normal crossing divisor. 
    Denote by $A$ the reduced closed subscheme of the central fiber. Assume that the geometric generic fiber is rational.
The $\ell$-adic cycle class map
\[
\text{CH}_1(A) {\otimes}\ZZ_\ell \to H_{2}^{\text{BM}}(A, \ZZ_\ell(1)).
\]
is surjective.
\end{lem}
\begin{proof}
    This follows from computations in \cite{WittenbergEsnault_0_cycle}, discussed in the first paragraph of Section 4.3 of the paper. 
    
More precisely, in the course of the proving Theorems 3.1 and 4.1 of \cite{WittenbergEsnault_0_cycle}, Esnault and Wittenberg have shown that, under the hypothesis of the lemma, certain complex ((4.6) op. cit.) is injective, which, by an argument similar to Section 2.2 op. cit., turns out to be equivalent to the surjectivity of the cycle class map ((4.7) op. cit.):
\[
\text{CH}_1(A) \hat{\otimes}\ZZ_\ell \to H^{4}_A(\mathfrak{X}, \ZZ_\ell(2)).
\]
 Denote by $\{A_i, i\in I\}$ the set of reduced irreducible components of $A$. Each $A_i$ is a smooth projective surface.
Thus $A_1(A_i)$ is finitely generated (Theorem of the base \cite[XIII Th\'eor\`eme 5.1]{SGA6}).
Therefore the natural map $$\text{CH}_1(A_i)\otimes \ZZ_\ell \to A_1(A_i) \otimes \ZZ_\ell \cong A_1(A_i)\hat{\otimes}\ZZ_\ell\cong \text{CH}_1(A_i)\hat{\otimes} \ZZ_\ell$$ is surjective, where the last equality holds because $\text{CH}_1(A_i)_\text{alg}$ is divisible.
We also have a surjection $\oplus \text{CH}_1(A_i)/\ell^m \to \text{CH}_1(A)/\ell^m$. Since $\oplus \text{CH}_1(A_i)/\ell^m$ is a finite group, the Mittag-Leffler condition is satisfied and we have a surjection $\oplus \text{CH}_1(A_i)\hat{\otimes} \ZZ_\ell \to \text{CH}_1(A) \hat{\otimes} \ZZ_\ell$ after taking the inverse limit.
Therefore the $\ell$-adic cycle class map
\[
\oplus \text{CH}_1(A_i)\otimes \ZZ_\ell \to \text{CH}_1(A) \hat{\otimes}\ZZ_\ell \to H^{4}_A(\mathfrak{X}, \ZZ_\ell(2))
\]
is surjective. 
Since the above map factors through $\text{CH}_1(A) {\otimes}\ZZ_\ell$, the $\ell$-adic cycle class map
\[
\text{CH}_1(A) {\otimes}\ZZ_\ell \to H^{4}_A(\mathfrak{X}, \ZZ_\ell(2)).
\]
is surjective.

We have a canonical isomorphism
\[
H_2^{\text{BM}}(A, \ZZ/\ell^m(1)) \cong H^{4}_A(\mathfrak{X}, \ZZ/\ell^m(2)).
\]
To prove this isomorphism, consider the following commutative diagram:
\[
\begin{CD}
    A @>\iota>> \mathfrak{X}\\
    @VVgV @VVfV\\
    \SP \kappa @>h>> \SP R
\end{CD}
\]
For $\Lambda=\ZZ/\ell^m$, 
 we have $g^!\Lambda=\iota^!\Lambda(d+1)[2d+2]$ (where every functor is derived) in the derived category of \'etale sheaves of $\Lambda$-modules \cite[(2.6)]{WittenbergEsnault_0_cycle}, where $d=2$ is the relative dimension of $\mathfrak{X} \to \SP R$. Since $\iota^!$ is the same as the derived functor of $\Gamma_{A}$, the global section with support functor (\cite[XVIII 3.1.8]{SGA4_3}), we have this isomorphism.
 Taking inverse limit, we have an isomorphism
 \[
H_2^{\text{BM}}(A, \ZZ_\ell(1)) \cong H^{4}_{A}(\mathfrak{X}, \ZZ_\ell(2)).
\]
Moreover, this isomorphism is compatible with the cycle class map from $\text{CH}_1(A)\otimes \ZZ_\ell$ by the commutativity of the following diagram, where $Z\subset A$ is a closed subvariety of dimension one.
\[
\begin{CD}
    H_2^{\text{BM}}(Z, \ZZ_\ell(1)) @>>> H^{4}_{Z}(\mathfrak{X}, \ZZ_\ell(2))\\
    @VVV @VVV\\
    H_2^{\text{BM}}(A, \ZZ_\ell(1)) @>>> H^{4}_{A}(\mathfrak{X}, \ZZ_\ell(2)).
\end{CD}
\]
\end{proof}
\begin{lem}\label{lem:H2d}
    Let $\mcX^0 \to B^0$ be a smooth projective family of relative dimension $d$ of separably rationally connected varieties over a smooth \emph{affine} curve $B^0$ defined over an algebraically closed field.
    Then the \'etale cohomology group $H^{2d}(\mcX^0, \ZZ_\ell(d))$ is generated by the class of a section.
\end{lem}
\begin{proof}
    We compute $H^{2d}({\mcX}^0, \ZZ_\ell(d))$ via the Leray spectral sequence.
Since the base $B^0$ is an affine curve, the $E_2$-page has only two columns: $H^0(B^0, \mathcal{H}^q), H^1(B^0, \mathcal{H}^{q})$, where $\mathcal{H}^q$ is the $\ell$-adic local system coming from degree $q$ $\ell$-adic cohomology of the fibers ($q=0, \ldots, 2d$), and thus the spectral sequence degenerates at $E_2$.
Moreover, since the smooth fibers have no $H^{2d-1}$ since smooth projective SRC varieties are simply connected, we have an isomorphism
\[
H^{2d}({\mcX}^0, \ZZ_\ell(d))\xrightarrow{\cong}H^0(B^0,\mathcal{H}^{2d})\cong \ZZ_\ell.
\]
We have a section of ${\mcX}^0 \to {B}^0$ by the main theorem in \cite{deJongStarr_GHS}.
The cycle class of a section generates the group $H^0(B^0,\mathcal{H}^{2d})\cong \ZZ_\ell$.
\end{proof}
\begin{proof}[Proof of Theorem \ref{thm:integralTateSurface}]
We first show that hypotheses (A)-(D) in Theorem \ref{thm:integralTate} are satisfied for the smooth projective model $\mcX$.

For hypothesis (A), note that the surjectivity of the cycle class map for one-cycles is a birational invariant (Lemma \ref{lem:birational}). 
So using resolution of singularities for $3$-folds \cite{Resolution1, Resolution2, AbhyankarResolution}, we may assume that the singular fibers are SNC divisors.

We take $\cup \mcX_i$ to be a non-empty finite union of fibers that contains all the singular fibers of the fibration $\mcX \to B$ and $\mcX^0$ its complement.
We have a commutative diagram of localization exact sequences, where the horizontal maps are localization sequences for Chow groups \cite[Proposition 1.8]{Fulton98} (resp. \'etale Borel-Moore homology) applied to the triple $(\overline{\mcX}, \cup \overline{\mcX}_{i, \text{red}}, \overline{\mcX}^0)$ (where $\overline{\mcX}_{i, \text{red}}$ means the reduced closed subscheme of $\overline{\mcX}_{i}$ ), and the vertical maps are cycle class maps for Chow groups:
\begin{equation}\label{eq:loc}
    \begin{CD}
 \oplus \text{CH}_1(\overline{\mcX}_{i, \text{red}})\otimes \ZZ_\ell @>>> \text{CH}_1(\overline{\mcX})\otimes \ZZ_\ell @>>>\text{CH}_1 (\overline{\mcX}^0)\otimes \ZZ_\ell @>>> 0 \\
@VVV @VVV @VVV\\
\oplus H_{2}^{\text{BM}}(\overline{\mcX}_{i, \text{red}}, \ZZ_\ell(1)) @>>>H_2^{\text{BM}}(\overline{\mcX}, \ZZ_\ell(1))@>>>H_2^{\text{BM}}(\overline{\mcX}^0, \ZZ_\ell(1))\\
\end{CD}
\end{equation}
  The commutativity of the left square follows from functoriality of Borel-Moore homology under proper push-forward. The commutativity of the right square follows from the fact that fundamental class commutes with restriction to an open subset (See \cite[1.3.4, and Example 2.1]{BlochOgus} for the case of $\ZZ/\ell^m$-coefficient. The $\ZZ_\ell$-coefficient case follows by taking inverse limit).

We apply Lemma \ref{lem:surj_fiber} in our situation. 
For each fiber $\overline{\mcX}_i$ over a point $b_i \in B(\bar{\FF}_q)$, we base change $\overline{\mcX} \to \bar{B}$ to ${\mathfrak{X}}_i \to \SP \OO_{\bar{B}, b_i}^{\text{h}}$, where $\OO_{B, b_i}^{\text{h}}$ is the henselization of the local ring $\OO_{\bar{B}, b_i}$ at $b_i$. Still denote by $\overline{\mcX}_i$ the closed fiber of $\mathfrak{X}_i$.
The $\ell$-adic cycle class map
\[
\text{CH}_1(\overline{\mcX}_{i, \text{red}})\otimes \ZZ_\ell \to H_2^{\text{BM}}(\overline{\mcX}_{i, \text{red}}, \ZZ_\ell(1))
\]
is surjective by Lemma \ref{lem:surj_fiber}.

Since we assume that the union $\cup \mcX_i$ is non-empty, the open complement $\overline{\mcX}^0$ is over an affine curve $B^0\subset B$.
By Lemma \ref{lem:H2d},
the cycle class of a section generates the group $ H_2^{\text{BM}}(\overline{\mcX}^0, \ZZ_\ell(1))\cong H^0(B^0,\mathcal{H}^4)\cong \ZZ_\ell$, and
thus the right vertical cycle class map in the commutative diagram (\ref{eq:loc}) is surjective.
So the middle map in the commutative diagram (\ref{eq:loc}) is also surjective.
This proves that hypothesis (A) in Theorem \ref{thm:integralTate} is satisfied.

The result of Bloch-Srinivas \cite[Theorem 1, (iii)]{BlochSrinivas} shows that the Griffiths group of one-cycles on $\overline{\mcX}$ is $p$-torsion (since for $\overline{\mcX}$ the rational Chow group of zero-cycles is universally supported on a curve by Corollary \ref{rem:support}). So hypothesis (B) in Theorem \ref{thm:integralTate} is satisfied.

Hypotheses (C) and (D) in Theorem \ref{thm:integralTate} are also satisfied in this case by Lemma \ref{3fold}.

The Hochschild-Serre spectral sequence gives a short exact sequence:
\[
0 \to H^1(\FF_q, H^{2d-2r-1}_\et(\overline{X}, \ZZ_\ell(d-r))) \to H^{2d-2r}_\et(X, \ZZ_\ell(d-r)) \to H^{2d-2r}_\et(\overline{X}, \ZZ_\ell(d-r))^G \to 0,
\]
where $G$ is the absolute Galois group of $\overline{\FF}_q/\FF_q$.
Thus Theorem \ref{thm:integralTate} implies that $\text{CH}_1(\mcX) \otimes \ZZ_\ell \to H^4(\mcX, \ZZ_l(2))$ is surjective.
\end{proof}

\begin{proof}[Proof of Theorem \ref{thm:zerocycle}]
This follows from Theorem \ref{thm:TateImpiesCT}, \ref{thm:integralTateSurface}, and Remark \ref{rem:WE}.
\end{proof}

\begin{proof}[Proof of Theorem \ref{thm:delpezzo4}]
If there is a rational point everywhere locally and the local points satisfy the Brauer-Manin constraint, there is a zero-cycle of degree $1$ over the function field by Theorem \ref{thm:zerocycle}.
A del Pezzo surface of degree $4$ is a complete intersection of $2$ quadrics in $\PP^4$. 
Such a complete intersection has a rational point if and only if it has an odd degree $0$-cycle \cite{Brumer_CI22}.
Hence we have the result.
\end{proof}

\begin{proof}[Proof of Corollary \ref{cor:CTK}]
The vanishing of $H^3_{\text{nr}}(X, \QQ_\ell/\ZZ_\ell(2))$ follows from Theorem \ref{thm:integralTateSurface} and \cite[Th\'eor\`eme 2.2, Proposition 3.2]{CTKahnCycleCodim2} (recalled in Theorem \ref{thm:ctk}), and the fact that $CH_0(\overline{X})\otimes \QQ$ is universally supported in a curve (Corollary \ref{rem:support}).

We have already shown that the cycle class map is surjective in Theorem \ref{thm:integralTateSurface}. But it is also injective on torsions \cite[Th\'eor\`eme 4]{CT_Sansuc_Soule_torsion}. In particular, it is injective on the algebraically trivial part, i.e. $\text{CH}_1(X)_{\text{alg}} \otimes \ZZ_\ell \to H^1(\FF, H^3(\overline{X}, \ZZ_\ell(2))$ is injective. We also have an isomorphism $A_1(X) \otimes \ZZ_\ell \cong A_1(\overline{X})^G\otimes \ZZ_\ell \cong H^4(\overline{X}, \ZZ_\ell(2))^G$. This proves that the cycle class map is an isomorphism.

It follows from the exact sequence (\ref{eq:CTK}) that we have the description of the Chow groups of $X$ and $\overline{X}$.
\end{proof}

\section{Examples}\label{sec:example}
We give some examples where one can check the conditions in Theorem \ref{thm:integralTate}.
We first make a simple observation about the coniveau filtration.
\begin{lem}\label{lem:sm_fil}
    Let $\mcX^0 \to B^0$ be a smooth projective family of relative dimension $d$ over an affine curve $B^0$ defined over an algebraically closed field.
    Fix a prime number $\ell$ different from the characteristic.
    Assume that a general fiber $\mcX_b$ is separably rationally connected, and that the cycle class map induces a surjective map $$A_1(\mcX_b)\otimes \ZZ_\ell \to H^{\text{BM}}_2(\mcX_b, \ZZ_\ell(d-1)).$$
    Then $N^{d-1}H^{2d-1}(\mcX^0, \ZZ_\ell(d))=H^{2d-1}(\mcX^0, \ZZ_\ell(d))$.
\end{lem}
\begin{proof}
    We use the Leray spectral sequence for $\mcX^0 \to B^0$.
    Since $B^0$ is affine, the $E_2$-page has only two columns: $H^0(B^0, \mathcal{H}^q)$ and $H^1(B^0, \mathcal{H}^q)$, where $\mathcal{H}^{q}$ is the $\ell$-adic local system corresponding to the degree $q$ $\ell$-adic cohomology of the fiber, and thus the spectral sequence degenerates.
    Since a general fiber is separably rationally connected, it is simply connected. 
    In particular, a general fiber has no $H^{2d-1}$.
    So we have the isomorphism 
\[
H^{2d-1}(\mcX^0, \ZZ_\ell(d))\cong H^1(B^0, \mathcal{H}^{2d-2}).
\]

By assumption, for a general fiber $\mcX_b$, $H^{2d-2}(\mcX_b, \ZZ_\ell(d-1))$ is generated by the classes of finitely many algebraic curves.
Since $\mcX_b$ is separably rationally connected, the usual argument involving adding very free curves and smoothing shows that $H^{2d-2}(\mcX_b, \ZZ_\ell(d-1))$ is generated by classes of finitely many curves that deforms to general fibers (see for example \cite[Theorem 46]{Kollar_Tian}). 
For each of these curves, choose a deformation $T \leftarrow \Sigma_i \to \mcX$ which surjects onto $B^0$.
Then 
\[
H^{2d-1}(\mcX^0, \ZZ_\ell(d))\cong H^1(B^0, \mathcal{H}^{2d-2})
\]
is supported on the union of $\cup_i \Sigma_i$.
\end{proof}

\begin{cor}\label{cor:sm_fil}
    Let $\mcX \to B$ be a flat projective family over a smooth projective curve defined over an algebraically closed field.
    Assume that the generic fiber is smooth projective and separably rationally connected. 
    Fix a prime number $\ell$ different from the characteristic.
    If the following two conditions hold:
    \begin{enumerate}
        \item for a general fiber $\mcX_b$, the cycle class map induces a surjective map $$A_1(\mcX_b)\otimes \ZZ_\ell \to H^{\text{BM}}_2(\mcX_b, \ZZ_\ell(d-1));$$
        \item for every singular fiber $\mcX_b$, $H_3^{\text{BM}}(\mcX_b, \ZZ_\ell(1))$ is supported on a two dimensional Zariski closed subset of $\mcX_b$,
    \end{enumerate}
    we have
    \[
    N^{d-2}H^{2d-1}(\mcX, \ZZ_\ell(d))=H^{2d-1}(\mcX, \ZZ_\ell(d)).
    \]
\end{cor}
\begin{proof}
 Apply the localization exact sequence to the finite union of singular fibers and its complement.
\end{proof}
To determine $A_1(\overline{\mcX})$ seems to be a difficult problem in general.
We provide a simple example where one can use some global geometry to do so.
\begin{lem}\label{lem:multisection}
    Let $\mcX \subset \PP_B(E)$ be a family of complete intersections of degree $d_1, \ldots, d_c$ in $\PP^n$ over a smooth projective curve $B$ over an algebraically closed field.
    Assume that $\sum d_i^2\leq n$.
    \begin{enumerate}
        \item The cycle class of any two sections $\sigma_1, \sigma_2 \subset \mcX$ of $\mcX \to B$ are rationally equivalent modulo lines in general fibers.
        \item More generally, if $C\subset \mcX$ is any irreducible curve such that $C \to \mcX \to B$ has degree $d\geq 1$, then the cycle class of $C$ is rationally equivalent to $d[\sigma_0]$ for a section $\sigma_0$ modulo lines in general fibers.
        \item Any curve in a fiber is rationally equivalent to a sum of lines in general fibers.
        \item The group of one-cycles modulo algebraic equivalence, $A_1(\mcX)$, is a free abelian group of rank $2$, generated by the class of a section and a line in a general fiber.
    \end{enumerate}

\end{lem}
Note that Tsen-Lang theorem implies that there are sections of $\mcX \to B$ under the assumptions on degrees.
\begin{proof}
Let $H$ be the scheme parameterizing the space of chains of two lines passing through $\sigma_1$ and $\sigma_2$, that is, curves of the form $(L_1\cup L_2, x_0, x_1, x_2)$, where $L_1$ and $L_2$ are two lines in a fiber $\mcX_b, b \in B$ intersecting at a single point $x_0$, $x_1 \in L_1, x_2 \in L_2$ and $x_1=\sigma_1(b), x_2=\sigma_2(b)$. 
The morphism $H \to B$ is proper, because a degeneration of a chain of two lines is again a chain of two lines.

    The space of chains of two lines passing through two points in a complete intersection of degree $d_1, \ldots, d_c$ in $\PP^n$ is defined by equations of degree $1,1,2,2, \ldots, d_1-1, d_1-1, d_1, 1, 1, \ldots, d_2-1, d_2-1, d_2, \ldots, 1, 1, \ldots, d_c-1, d_c-1, d_c$ in $\PP^n$ (see, for example, Lemma 3.4 in \cite{Pan_conic}).
    
The classical Tsen-Lang theorem applied to the generic fiber of $H \to B$ implies that there is a rational section $B \dashrightarrow H$. Since $B$ is a curve, we can extend this to a section.
This is equivalent to the following.
The family $\mcX \to B$ contains two ruled surfaces over $B$: $B \leftarrow \Sigma_1 \subset \mathcal{X}, B \leftarrow \Sigma_2 \subset \mathcal{X}$, such that $\sigma_1 \subset \Sigma_1$ (resp. $\sigma_2 \subset \Sigma_2$) is a section.
Moreover, $\Sigma_1\cap \Sigma_2=\sigma_0$ is a common section of both ruled surfaces. And $\Sigma_1\cup_{\sigma_0} \Sigma_2 \subset \mcX$ is a family of chains of two lines connecting $\sigma_1, \sigma_2$.

Any  two sections in a $\PP^1$-bundle over a curve are rationally equivalent modulo general fibers.
Thus both $\sigma_1, \sigma_2$  are rationally equivalent to $\sigma_0$ up to lines in general fibers.
Pushing forward this rational equivalence to $\mcX$ proves the first statement.

Given $C\subset \mcX$ such that $C \to B$ is dominant of degree $d$, we make a base change $\mcX_C=\mcX \times_B C \xrightarrow{(p, q)} \mcX \times C$. 
The family $\mcX_C \to C$ admits two sections: $\sigma_C$, induced by the inclusion $C \subset \mcX$, and $\sigma_B$, induced by $C \to B \to \sigma_0 \subset \mcX$.
By previous argument, $[\sigma_C]$ and $[\sigma_B]$ are rationally equivalent modulo lines in general fibers.
We push forward this rational equivalence relation to $\mcX$ via $p$. Since $p_*[\sigma_C]=[C], p_*[\sigma_B]=d[\sigma_0]$, and the push-forward of lines in general fibers are still lines in general fibers, we prove the second statement.

Any curve in a fiber is rationally equivalent to the difference of two multi-sections of the same degree.
So it is also rationally equivalent to a sum of cycle class of lines in general fibers.

Finally, since the Fano scheme of lines of a complete intersection of degree $(d_1, \ldots, d_c)$ in $\PP^n$ is connected as long as $2n-c-\sum d_i>2$ \cite[Th\'eor\`eme 2.1]{DebarreManivelFanoScheme}, all lines in the fibers are algebraically equivalent.
This proves the fourth statement.
\end{proof}

\begin{prop}\label{ci}
Let $\mcX \subset \PP_B(E)$ be a family of complete intersections of degree $d_1, \ldots, d_c, d_i \geq 2,$ in $\PP^n$ over a smooth projective curve $B$ over $\FF_q$. 
Assume that the generic fiber $X$ is smooth separably rationally connected of dimension $d=n-c$, that $\sum d_i^2\leq n$, and that the total space $\mcX$ is smooth.
If one of the following conditions holds
\begin{enumerate}
    
    \item $c=1$, i.e. $X$ is a smooth hypersurface;
    \item for every geometric fiber $X_b$, $H^\text{BM}_2(X_b, \ZZ_\ell)$ is generated by the class of a line and $H^\text{BM}_3(X_b, \ZZ_\ell)=0$.
\end{enumerate}
Then the cycle class map 
\[
\text{CH}^d(\mcX) \otimes \ZZ_\ell \to H^{2d}_{\text{{\'e}t}}(\mcX, \ZZ_\ell(d))
\]
is surjective and Conjectures \ref{conj:CT1} and \ref{conj:CT2} hold for the generic fiber $X$ over $\FF_q(B)$.
\end{prop}
\begin{rem}
    Under the assumptions, there is a rational point for the generic fiber $X$ since $\FF_q(B)$ is $C_2$ and normic forms exist. Thus the Hasse principle part (Conjecture \ref{conj:CT2}) is automatic.
\end{rem}

\begin{rem}
 In general it is still an open question if a smooth Fano complete intersection is separably rationally connected. However, one can show that if the characteristic $p$ is larger than all the $d_i$, then every smooth Fano complete intersection of degree $d_1, \ldots, d_c$ is separably rationally connected \cite{WASTZ2018}.
\end{rem}

\begin{rem}\label{rem:bm_condition}
The numerical conditions imply that $X$ is either a quadric of dimension $3$ or has dimension at least $4$.
Thus the condition \ref{ci}.(2) on Borel-Moore homology holds for all smooth fibers by Lefschetz hyperplane theorem.
It also holds for singular fibers with ordinary double points.
\end{rem}

\begin{proof}[Proof of Propostion \ref{ci}]By Theorem \ref{thm:integralTate} and  by Theorem \ref{thm:TateImpiesCT},
it suffices to show:
\begin{enumerate}
    \item $A_1(\overline{\mcX})\otimes \ZZ_\ell \cong H^{2d}(\overline{\mcX}, \ZZ_\ell(d)).$
    \item $N^{d-1}H^{2d-1}(\overline{\mcX}, \ZZ_\ell(d))  
=H^{2d-1}(\overline{\mcX}, \ZZ_\ell(d)).$
\end{enumerate}
There is an open subset $B^0 \subset \bar{B}$ such that $\mcX^0=\overline{\mcX}|_{B^0}$ is a complete intersection of ample divisors in $\PP(E)|_{B^0}$.
In case \ref{ci}.(1), we take $B^0$ in such a way that $B\backslash B^0$ is a single point with smooth fiber.
In case \ref{ci}.(2), we take $B^0$ so that ${\mcX}^0 \to B^0$ is smooth.

Write $\mcX_0$ to be the disjoint union of all the fibers over $B\backslash B^0$.
Denote by $\PP^0$ the open subscheme $\overline{\PP_B(E)}|_{B^0}$, and $U^0$ the open subscheme $\PP^0\backslash \mcX^0$.

We use the localization sequence of \'etale cohomology
\begin{align*}
   \ldots \to &H^{2d-1}_{\mcX_0}(\overline{\mcX}, \ZZ_\ell(d-1)) \to H^{2d-1}(\overline{\mcX}, \ZZ_\ell(d)) \to H^{2d-1}({\mcX}^0, \ZZ_\ell(d)) \\
    \to&H^{2d}_{\mcX_0}(\overline{\mcX}, \ZZ_\ell(d-1)) \to H^{2d}(\overline{\mcX}, \ZZ_\ell(d)) \to H^{2d}({\mcX}^0, \ZZ_\ell(d))\\ \to &H^{2d+1}(\mcX_0, \ZZ_\ell(d-1))\ldots,
\end{align*}
We have the isomorphisms \[
H^{2d-3}_{\mcX_0}(\overline{\mcX}, \ZZ_\ell(d-1))\cong H^\text{BM}_3(\mcX_0, \ZZ_\ell(1))=0,\]
\[
H^{2d-2}_{\mcX_0}(\overline{\mcX}, \ZZ_\ell(d-1))\cong \oplus_{b \in B\backslash B^0} H^{\text{BM}}_2(\mcX_b, \ZZ_\ell(1)) \cong \oplus_{b \in B\backslash B^0} \ZZ_\ell\cdot [L_b],\]
where $[L_b]$ is the class of a line in $\mcX_b$.
In the case where $X$ is a hypersurface,
the above isomorphisms follow from Remark \ref{rem:bm_condition} and standard facts about cohomology of smooth hypersurfaces.
In the case of complete intersections, the above isomorphisms follow from our assumption.

Moreover the map
\[
H^{2d-2}_{\mcX_0}(\overline{\mcX}, \ZZ_\ell(d-1)) \to H^{2d}(\overline{\mcX}, \ZZ_\ell(d))
\]
has one-dimensional image spanned by the class of a line in the total space by Lemma \ref{lem:multisection}, part (3) and (4).

In case \ref{ci}.(1), we use the following localization long exact sequence:
\begin{align*}
&H^{2d}(U^0, \ZZ_\ell(d+1))\to H^{2d-1}(\overline{\mcX}|_{B^0}, \ZZ_\ell(d)) \xrightarrow{\cong} H^{2d+1}(\PP^0, \ZZ_\ell(d+1)) \\
    &H^{2d+1}(U^0, \ZZ_\ell(d+1))\to H^{2d}(\overline{\mcX}|_{B^0}, \ZZ_\ell(d)) \xrightarrow{\cong} H^{2d+2}(\PP^0, \ZZ_\ell(d+1)) \\
    \to &H^{2d+2}(U^0, \ZZ_\ell(d+1))\ldots .
\end{align*}
The open complement $U^0$ is an affine variety of dimension $n+1$. Under our assumptions on $d$ and $n$, we have $2d=2n-2>n+1$.
So by Artin vanishing for \'etale cohomology, 
\[
H^{2d}(U^0, \ZZ_\ell(d+1))= H^{2d+1}(U^0, \ZZ_\ell(d+1))= H^{2d+2}(U^0, \ZZ_\ell(d+1))=0.
\]
Therefore we have isomorphisms
\[
H^{2d-1}(\overline{\mcX}|_{B^0}, \ZZ_\ell(d)) \xrightarrow{\cong} H^{2d+1}(\PP^0, \ZZ_\ell(d+1)), \]
\[ H^{2d}(\overline{\mcX}|_{B^0}, \ZZ_\ell(d)) \xrightarrow{\cong} H^{2d+2}(\PP^0, \ZZ_\ell(d+1)).
\]
We know that $H^{2d}(\PP^0, \ZZ_\ell(d+1))$ is spanned by the class of a section by Lemma \ref{lem:H2d}, and that $H^{2d-1}(\PP^0, \ZZ_\ell(d+1))$ is supported on a surface swept out by a family of lines by Lemma \ref{lem:sm_fil}.
Since we also have a section of $\overline{\mcX} \to B$ and we may choose the family of lines to lie in $\overline{\mcX}$, the same is true for $\mcX^0=\overline{\mcX}|_{B^0}$ in this case.

In case \ref{ci}.(2), the group $H^{2d}(\overline{\mcX}|_{B^0}, \ZZ_\ell(d))$ is spanned by the class of a section by Lemma \ref{lem:H2d}.
So $H^{2d}(\overline{\mcX}, \ZZ_\ell(d))$ is spanned by the class of a section and a line.
By Corollary \ref{cor:sm_fil}, our assumption on singular fibers also implies that $H^{2d-1}(\overline{\mcX}, \ZZ_\ell(d))$ is supported on a surface.

Over the algebraic closure $\overline{\FF}_q$,  by Lemma \ref{lem:multisection}, the group $A_1(\overline{\mcX})$ is a free abelian group of rank $2$, generated by the class of a section and a line in a general fiber.
As a result, \[
A_1(\overline{\mcX})\otimes \ZZ_\ell \cong H^{2d}_\et(\overline{\mcX}, \ZZ_\ell(d)).\]
This concludes the proof.
\end{proof}

\begin{prop}
Let $X$ be a smooth projective variety defined over $\FF_q(B)$.
Assume that 
\begin{enumerate}
    \item $X$ admits a regular projective model $\mcX \to B$ over $\FF_q$.
    \item The base change of $X$ to $\bar{\FF}_q(B)$ becomes stably rational (over $\bar{\FF}_q(B)$).
\end{enumerate}
Then the cycle class map
\[
\text{CH}^d(\mcX) \otimes \ZZ_\ell \to H^{2d}_{\text{{\'e}t}}(\mcX, \ZZ_\ell(d))
\]
is surjective (where $d$ is the dimension of $X$) and Conjecture \ref{conj:CT2} holds for $X$.
\end{prop}

\begin{proof}
The validity of the hypotheses in Theorem \ref{thm:integralTate} only depends on the stable birational class of the generic fiber over $\overline{\FF}_q(B)$ by Lemma \ref{lem:birational}. They are clearly true if the generic fiber over $\overline{\FF}_q(B)$ is $\PP^n$.
\end{proof}

\begin{rem}
    Certain moduli spaces of stable vector bundles on curves defined over $\FF_q(B)$ become rational/stably rational over $\overline{\FF}_q(B)$ (e.g. if the curve has a rational point over $\overline{\FF}_q(B)$ \cite[Corollary 6.2]{ Hoffmann_rationality}). 
    Unfortunately, it is not known if there exists a regular model for such moduli spaces over $B$ in the general case.
\end{rem}

\section{Questions and expectations}\label{sec:questions}
\subsection{Validity of hypotheses in Theorem \ref{thm:integralTate}} The main questions left open from the current work is the validity of the hypotheses (A)-(D) in Theorem \ref{thm:integralTate}.

A result of Schoen \cite{SchoenIntegralTate} says that if the Tate conjecture is true for divisors on all smooth projective surfaces defined over finite fields, then for any smooth projective variety $V$ defined over a finite field $\FF$, the cycle class map
 \[
 \text{CH}_1(\bar{V})\otimes \ZZ_l \to \cup_{K/\FF} H^{2d-2}(\bar{V}, \ZZ_\ell(d-1))^{\text{Gal}(\bar{\FF}/K)}(\subset H^{2d-2}(\bar{V}, \ZZ_\ell(d-1)))
 \]
 is surjective, where $\bar{V}$ is the base change of $V$ to an algebraic closure $\overline{\FF}$ of $\FF$, and the union is taken over all finite extensions $\FF\subset K \subset \overline{\FF}$.

 If furthermore $V$ is SRC in codimension one, since its rational Chow group of zero-cycles is universally supported in a curve (Corollary \ref{rem:support}), it is easy to see that every class in $H^{2d-2}(\bar{V}, \QQ_\ell(d-1))$ is algebraic. Thus every class in $H^{2d-2}(\bar{V}, \ZZ_\ell(d-1))$ is fixed by some open subgroup of the Galois group.
 So in this case, Schoen's theorem implies that  we always have a surjection
 \[
 \text{CH}_1(\bar{V}) \otimes \ZZ_\ell \to H^{2d-2}(\bar{V}, \ZZ_\ell(d-1)), 
 \]
 provided that the Tate conjecture holds for all surfaces.
 
The author has made conjectures on the Kato homology of rationally connected fibrations over an algebraically closed field of characteristic $0$ in \cite[Conjecture 1.16]{zerocycleLaurent}. A special case of the conjectures in loc. cit. is the following:
\begin{conj}\label{conj:tian}
    Let $X$ be a smooth projective variety defined over an algebraically closed field of characteristic $0$. Assume that $X$ is separably rationally connected in codimension one. Then the cycle class map on higher Chow groups induces an isomorphism
    \[
    \text{CH}_1(X, i, \ZZ/N) \cong H_{2+i}(X, \ZZ/N)
    \]
    for all $i$, where the homology group is the Borel-Moore homology.
\end{conj}
This conjecture implies that for a smooth projective variety that is separably rationally connected in codimension one defined over an algebraically closed field of characteristic $0$, the statements of hypotheses (A), (B), (C), and (D) hold. It is quite reasonable to believe that the same is true in characteristic $p>0$ with $N$ relatively prime to $p$.

Over the complex numbers, Conjecture \ref{conj:tian} can be formulated in terms of Lawson homology with $\ZZ$-coefficients. Using the decomposition of the diagonal argument (See the proof of Proposition 5.2, 5.3, 5.4, 5.5 in \cite{sstK3fold} for examples of the argument in dimension $3$ and $4$), one can show that when $X$ is SRC in codimension one, for each $i$, the kernel and cokernel of s-maps $L_1H_i(X) \to H_i(X)$ are $N$-torsion.
We have a commutative diagram of long exact sequences
\[
\begin{CD}
L_1H_i(X, \ZZ) @>\cdot N>>L_1H_i(X, \ZZ) @>>>L_1H_i(X, \ZZ/N)  @>>>L_1H_{i-1}(X, \ZZ)\\
@VVV @VVV @VVV @VVV\\
H_{i}(X, \ZZ) @>\cdot N>> H_{i}(X, \ZZ) @>>> H_{i}(X, \ZZ/N) @>>>H_{i-1}(X, \ZZ).\\
\end{CD}
\]

By results of Suslin-Voevodsky \cite[Theorem 9.1]{Suslin_Voevodsky_homology} and the Beilinson-Lichtenbaum conjecture proved by Voevodsky (Theorem \ref{bl}),
there is an isomorphism 
\[
L_1H_{2+i}(X, \ZZ/N) \xrightarrow{\cong} \text{CH}_1(X, i, \ZZ/N) \xrightarrow{\cong} \mathbb{H}^i(X, \tau^{\leq \dim X-1 }R\pi_*(\ZZ/N))
\] 
between torsion Lawson homology, Bloch's higher Chow group,  and certain Zariski cohomology group, where $\pi: X_{cl} \to X_{zar}$ is the continuous map from $X(\CC)$ with the analytic topology to $X$ with the Zariski topology (See Corollary \ref{cor:chow_zar} for the last isomorphism).

Then Conjecture \ref{conj:tian} is equivalent to the following.
\begin{conj}\label{conj:tian_lawson}
     Let $X$ be a smooth projective complex variety that is SRC in codimension one. The s-maps $$s:L_1H_i(X, \ZZ) \to H_i(X, \ZZ)$$ are isomorphisms for all $i$.
\end{conj}

\subsection{Injectivity of cycle class map}
Even though the author believes that Question \ref{q:CH0} has a negative answer in general, when we restrict ourselves to separably rationally connected varieties, there might be a chance of having a positive answer.
\begin{ques}\label{ques}
Let $X$ be a smooth projective separably rationally connected variety defined over a henselian local field with finite or separably closed residue field. Is the cycle class map
\[
CH_0(X)\hat{\otimes} \ZZ_\ell \to H^{2d}(X, \ZZ_\ell(d))
\]
injective? Here $\ell$ is a prime number invertible in the residue field.
\end{ques}

By the work of Saito-Sato \cite[Theorem 1.16]{SaitoSato_0_cycle}, if $X$ admits a regular model $\mathfrak{X} \to \SP R$ such that the cental fiber is a simple normal crossing (SNC) divisor, the cycle class map induces an isomorphism 
\[
\text{CH}^d(\mathfrak{X})/\ell^n \xrightarrow{\cong} H^{2d}_\et(\mathfrak{X}, \mu_{\ell^n}),
\]
where $d$ the the dimension of $X$.
We have the following commutative diagram of localization sequences for Chow groups and \'etale cohomology:
\[
\begin{CD}
    \text{CH}^{d-1}(\mathfrak{X}_0)/\ell^n @>>> \text{CH}^d(\mathfrak{X})/\ell^n @>>> \text{CH}^d(X)/\ell^n \to 0\\
    @VVV @VV\cong V @VVV\\
    H^{2d}_{\mathfrak{X}_0}(\mathfrak{X}, \mu_{\ell^n}^{\otimes d})@>>>H^{2d}(\mathfrak{X}, \mu_{\ell^n}^{\otimes d})@>>> H^{2d}(X, \mu_{\ell^n}^{\otimes d}),
\end{CD}
\]
where the left vertical map maybe identified with the cycle class map to the Borel-Moore homology
\[
\text{CH}_1(\mathfrak{X}_0)/\ell^n \to H_2^{\text{BM}}(\mathfrak{X}_0, \mu_{\ell^n}).
\]
By a simple diagram chasing, we would have a positive answer to Question \ref{ques} if this cycle class map were surjective (i.e. the central fiber $\mathfrak{X}_0$ satisfies the integral Tate conjecture for one-cycles).

The author studied this question for rationally connected varieties defined over the Laurent field over $\CC$ in \cite{zerocycleLaurent}. The main technical result in \cite{zerocycleLaurent} is a d\'evissage statement, a special case of which implies that the surjectivity of cycle class maps of one-cycles for SNC degenerations of rationally connected varieties over Laurent fields would follow from the corresponding surjectivity for smooth projective rationally connected varieties. The proof of this d\'evissage only depends on properties of the minimal model program. Since the minimal model program is expected to hold in positive and mixed characteristic as well, we may expect that such a d\'evissage type result is also true in positive and mixed characteristic (but in positive and mixed characteristic, one has to be careful about what should be the right class of varieties, since rational connectedness is not as well-behaved). Then combined with the expectations from previous subsection (and Theorem \ref{thm:integralTate}),  we may expect a positive answer to Question \ref{ques} when $X$ is separably rationally connected.

\bibliographystyle{alpha}
\bibliography{MyBib}

\end{document}